\documentclass[12pt]{amsart}

\theoremstyle{plain}
\newtheorem{theorem}{Theorem}[section]
\newtheorem{lemma}[theorem]{Lemma}
\newtheorem{proposition}[theorem]{Proposition}

\newtheorem{Definition}[theorem]{Definition}
\theoremstyle{remark}
\newtheorem{remark}{Remark}[section]
\newtheorem{example}{Example}[section]

\numberwithin{equation}{section}

\usepackage{amsmath}
\usepackage{amssymb} 
\usepackage{array}
\usepackage{enumitem}
\usepackage{hyperref}
\usepackage[all]{xy}
\usepackage{verbatim} 
\usepackage{color}
\definecolor{battleshipgrey}{rgb}{0.52, 0.52, 0.51} 
\hypersetup{
       colorlinks=true,       
    linkcolor=battleshipgrey,     
    citecolor=battleshipgrey,        
    filecolor=magenta,      
    urlcolor=battleshipgrey          
}
\usepackage{pstricks-add}

\usepackage{pgf,tikz,tikz-cd}

\usepackage{pstricks-add}

\usetikzlibrary{arrows,patterns}
\usetikzlibrary{matrix,positioning}
\tikzstyle{every node}=[font=\normalsize]
\usepackage[utf8x]{inputenc}
\usepackage[T1]{fontenc}

\newcommand{\bA}{\mathbb{A}}

\newcommand{\bO}{\mathbb{O}}
\newcommand{\K}{\mathbb{K}}

\newcommand{\R}{\mathbb{R}}
\newcommand{\Z}{\mathbb{Z}}
\newcommand{\bC}{\mathbb{C}}

\newcommand{\PP}{\mathbb{P}}
\newcommand{\F}{\mathbb{F}}

\newcommand{\bH}{\mathbb{H}}

\newcommand{\s}{\mathfrak{S}} 

\newcommand{\bfD}{\mathbf{D}}
\newcommand{\bfC}{\mathbf{C}}
\newcommand{\bfH}{\mathbf{H}}

\newcommand{\Gl}{\mathrm{Gl}}
\newcommand{\GL}{\mathrm{GL}}

\newcommand{\UU}{\mathrm{U}}
\newcommand{\SU}{\mathrm{SU}}
\newcommand{\OO}{\mathrm{O}}
\newcommand{\SO}{\mathrm{SO}}
\newcommand{\Sp}{\mathrm{Sp}}

\newcommand{\End}{\mathrm{End}}

\newcommand{\tr}{\mathrm{tr}}

\newcommand{\id}{\mathrm{id}}

\newcommand{\Cl}{\mathbf{Cl}} 

\newcommand{\eps}{\varepsilon}

\newcommand{\msk}{\medskip}
\newcommand{\ssk}{\smallskip}
\newcommand{\nin}{\noindent}

\newcommand{\wt}{\widetilde}

\begin{document}

\title[On Group and Loop Spheres]{On Group and Loop Spheres}                                     
\author{Wolfgang Bertram}                 



\address{%
Wolfgang Bertram\\              
Institut Elie Cartan de Lorraine,\\
Site de Nancy
\\
B.P. 70239
\\
F - 54506 Vand\oe uvre Cedex,
France \\            
wolfgang.bertram@univ-lorraine.fr  
\\
\url{http://wolfgang.bertram.perso.math.cnrs.fr}          
}

\begin{abstract}
We investigate the problem of {\em defining group or loop structures on
spheres}, where by ``sphere'' we mean the level set $q(x)=c$ of a general
$\K$-valued quadratic form $q$, for an invertible scalar $c$.
When $\K$ is a field and $q$ non-degenerate, then this corresponds to
the classical theory of {\em composition algebras}; 
in particular, for $\K=\R$ and positive definite forms, we obtain the
sequence of the four real division algebras
$\R,\bC,\bH$ (quaternions), $\bO$ (octonions). 
Our theory is more general, allowing that
 $\K$ is merely a ring, and the form $q$ possibly degenerate.
To achieve this goal, we give a more geometric formulation,
replacing the theory of binary composition algebras by
{\em ternary algebraic structures}, thus defining categories of
{\em group spherical} and of {\em Moufang spherical spaces}.
In particular, we develop a theory of {\em ternary Moufang loops},
and show how it is related to the Albert-Cayley-Dickson construction
and to generalized ternary octonion 
algebras.
At the bottom, a starting point of the whole theory is the (elementary)
result that {\em every $2$-dimensional quadratic space carries a canonical
structure of commutative group spherical space}.
\end{abstract}

\maketitle

\section{Introduction}

\subsection{Spheres and groups.}
Some spheres ``are'' groups, some ``are'' loops, but most are not.
Those that are bear a close relation to the
{\em four real division algebras}
$\R$, $\bC$, $\bH$ (quaternions), and $\bO$ (octonions).
Namely, the following are 

\ssk
-- {\em groups}: $S^0 = \OO(1)$, $S^1 = \SO(2) = \UU(1)$, $S^3 = \SU(2) = \Sp(1)$, or

-- {\em (Moufang) loops}: case of the unit sphere $S^7$ of the octonions.
\ssk

\nin  
More generally, let us call  {\em sphere} the level set $S = \{ x \in V \mid q(x) = c \}$ of a 
 quadratic form $q:V \to \K$, where $c\in \K^\times$. 
 The problems I am going to study are:
 \begin{itemize}
 \item
 What is  a ``natural'' group structure on a sphere ?
 
Develop a theory of such ``group spherical spaces''!
 \item
 What is a ``natural'' Moufang loop structure on a sphere?

 Develop a theory of such ``Moufang spherical spaces''!
\end{itemize}

\nin
 We do not assume that $2$ is invertible in $\K$; so the important cases $\K=\Z$ and $\K = \F_{2^m}$ are allowed; and we will not assume that $q$ is non-degenerate. 
 In the case where 
 the quadratic form $q$ is {\em non-degenerate}, 
very much, if not all, is known, since such spheres correspond to 
{\em composition algebras}, on which there exists a huge literature.
Without  being exhaustive, let met just mention
\cite{CSl, CS, Eb, McC, Fa, KMRT, SV}.
However, as far as I see, none of these authors considered the question for
completely general forms, and did not aim at
defining a suitable {\em category} of such spaces. 
To do this, we need a more geometric approach,
which I will develop in the present work: the
 {\em ternary algebraic viewpoint}.
Thus the present work contains both an overview over classical material, and
a new look on it, with a broader scope, new examples, and
perspectives for further research imbedding the classical theory into a ``categorical''
view on general geometric and algebraic structures.
Indeed, my approach is guided by trying to understand the {\em interplay of  
Jordan and Lie structures} (see \cite{Be00, Be14, BeKi1}), where
the four real division algebras play a key r\^ole, both
as ``number systems''  and for
geometry and algebra: they are pervasive  in the interaction of
Lie- and Jordan-theory,  in 
 applications and examples, and the octonions are related to almost all
``exceptional''' Lie- and Jordan structures. 
In the last section \ref{sec:outlook} we outline some of these further topics.

\subsection{Binary {\em versus} ternary algebra.}
A sphere does not have any ``canonical'' base point, and every choice
of ``unit element in a sphere'' would be ungeometric and artificial.
Therefore no sphere will carry a natural {\em group} structure -- rather, it could be
a ``group with unit forgotten''. 
There is a very simple way to ``forget the unit in a group'': replace the {\em binary}
product $ab$ by the {\em ternary} product
\begin{equation}
(a \vert b \vert c) := 
(abc):= ab^{-1} c .
\end{equation} 
This product has two characteristic properties defining a {\em torsor}
(see Appendix \ref{app:torsors})

\ssk
 (IP): {\em idempotency}, $(aab) = b = (baa)$,

 (PA): {\em para-associativity}, $(ab(cde))=((abc)de)=(a(dcb)e)$,

\ssk
\nin
and from these the group laws can be recovered by
$a\cdot c:= (abc)$, where the middle element $b$ becomes the new origin
(Theorem \ref{th:Torsors}).
Summing up, ``torsors are for groups what affine spaces are for vector spaces''.
Our first question can now be reformulated:
{\em 
What is a ``natural torsor structure'' on a sphere?}
To fix ideas, let us illustrate it by the example of the unit circle $S^1$, or more generally,
a {\em conic centered at $0$}, the
 level set $\{ x \in \K^2 \mid q(x) = c \}$ of a {\em binary quadratic form}
$q:\K^2 \to \K$.

\subsection{All circles are torsors.}
What is the ``most natural'', or  ``simplest'', construction of the torsor law on the unit circle $S^1$?
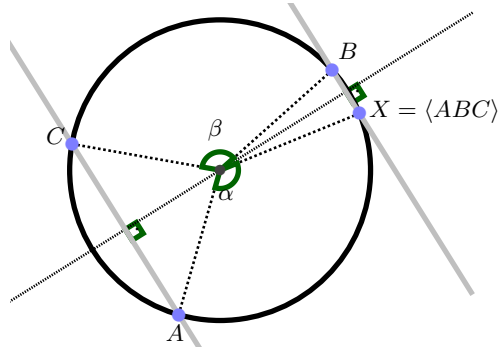
\begin{figure}[h]\label{fig:1}\caption{Construction of the group law on the circle}
\begin{center}
\newrgbcolor{xdxdff}{0.49019607843137253 0.49019607843137253 1.}
\newrgbcolor{qqwuqq}{0. 0.39215686274509803 0.}
\psset{xunit=0.5cm,yunit=0.5cm,algebraic=true,dimen=middle,dotstyle=o,dotsize=5pt 0,linewidth=2.pt,arrowsize=3pt 2,arrowinset=0.25}
\begin{pspicture*}(-5.566459233714445,-4.682418646844064)(7.7139601307903645,4.441010744768212)
\pspolygon[linewidth=2.pt,linecolor=qqwuqq,fillcolor=white,fillstyle=solid,opacity=0.1](-2.338216981545509,-1.8664975204371506)(-2.0457286088271425,-1.683868596821653)(-2.22835753244264,-1.3913802241032864)(-2.5208459051610066,-1.5740091477187843)
\pspolygon[linewidth=2.pt,linecolor=qqwuqq,fillcolor=white,fillstyle=solid,opacity=0.1](3.5259747959061274,1.7950874589496058)(3.818463168624494,1.9777163825651034)(3.635834245008996,2.2702047552834697)(3.3433458722906297,2.087575831667972)
\pscircle[linewidth=2.pt](0.,0.){2.}
\psplot[linewidth=1.pt,linestyle=dashed, dash=0.5pt 0.5pt]{-6.566459233714445}{6.7139601307903645}{(-0.--2.8359426751122814*x)/4.5418887750361385}
\psline[linewidth=1.pt,linestyle=dashed,dash=1pt 1pt](0.,0.)(-1.1028745676048661,-3.844953535236854)
\psline[linewidth=1.pt,linestyle=dashed,dash=1pt 1pt](0.,0.)(-3.9388172427171475,0.696935239799285)
\psline[linewidth=1.pt,linestyle=dashed,dash=1pt 1pt](0.,0.)(3.704143125978822,1.5097429258863388)
\psline[linewidth=1.pt,linestyle=dashed,dash=1pt 1pt](0.,0.)(2.9825486186024373,2.6654087374496043)
\pscustom[linewidth=2.pt,linecolor=qqwuqq,fillcolor=white,fillstyle=solid,opacity=0.1]{
\parametricplot{0.7293058135106493}{2.9664650305700957}{0.4876530978398339*cos(t)+0.|0.4876530978398339*sin(t)+0.}
\lineto(0.,0.)\closepath}
\pscustom[linewidth=2.pt,linecolor=qqwuqq,fillcolor=white,fillstyle=solid,opacity=0.1]{
\parametricplot{-1.8501335700068862}{0.38702564705255993}{0.4876530978398339*cos(t)+0.|0.4876530978398339*sin(t)+0.}
\lineto(0.,0.)\closepath}
\psplot[linewidth=2.pt,linecolor=lightgray]{-6.566459233714445}{6.7139601307903645}{(-15.913201433279403-4.5418887750361385*x)/2.8359426751122814}
\psplot[linewidth=2.pt,linecolor=lightgray]{-6.566459233714445}{6.7139601307903645}{(--5.37017377470074-1.1556658115632654*x)/0.7215945073763845}
\begin{scriptsize}
\psdots[dotsize=4pt 0,dotstyle=*,linecolor=darkgray](0.,0.)
\psdots[dotstyle=*,linecolor=xdxdff](2.9825486186024373,2.6654087374496043)
\rput[bl](3.1378374132982496,2.951510565353026){{$B$}}
\psdots[dotstyle=*,linecolor=xdxdff](-3.9388172427171475,0.696935239799285)
\rput[bl](-4.6207856032191755,0.7083063152897905){{$C$}}
\psdots[dotstyle=*,linecolor=xdxdff](-1.1028745676048661,-3.844953535236854)
\rput[bl](-1.446101706396189,-4.547030798983153){{$A$}}
\psdots[dotstyle=*,linecolor=xdxdff](3.704143125978822,1.5097429258863388)
\rput[bl](3.966847679625967,1.2447247229136076){{${X = \langle ABC \rangle}$}}
\rput[bl](-0.04816282592199858,-0.8034182880136942){{$\alpha $}}
\rput[bl](-0.32449958136457113,0.7733267283351016){{$\beta $}}
\rput[bl](-2.2913670759852343,-1.7137040706480506){{$\cdot$}}
\rput[bl](3.560470098092772,1.9436941631507028){{$\cdot $}}
\end{scriptsize}
\end{pspicture*}
\end{center}
\end{figure}
Let $A,B,C$ be three points on $S^1$, and consider $B$ as ``origin'' of $S^1$.
The product $A \cdot C$ with respect to this origin is constructed as follows:
assume first $A,B,C$ are pairwise distinct.
Let $ X := \langle ABC\rangle$ be the second intersection point of the parallel to the line
${AC}$ through $B$ with the circle.  By construction,
$ X$ is the image of $B$ under the orthogonal symmetry $S_{A,C}$
that exchanges $A$ and $C$ and fixes the center of the circle. 
But $X$ is also the image of $A$ under the rotation $R_{C,B}$ taking $B$ to $C$
(Figure \ref{fig:1}: note the equality of angles $\alpha = \beta$),
as well as the image of $C$ under the rotation $L_{A,B}$ taking $B$ to $A$.
But this means exactly that
  $A\cdot C = \langle ABC \rangle = L_{A,B}(C)$ is the product of $A$ with $C$ in the group
  $S^1$ with origin (unit element) $B$!
  When $A=C$, then use the tangent line at $A$, and the same construction works; and when
  $A=B$ or $C=B$, then simply
  $\langle BBC\rangle = C$ and $\langle ABB\rangle = A$.
Summing up, 
the ternary product $\langle ABC \rangle$ is in fact the torsor law $(ABC)$ on the
circle $S^1$:
\begin{equation}\label{eqn:ABC}
A \cdot_B C = \langle ABC \rangle = X = S_{A,C}(B) = L_{A,B}(C) = R_{C,B}(A).
\end{equation}
Next, observe that this ternary map
  extends naturally to a {\em $\R$-trilinear product} on $V=\R^2$.
Identifying $\R^2$ with the complex plane $\bC$, this trilinear extension is given by
\begin{equation}
\langle ABC \rangle = A \cdot \overline B \cdot C.
\end{equation}
All of this generalizes:
 {\em all circles (level sets of a binary quadratic form) carry a canonical commutative torsor structure}.
This observation might be considered to be folklore\footnote{E.g., see  \cite{Bo15}, p. 48--50, where
this construction is given under some unnecessarily restrictive assumptions,
or \url{https://math.stackexchange.com/questions/3951052/groups-of-conics}.};
however, 
we have not been able to find it explicitly stated in the literature, and
the fact that it holds in the completely general case (no assumption whatsoever on the binary
quadratic form) seems to be new; 
therefore we state and prove it in this generality (Theorem \ref{th:circles}).

\subsection{Group spherical spaces.}\label{sec:Gsp}
The preceding example motivates the following definition of a ``natural'' torsor structure on 
a sphere:
it should be given by a {\em $\K$-trilinear product map}
$V^3 \to V$, $(x,y,z) \mapsto \langle x\vert y \vert z\rangle$ (shorter: $\langle xyz \rangle$),
satisfying 
\begin{enumerate}
\item
the ``$q$-analog'' of idempotency: 
$\langle xxy \rangle = q(x)y = \langle yxx \rangle$ (following McCrimmon, \cite{McC},
we call this also  the {\em Kirmse identity}, (K)),
\item
the para-associative law (PA) in the form given above,
\item
the {\em ternary composition law} (TC)
$q(\langle xyz \rangle) = q(x) q(y) q(z)$.
\end{enumerate}
A {\em group spherical space} is a quadratic space $(V,q)$ together with 
a trilinear map having these properties, and such that non-empty spheres exist:
the set
\begin{equation}\label{eqn:Vtimes}
V^\times := \{ x \in V \mid q(x) \in \K^\times \} 
\end{equation}
shall be non empty. The space
is called {\em commutative} if moreover $\langle xyz \rangle = \langle zyx \rangle$.
For instance, the complex plane with
$T(x,y,z)=x \overline y z$ clearly satisfies these properties.
Our general result (Theorem \ref{th:circles}) states that
{\em every binary quadratic form $q:\K^2 \to \K$ admits exactly one ternary product satisfying
(K), and this product also satisfies (PA) and (TC)}. 
This settles the $2$-dimensional case.

The general theory of group spherical spaces is close to the one of torsors,
since their ternary products have very similar properties. 
In both cases, 
the ternary concept is base point-free and geometric, and gives, in three ways,
rise to binary compositions:
\begin{enumerate}
\item[(A)]
 We may take the ``diagonal'' $x=z$ and define a binary product
$\mu(x,y):=(xyx)$.
For a (Lie) group $G$, this means to consider $G$ as {\em symmetric space}
with product $\mu(x,y)=xy^{-1}x$, see \cite{Lo69}.
In the same way, {\em any} sphere carries a canonical structure of
``symmetric space over $\K$''. If the sphere is of group type, then this
structure is ``underlying'' to its group structure (Section \ref{sec:symmetricspaces}).
\item[(B)]
 For any ternary product,
{\em inner operators}
$L_{x,y}$ and $R_{z,y}$ and   $S_{x,z}$, are defined
via $\langle xyz \rangle = S_{x,z}(y) = L_{x,y}(z) = R_{z,y}(x)$,
see Equation (\ref{eqn:ABC}). 
In a torsor, the left (resp.\ right) translations form a group isomorphic to the original one,
and in a group spherical space, they
 form an algebra generalizing the algebra $\bC$ of similarities of
Euclidean plane (Subsection \ref{sec:spirations}).
In this case,  the operators
$L_{x,y}$ and $R_{x,y}$ are called  {\em left (right) spirations} 
and $S_{x,z}$ {\em spiflections}.
\item[(C)]
Fixing a base point also leads back to binary algebras:   in a group spherical space,
the ``homotope'' $x \cdot z := \langle xe z \rangle$ is, for every base point
$e \in V^\times$, a {\em (generalized) composition algebra}, and conversely,
the whole structure can be recovered from the binary product
(Theorem \ref{th:binary-ternary}).
\end{enumerate}
Via (C),
the case of a {\em non-degenerate} form $q$ corresponds to the case of
a (usual, i.e., non-degenerate) composition algebra, and in this case much is known.
In particular, non-degenerate group spherical spaces are

-- either commutative, $1$ or $2$-dimensional (``unarions'', ``binarions''),

-- or non-commutative, necessarily $4$-dimensional (quaternion algebras).

\nin
For instance, the algebra $V=M(2,2;\K)$ with quadratic form
$q(x)=\det(x)$ and product $\langle XY Z \rangle = XY^\sharp Z$
(where $Y^\sharp$ is the adjugate matrix) is an example of a quaternion algebra
(split case; the form is of signature $(2,2)$).

In the degenerate case, there are new examples: it is not true that the dimension
must be either $1,2$ or $4$; there are degenerate group spherical spaces of any
dimension (Section \ref{sec:representations}).
They are constructed using ideas from Jordan theory \cite{Lo75}: 
there is a notion of {\em (general) representation $U$ of a group spherical space $W$},
and of a {\em split null extension} $V=W \oplus U$ defined by such a representation.
If $U$ is {\em commutative} then $V$ is again a  group spherical space.
Such spaces form an interesting class of spaces, a sort of degenerate analog of
quaternion algebras in arbitrary dimension.
For the time being we have no
examples that are not either of this form, or non-degenerate.

Finally,
in the theory of Jordan-, associative or alternative algebras, an important conceptual step consists in
introducing the notion  of (Jordan-, associative, resp.\ alternative) {\em pair}, see \cite{Lo75}.
Going from triple systems to pairs
 is the analog of considering pairs of vector spaces in duality, instead of considering
spaces with bilinear form. 
We explain in Section \ref{sec:polarized} how
this can be done in the category of group spherical spaces.

\subsection{Moufang spherical spaces.}
The fourth real division algebra, after $\R,\bC,\bH$, is the {\em algebra of octonions}
$\bO$, which has been widely popularized by the paper \cite{Baez}, and is central 
for constructions of ``exceptional objects'' of all sorts, cf.\  \cite{CS, CSl}.
It is not associative but {\em alternative}, and its unit sphere $S^7$ is a 
{\em Moufang loop}. This remains true for more general octonion algebras, 
and tor this reason it is desirable to extend our theory to the case of Moufang loops.

In a first step, we must define a {\em ternary} concept corresponding to the one of 
Moufang loop, in the same way as torsors correspond to groups.
This is done in Appendix \ref{app:torsors} (Theorem \ref{th:ternaryMouf}), which
may have some independent interest.
On the one hand, it takes up Loos' concept of {\em alternative triple system}
\cite{Lo72b, Lo75}, whose defining identities (AP1), (AP2) have been used in joint
work with Michael Kinyon \cite{BeKi14} to define a concept of {\em ternary Moufang loop}.
In Appendix \ref{app:torsors}, we stress that this concept has in fact two versions,
a ``left'' and a ``right'' one, and we give new axioms characterizing them,
inspired by the presentation given by Conway and Smith (\cite{CS}, Section 7):
they are {\em ternary inverse loops}, satisfying moreover a left (resp.\ right)
version of the classical {\em Chasles relation} and of the {\em autotopy property}.
These axioms feature the geometric properties of Moufang loops. 

In a second step, we then define the notion of {\em Moufang spherical space}
(Section \ref{sec:Moufangspheres}):
as above, it is defined by a trilinear map an a quadratic space $(V,q)$ which
again satisfies Kirmse (K) and ternary composition (TC), along with 
Loos' axioms of an alternative triple system; by the above discussion it is then
clear that spheres in such spaces become (left, resp.\ right) ternary Moufang
torsors. Every identity valid in a ternary Moufang torsor has a ``$q$-analog''
for the trilinear product, and the relation between loops and algebras becomes
more transparent.

\subsection{The Albert-Cayley-Dickson (ACD) construction.}
The non-degenerate composition algebras are all constructed, starting from the
unarions $\K$, in several steps by the {\em Albert-Cayley-Dickson construction},
which generalizes the sequence $(\R,\bC,\bH,\bO)$.
It is one of the aims of the present work to contribute to a better understanding of this
construction, from a conceptual viewpoint:

First, the starting point rather is  the ``binarion'' algebra and not the ``unarions'', since
{\em every} binary quadratic form can be used to initialize the construction.
(This makes a difference if $2$ is not invertible in $\K$; cf.\ \cite{Fa}, p.53.)

Second, there is an abstract version of this construction, the ``Moufang double of a group''
(Appendix \ref{app:Moufang-double}), associating to every abstract group $G$ a Moufang
loop $D(G)=G \sqcup G$. The ACD-construction can be understood as being encoded
by such an abstract Moufang double construction.
More specifically, 
$D(G)$ satisfies relations similar to those of the
{\em generalized dicyclic group} of $G$ -- which is a group if $G$ is commutative,
but a non-associative loop if $G$ is non-commutative (cf.\ Def.\ \ref{def:dih(G)}).

Third, there is a {\em ternary} and base point-free
 version of the ACD-construction (``ABCD-construction'').
This version shows that, again, it is necessary and useful to distinguish two versions
of the ACD-construction,  ``left'' and ``right'' 
ones.\footnote{McCrimmon \cite{McC} uses the ``left'', and 
Faulkner \cite{Fa} the ``right'' one.}

Fourth, there is a close relation between quaternion and {\em Clifford algebras}:
we discuss this in Subsections \ref{sec:Clifford1} and \ref{sec:Clifford2}.

Finally,
the construction of the split null extension can be seen as a ``degenerate'' case of the
ABCD-construction.
In particular,  a  split null extension of a non-commutative group spherical space
is a non-associative Moufang spherical space.
This gives new examples of such spaces, which are not given by classical
octonion algebras, and not always of dimension $8$.
It should be interesting to develop a more complete structure theory, either proving
that all new examples are such split null extensions, or exhibiting new non-split
extensions, and describing the structure of a general Moufang spherical space in
terms of extensions of non-degenerate spaces.

\subsection{Further topics.}
I have been working on preliminary versions of this text for several years, and they have
all become too long, so for the present version I decided to exclude several 
sections that would have lead too far ahead.  I will  mention some of these topics
in a last Section \ref{sec:outlook}.





\msk
\nin {\bf
Acknowledgements.}
Many of the ideas and approaches presented in this work have their origin in 
Jordan algebra theory, and I owe gratitude to the pioneers of this domain, and in particular
to Ottmar Loos for his advices and his  criticism. 
I would like to dedicate this work to his memory.

\setcounter{tocdepth}{1}
 \tableofcontents

 \bigskip \nin
{\bf Notation.} 
By $\K$ we denote a commutative base ring $\K$ with unit $1$.
We do not assume that $2$ be invertible in $\K$.
For instance, $\K=\Z$ is admitted.
The canonical symplectic form on $\K^2$ is denoted by
$[x,y]=x_1 y _2 - x_2 y_1$.

\section{All spheres are symmetric spaces}\label{sec:symmetricspaces}

\subsection{Spheres.}
A  map $f:V \to W$ between $\K$-modules $V,W$
 is called {\em quadratic} if 
\begin{enumerate}
\item[(q1)] it is {\em homogeneous of degree two}:
$\forall \lambda \in \K, \forall v \in V$ : $q(\lambda v) = \lambda^2 q(v)$,
\item[(q2)] the {\em polarized map} $b_q$ is $\K$-bilinear: 
\begin{equation}
b_q: V^2 \to W, \quad (u,v) \mapsto b_q(u,v) := q(u+v) - q(u) - q(v).
\end{equation}
\end{enumerate}
A {\em quadratic space} is a $\K$-module $V$ together with a quadratic form
$q:V \to \K$. 
By {\em sphere}, or {\em quadric with center $0$}, we mean the level set
$q^{-1}(c)=\{ x \in V \mid \, q(v) = c \}$ of a quadratic form
$q$, where the scalar $c \in \K$ is assumed to be
{\em invertible in $\K$}: $c \in \K^\times$.
When $V =\K^2$, the form is called a {\em binary quadratic form}, and the spheres
are also called {\em  $q$-conics},  or {\em $q$-circles}.
We will make no assumption on the quadratic form, except that it admits
some non-empty sphere. In other words, we assume that the
{\em set $V^\times$ of invertible elements} of the quadratic form $q$, defined by
(\ref{eqn:Vtimes}), is non-empty.
This set has a double fibration: the ``horizontal'' fibers are the spheres, and the ``vertical'' fibers
are the {\em radii} 
$\K^\times v$, orbits of the action of $\K^\times$ on $V^\times$.

\begin{remark}\label{rk:convention}
Note that $b_q(x,x)=2 q(x)$. 
For every quadratic form $q$,
there exist possibly non-symmetric  bilinear maps $b$ such that $q(x) = b(x,x)$.
If $2$ is invertible in $\K$, then we may take $b = \frac{1}{2} b_q$, so in particular $b$ can be chosen
symmetric. In general, it may be impossible to choose $b$ symmetric.
For a {\em binary quadratic form} $q:\K^2 \to \K$, the choice of a such a form $b$ 
can be done as follows:
we write
\begin{equation}\label{eqn:binaryq}
q(x) = \alpha x_1^2 + \beta x_1 x_2 + \gamma x_2^2 ,
\end{equation}
and then we choose
\begin{equation}\label{eqn:binaryb}
b(x,y) = \alpha x_1 y_1 + \beta x_1 y_2 + \gamma x_2 y_2 .
\end{equation}
Then $b$ is in general not symmetric, and $q(x)=b(x,x)$,
$b_q(x,y) = b(x,y) + b(y,x)$. 
\end{remark}

\subsection{The Jordan structure.}

\begin{Definition}\label{rk:JTS}\label{def:JTS}
In any quadratic space $(V,q)$ we define the {\em Jordan maps}
\begin{align*}
Q:V \times V \to V, \quad 
& Q_x (y ) := Q_x y :=  b_q(x,y) x - q(x) y , \\
D:V^3 \to V, \qquad & D_{x,z}  y  : =     Q_{x+z} y - Q_x y-Q_z y 
\\ & \qquad = b_q(x,y)z + b_q(y,z)x - b_q(x,z)y  .
\end{align*}
\end{Definition}

\begin{remark}\label{rk:Jordan}
Although we won't use this here, let us mention that $(Q,D)$ defines
 a  {\em (quadratic) Jordan triple system},
see \cite{Lo75}, which is sometimes called a {\em spin factor}.
It is uniquely determined by the quadratic form $q$. 
We will see later that it satisfies also the so-called {\em fundamental formula}
(which can be proved by direct but lengthy computation)
\begin{equation}\label{eqn:Fufo}
Q_x \circ Q_y \circ Q_x = Q_{Q_xy} .
\end{equation}
 In general Jordan pair theory,
an element is called {\em invertible} if the operator $Q_x$ is invertible.
The following lemma shows that this is the case iff $q(x) \in \K^\times$, i.e., iff
$x\in V^\times$, and so our notation is in keeping with the Jordan theoretic one.
\end{remark}

\begin{lemma}\label{lemma:Q}
For any quadratic space $(V,q)$ and $e \in V$, the operator $Q_e$ satisfies:
\begin{align*}
(Q_e)^2 & = q(e)^2 \id, \\
q( Q_e(x)) & = q(e)^2 q(x).
\end{align*}
It follows that, if $q(e)$ is invertible in $\K$, then $s_e:V \to V$,
$$
s_e(x) := \frac{Q_e}{q(e)}(x) = 
 \frac{b_q(e,x)}{q(e)} e - x
 $$
defines an isometry of order $2$, fixing $e$.
\end{lemma}

\begin{proof}
The proof is standard in the theory of quadratic forms (cf.\ \cite{Fa}, p.127):
\begin{align*}
Q_e(Q_e(x)) & = b_q\bigl(e, b_q(e,x) e - q(e) x\bigr) e - q(e)  (b_q(e,x) e - q(e) x)
\\ & = \bigl( b_q(e,e) b_q(e,x) - 2 q(e) b_q(e,x)  \bigr) e  +  q(e)^2 x  = q(e)^2 x,
\\
q( Q_e (x)) & = q ( b_q(e,x) e - q(e) x ) 
\\ & = q(b_q(e,x)e) + q(q(e)x) - b_q( b_q(e,x)e, q(e)x) 
\\
& = b_q(e,x)^2 q(e) +
q(q(e)x) - b_q(e,x)^2 q(e)  = q(e)^2 q(x).
\end{align*}
Dividing by $q(e)$, the statements about $s_e$ follow.
(Note $Q_e(e) =q(e)e$.)
\end{proof}

\subsection{Existence of $q$-compatible ternary products.}

\begin{lemma}\label{la:existence} 
On every quadratic space $(V,q)$, there exists a $q$-compatible trilinear product $\langle xyz \rangle$,
i.e., we have the Kirmse identity 
 $\langle xxy \rangle = q(x)y = \langle yxx\rangle$. 
This product is in general not unique, but
 the ``outer diagonal''  $\langle xyx \rangle$ is, and it agrees with the Jordan map
defined above: 
$$
\forall x,y \in V : \qquad \langle x \vert y \vert x \rangle = b_q(x,y)x - q(x) y = Q_xy .
$$
\end{lemma} 

\begin{proof}
Choose any bilinear form $b$ on $V$ such that $q(x)=b(x,x)$, and let
\begin{equation}\label{eqn:drei} 
\langle x\vert y \vert z\rangle  :=
b(x,y)z - b(x,z)y + b( y,z) x .
\end{equation}
This clearly is trilinear, and for $x=y$ it gives
$b(x,x) z - b(x,z)x+b(x,z)x= q(x) z$ and for $y=z$ we obtain
$b(x,z)z - b(x,z,) z + b(z,z)x= q(z)x$, whence $q$-compatibility.
Moreover,  polarizing the condition $q(x) y = \langle x \vert x \vert y \rangle = q(x)y$,
 we get
\begin{equation}\label{eqn:R+}
\langle  x\vert z \vert y \rangle  + \langle  z \vert x\vert y \rangle  = b_q(x,z) y.
\end{equation} 
Letting $z=y$, we get $q(z)x + \langle  z\vert x \vert z\rangle   = b_q(x,z)z$, whence the formula for the
outer diagonal $\langle x y x \rangle$.
\end{proof}

The trilinear product from the lemma is related to the Jordan map $D$ via
$\langle xyz\rangle + \langle zyx \rangle = D(x,z)y$.
In general, the ternay product $\langle xyz \rangle$
 will have a skew-part, which depends on the choice of 
the skew-part of $b$.  
The notation $\langle xyx \rangle$ is often more intuitive than $Q_x(y)$ since it becomes better
visible that this expression is {\em quadratic in $x$ and linear in $y$}.

\subsection{Reflection spaces and symmetric spaces.}\label{sec:reflectionspace}
Following Loos \cite{Lo67, Lo69}, we call {\em reflection space} a set $M$ together with a 
``product map''
$\sigma: M \times M \to M$, $(x,y) \mapsto \sigma(x,y)=:\sigma_x(y)$,
satisfying the following identities:
\begin{enumerate}
\item[(S1)] $\sigma_x(x)=x$
\item[(S2)] $\sigma_x \circ \sigma_x = \id$
\item[(S3)] $\sigma_x \circ \sigma_y \circ \sigma_x =\sigma_{\sigma_x(y)}$.
\end{enumerate}
In other words, the ``left translations'' $\sigma_x$ are maps of order two fixing $x$ and
are automorphisms of the whole structure.
If, moreover, $M$ carries a topology such that the fixed point $x$ of $\sigma_x$ is isolated,
then we call $M$ a {\em (topological) symmetric space}.
Every {\em smooth and connected} symmetric space is a homogeneous space
$M = G/H$ under a Lie group action (\cite{Lo69}), but in general it is not true that the automorphism
group of a reflection space $M$ acts transitively on $M$.
For instance, $S^n = \SO(n+1)/\SO(n)$ is a compact homogeneous symmetric space,
but the set $V^\times$ will in general be far from being homogeneous:

\subsection{Spheres as symmetric spaces.}
{\em All} spheres  carry a canonical structure of
{\em symmetric space}. 
The set $V^\times$ carries even {\em three}, in general different, structures of 
reflection space. This is related to the ``double fibration'' of $V^\times$, mentioned above.
To fix ideas, think of $\bC^\times$ with three ``inversions at $e=1$'', namely
$$
s_e(z)=\overline z, \qquad
j_e(z) = \frac{z}{\vert z \vert^2} = \frac{1}{\overline z}, \qquad 
\sigma_e(z) = \frac{1}{z} ,
$$
the first two anti-holomorphic (fixed points are ``real forms''), and the last one holomorphic
(fixed points isolated).

\begin{theorem}
Assume $(V,q)$ is a quadratic space having invertible elements, i.e., $V^\times $ is non-empty.
Then $M = V^\times$ carries three ``product maps'' which all satisfy the axioms 
{\rm (S1), (S2), (S3)}: 
\begin{enumerate}
\item
(``inversion at diameters'')
$$
s:  V^\times \times V^\times \to V^\times, \quad (x,y) \mapsto s_x(y) = \frac{Q_x}{q(x)} (y) = 
 \frac{b_q(y,x)}{q(x)} x - y
$$
\item
(``inversion at spheres'')
$$
j:  V^\times \times V^\times \to V^\times, \quad (x,y) \mapsto j_x(y) :=  \frac{q(x)}{q(y)} y
$$
\item
(``point inversive'')
$$
\sigma:  V^\times \times V^\times \to V^\times, \quad (x,y) \mapsto \sigma_x(y) :=  \frac{Q_x(y)}{q(y)} =  \frac{b_q(y,x)}{q(y)} x -  \frac{q(x)}{q(y)} y .
$$
\end{enumerate}
All of these maps have well-defined  restriction  to spheres  $S=\{ x \in V \mid q(x) = r \}$ with $r \in \K^\times$, which thus become   sub-reflection spaces of $V^\times$.
For $j$, these restrictions to spheres are trivial, and for $s$ and $\sigma$ they coincide. 
\end{theorem}

\begin{proof} 
First of all, note homogenity and degrees:

$Q_x y$ : quadratic in $x$, linear in $y$,

$s_x y$ : $0$ in $x$, linear in $y$,

$j_x (y)$: quadratic in $x$, $-1$ in $y$,

$\sigma_x(y)$: quadratic in $x$, $-1$ in $y$,

\nin and the fixed point spaces:

$s_x$ fixes pointwise the whole line $\K x$, since it is linear in $y$

$j_x$ fixes pointwise the whole sphere $\{ u \in V \mid q(u) = q(x)\}$, 

$\sigma_x = s_x j_x = j_x s_x$ fixes $x$, which appears to be an ``isolated'' fixed point.

\ssk
\nin
Concerning $s$, properties (S1), (S2) are contained in
Lemma \ref{lemma:Q}, and (S3) follows from that lemma,
observing that 
$g \circ s_y  \circ g^{-1} = s_{g(y)}$ for every $q$-isometry $g$, and in particular
for $g= s_x$.
Note that, by rescaling with invertible scalars, this gives us the ``fundamental formula''
(\ref{eqn:Fufo}), for $x,y \in V^\times$. 

Concerning
$j$,  properties (S1), (S2), (S3) follow by direct computation, only using that
$q$ is homogeneous of degree $2$.

Concerning $\sigma$,
 by direct computation $j_x \circ s_x = \sigma_x = s_x \circ j_x$, whence
 (S1) and (S2). 
  (S3) follows from (\ref{eqn:Fufo}), together with the  homogenity
 observed above.
 
Finally, $q(x)=r=q(y)$ implies that $q(s_x y) = q(y) = r$, etc., so spheres are stable under
the three binary products. Clearly, the restriction of $j$ to a sphere is trivial:
 $j_x(y) = y$  if $q(x)=q(y)=r$. 
On the other hand, the set of multiples of an element
 $v \in V^\times$ is also stable under the product, and
then $s$ restricts to the trivial structure, and the restriction of  $j$ corresponds to the symmetric space
structure $(r,s) \mapsto r s^{-1} r = r^2 s^{-1}$ on $\K^\times$. 
\end{proof}

\begin{remark}
The three reflection space structures behave very much like those of a {\em symmetric bundle}
in the sense of \cite{BeD}.
More precisely, the quotient of $V^\times$ by the equivalence relation
$v \sim w$ iff $\exists \lambda \in \K^\times: w = \lambda v$, becomes a symmetric space,
and $V^\times$ becomes a symmetric bundle over the quotient; the fibers are multiplicatively
written, rather than additively as in \cite{BeD}.
\end{remark} 

\subsection{The set of root vectors.}\label{sec:roots1}
The following remarks will not be used in the main text.
For every quadratic space $(V,q)$ over $\K$, we define it {\em set of root vectors}
\begin{equation}
V^\diamond := \{ 
y \in V \mid \, q(y) \not= 0, \,
\forall x \in V : \exists n = n_{x,y} \in \K :
b_q(x,y) = n_{y,x} \cdot q(y) \} .
\end{equation}
Clearly, $V^\times \subset V^\diamond$ (equality if $\K$ is a field), and for $y \in V^\times$,
we then have
\begin{equation}
n_{y,x} = \frac{b_q(x,y)}{q(y)}  = 2  \frac{b_q(x,y)}{b_q(y,y)}  ,
\end{equation}
which is uniquely determined, and then we have
\begin{equation}
s_x(y) = n_{x,y} x - y.
\end{equation}
More generally, $n_{x,y}$ is uniquely determined when the $\K$-module has no torsion.
For instance, this holds if $V$ is a free module over an integral domain, like $\Z$ --
this is the most interesting case, where this definition takes up the usual ones from the
theory of {\em root systems}.
We will not use it in the sequel, but see Subsection \ref{sec:roots2}.
Let us just note here that, unlike $V^\times$, the set  $V^\diamond$ does not depend on
scaling of the form: one may replace $q$ by $\lambda q$ for any invertible $\lambda$,
and that the product $s$ extends to a map
\begin{equation}
V^\diamond \times V^\diamond \to V^\diamond, \quad
(x,y) \mapsto s_x(y) = n_{x,y} x - y
\end{equation}
that turns $V^\diamond$ again into a reflection space.
The group generated by all $s_x$ with $x \in V^\diamond$ could be called the
{\em Weyl group of the quadratic form $q$}.

\begin{example}\label{ex:rootcase}
Let $\K = \Z$, $V = \Z^2$, and $q(x) = \alpha x_1^2 + \beta x_1 x_2 +\gamma x_2^2$.
Then $e_1$ is a root vector iff $\frac{\beta}{\alpha} \in \Z$ and $e_2$ is one iff
$\frac{\beta}{\gamma}\in \Z$. 
If the form is positive definite, the root vectors form a root system in the usual sense.
For instance, let
$q(x) = x_1^2 - x_1 x_2 + x_2^2$.
Then the unit sphere agrees with $V^\times$ and
 has $6$ elements $\{ \pm e_1,\pm e_2 , \pm(e_1 + e_2) \}$ (summits of regular hexagon),
 whereas $V^\diamond$ has $12$ elements, namely the 
 \href{https://en.wikipedia.org/wiki/G2_(mathematics)#/media/File:Root_system_G2.svg}{$12$ root vectors of the root system $G_2$}.
\end{example}

\section{All circles are (commutative) torsors}\label{chap:2D}

\subsection{The $2D$-Theorem.}
In general, existence of group structures on spheres is a very restrictive condition.
Therefore it is remarkable that
the following result holds for {\em any binary} quadratic form, even without assuming
that $V^\times$ is non-empty. However, if invertible elements exist, the proof is simpler and better
intelligible, so in the proof  we concentrate
on that case. 

\begin{theorem}\label{th:circles}\label{th:ternary}
Let $V = \K^2$ with quadratic form
$q(x) = \alpha x_1^2 + \beta x_1 x_2 + \gamma x_2^2$.
\begin{enumerate}
\item
There exists a unique  trilinear map
$V^3 \to V$, $(x,y,z)\mapsto \langle xyz \rangle$ satisfying the Kirmse identity
$\langle xxy \rangle = q(x)y = \langle yxx \rangle$.
\item
This trilinear map is commutative,  and para-associative.
\item
It satisfies the ternary composition rule  $q(\langle xyz\rangle)=q(x)q(y)q(z)$.
\end{enumerate}
Summing up,
$V$ is canonically a commutative group spherical space, and
every circle $S= \{ x \in V \mid q(x) = c \}$ (with $c$ invertible in $\K$) carries a canonical
torsor structure, which is commutative and given by
$$
(x,y,z) \mapsto \frac{1}{c} \langle xyz \rangle.
$$
\end{theorem} 

\begin{proof}
1. The existence statement holds for any quadratic space (Lemma \ref{la:existence}).
Let us prove uniqueness, in  case of dimension 2. 
Denote the standard basis of $\K^2$ by $(e_1,e_2)$. 
We have already noticed (Lemma \ref{la:existence}) that the values
$\langle xyx \rangle$ are uniquely determined by $q$. 
To compute 
$\langle e_i e_j e_k\rangle $, observe that always the values of two of the three
indices $i,j,k$ coincide, 
 so from $q$-compatibility  we obtain the values shown in  Table \ref{table:products}.
\begin{table}[h]\caption{Products of triples of basis vectors}\label{table:products}
\begin{center}
\begin{tabular}{|*{10}{c|}}
\hline
$i$ & $j$ & $k$  & $\langle  e_i \vert e_j \vert e_k \rangle $
\\
\hline
1 & 1 & 1 & $\alpha e_1$ 
\\
\hline
2 & 2 & 2 & $\gamma e_2$
\\
\hline
1 & 1 & 2 & $\alpha e_2$ 
\\
1 & 2 & 1 &  $\beta e_1 - \alpha e_2$ 
\\
2 & 1 & 1 & $\alpha e_2 $
\\
\hline
1 & 2 & 2 & $\gamma e_1$ 
\\
2 & 1 & 2 & $\beta e_2 - \gamma e_1$
\\
2 & 2 & 1 & $\gamma e_1$
\\
\hline
\end{tabular}
\end{center}
\end{table}
Uniqueness now follows:
by trilinearity,  the ternary product is given by 
\begin{equation}\label{eqn:ternary0} 
\langle  xyz \rangle  = \sum_{(i,j,k) \in \{ 1,2 \}^3} x_i y_j z_k  \,  \langle e_i e_j e_k \rangle  . 
\end{equation}
We collect and add terms belonging to indices $(i,j,k)$ such that
$ \langle e_i e_j e_k \rangle $ is a multiple of $e_1$ (first component), and then such that
$\langle e_i e_j e_k \rangle $ is a multiple of $e_2$ (second component).
Using table \ref{table:products}, the result is 
\begin{equation}\label{eqn:ternary}
\langle  xyz \rangle  = 
\begin{pmatrix}
\alpha x_1 y_1 z_1 + \beta x_1 y_2 z_1 + \gamma (x_1 y_2 z_2 + x_2 y_2 z_1 - x_2 y_1 z_2) 
\\
\gamma x_2 y_2 z_2 + \beta x_2 y_1 z_2  + \alpha (x_1 y_1 z_2 + x_2 y_1 z_1 - x_1 y_2 z_1)
\end{pmatrix} .
\end{equation} 
Thus the ternary map is uniquely determined by $q$-compatibility.

\begin{remark}\label{rk:uniqueness}\label{rk:dreier}
 It follows that the left hand side in (\ref{eqn:drei}) 
is independent of the choice of $b$ such that $q(x)=b(x,x)$. In particular, for $q=0$, 
$b$ is alternating, so it is a multiple of $\det$, we get an identity
$[x,y]z + [y,z]x+[z,x]y = 0$ called ``Dreier-Identit\"at'' in \cite{KK}.
\end{remark}

2. Commutativity (symmetry in $i$ and $k$) is clear from the table.
Let us prove para-associativity.
Assuming the theory of Clifford algebras, a short proof is indicated in
Subsection \ref{sec:Clifford1}.
In the following, we will give an elementary and direct proof, independent of the
theory of Clifford algebras.

Assume first that  $(V,q)$ admits invertible elements.  
Fix an invertible element $e$ and take $e_1 := e$ as first basis vector, completed
by a second vector $e_2$ to a basis of $V \cong \K^2$.
Consider the bilinear product 
\begin{equation}
x \cdot y := \langle x e y \rangle .
\end{equation}
From the table above, we deduce the following  binary  ``multiplication table''
\begin{center}
\begin{tabular}{|*{10}{c|}}
\hline
$\cdot_e$ & $e_1$ & $e_2$ \\
\hline
$e_1$ & $\alpha e_1$ & $ \alpha e_2$ \\
\hline
$e_2$ & $\alpha e_2$ & $ \beta e_2 - \gamma e_1$   \\
\hline
\end{tabular}
\end{center}
Clearly, this product is commutative, and it is  associative:
note first that $e_1$ is a multiple of the unit, and hence
we have $(e_i e_j)e_k = e_i (e_j e_k)$ when one of the indices $i,j,k$ equals $1$.
 And $e_2^2 e_2 = e_2 e_2^2$ by commutativity of the binary product, so the product is associative
(i.e, every unital $2$-dimensional commutative algebra is automatically associative.)

The reflection $\sharp = s_e$ defined in Lemma \ref{lemma:Q} 
is an isometry of order $2$  fixing $e$. 
Since the trilinear product is determined by $q$, the map $\sharp$
 is also an automorphism of this structure, and
since it fixes $e$, it is an automorphism of the bilinear product.
Now we define a trilinear product 
$$
A: V^3 \to V, \quad 
(x,y,z) \mapsto A(x,y,z):=  \frac{1}{q(e)} xy^\sharp z .
$$
Since the binary product is associative, and $\sharp$ an involution,
this ternary product is para-associative:
$((ab^\sharp c)d^\sharp e) = (a (d c^\sharp b)^\sharp e) = (ab^\sharp (cd^\sharp e))$.
Let us show that $xx^\sharp = q(x) e_1$:
\begin{align*}
x  x^\sharp & = 
(x_1 e_1 + x_2 e_2)   (x_1 e_1^\sharp + x_2 e_2^\sharp) 
\\
&=
(x_1 e_1 + x_2 e_2)   (x_1 e_1  + x_2 (\frac{\beta}{\alpha} e_1 - e_2)) 
\\ 
&=
(x_1 e_1 + x_2 e_2)   ((x_1 +\frac{\beta}{\alpha})  e_1  - x_2 e_2))
\\
& = (\alpha x_1^2 + \beta x_1 x_2 +\gamma x_2^2) e_1 + 0 e_2  = q(x) e_1  
\end{align*}
(using the ``multiplication  table''; all terms containing $e_2$ cancel out).
From this, we get $A(x,x,y) = \frac{1}{\alpha} xx^\sharp y = q(x) \frac{e_1 y}{\alpha} = q(x)y = A(y,x,x)$, i.e.,
$q$-compatibility of the ternary product map $A$. 
By uniqueness, we thus have
$A(x,y,z) = \langle xyz \rangle$.
Since $A$ is  para-associative, so is $\langle -,-,-\rangle$.
Similarly, using commutativity and associativity,  we get the  ternary composition law
$$
q(\langle xyz \rangle) = q(xy^\sharp z) = (xy^\sharp z) (xy^\sharp z)^\sharp =
xx^\sharp yy^\sharp zz^\sharp = q(x) q(y) q(z),
$$
finishing the proof in case $(V,q)$ admits invertible elements.
 
Now let us drop the assumption that $(V,q)$ contains invertible elements, and indicate
the proof in this more general case (without going into all of the details). 
To prove associativity,
by multilinearity, it is enough to show the associativity property for the base vectors, that is,
$$
\langle  \langle  e_i e_j e_k) e_\ell e_m \rangle  =
\langle  e_i \langle  e_j e_k e_\ell \rangle   e_m \rangle  =
\langle  e_i e_j \langle  e_k e_\ell e_m \rangle \rangle  
$$
for all $(i,j,k,\ell,m) \in \{ 1,2\}^5$. 
By commutativity,
the $2^5$ cases to be checked reduce effectively to the following 10 
cases: 
\begin{enumerate}
\item
$(i,i,i,i,i)$ (the values of all $5$ indices coincide, with $i=1$ or $2$),
\item
(the values of exactly $4$ indices coincide): types
$(j iiii)$ ,
$(ijiii)$ , 
$(iijii)$

\item
(the values of exactly $3$ indices coincide): types
$(jj iii)$ ,
$(jijii)$ ,
$(jiiji)$ ,
$(jiiij)$ , 
$(ijjii)$ ,
$(ijiji)$
\end{enumerate}
Using  Table \ref{table:products}, the computation is straightforward, and indeed in each case
we get the same result for all three ways of bracketing, which we denote by
$\langle e_i e_j e_k e_\ell  e_m\rangle$. 
For better readability we represent the cases given above by definite choices for $i$ and $j$.
The result is given in Table \ref{table:products2}. 
\begin{table}[h]\caption{Products of $5$-tuples of basis vectors}\label{table:products2}
\begin{center}
\begin{tabular}{|*{10}{c|}}
\hline
$i$ & $j$ & $k$  & $\ell$ & $m$  & $\langle e_i e_j e_k e_\ell  e_m\rangle$
\\
\hline
1 & 1 & 1 & 1 & 1 &  $\gamma  e_1$ 
\\
\hline
2 & 1 & 1 & 1 & 1 & $\alpha^2 e_2$
\\
1 & 2 & 1 & 1 & 1 & $\alpha \beta e_1 - \alpha \gamma e_2$
\\
1 & 1 & 2 & 1 & 1 & $\alpha^2 e_2$
\\
\hline
2 & 2 & 1 & 1 & 1 & $\alpha \gamma  e_1$
\\
2 & 1 & 2 & 1 & 1 & $\alpha \beta e_2 - \alpha \gamma e_1$
\\
2 & 1 & 1 & 2 & 1 & $\alpha \gamma e_1$
\\
2 & 1 & 1 & 1 & 2 & $\alpha \beta e_2 - \alpha \gamma e_1$
\\
1 & 2 & 2 & 1 & 1 & $\alpha \gamma  e_1$
\\
1 & 2 & 1 & 2 & 1 & $(\beta^2 - \alpha \gamma) e_1 - \beta \alpha e_2$
\\
\hline
\end{tabular}
\end{center}
\end{table}
Finally,  the proof of the  ternary 
composition law by ``brute force computation''
would be fairly long and involved. 
In the special case where  $a,b$ are linearly independent,
the identity $q(\langle axb \rangle) = q(a)q(b) q(x)$ follows from
 the following Lemma \ref{la:composition}:

\begin{lemma}\label{la:composition}
Assume  $a,b$ are linearly independent in $\K^2$, and $q$ is a binary quadratic form.
Then the  unique linear map $S = S_{a,b}:\K^2 \to \K^2$
such that
$S_{a,b}(a) = q(a)b$ and $ S_{a,b}(b) = q(b)a$
is a $q$-similarity with ratio $q(a) q(b)$, that is,
for all $x \in \K^2$, we have
$$
q ( S_{a,b} (x)) = q(a) q(b) q(x) .
$$
In particular, when $q(a) q(b) = 1$, then $S_{a,b}$ is an isometry.
\end{lemma}

\begin{proof}
Let $x = u_1 a + u_2 b$, so
$q(x) = u_1^2 q(a) + u_1 u_2 b_q(a,b) + u_2^2 q(b)$,
and
\begin{align*}
q\bigl(S_{a,b} x) & = q( u_1 q(a) b + u_2 q(b) a \bigr) \\
& = u_1^2 q(a)^2 q(b) + u_1 q(a) u_2 q(b) b_q(a,b) + u_2^2 q(b)^2 q(a) \\
& = q(a) q(b) \cdot  \bigl(  u_1^2 q(a) + u_1 u_2 b_q(a,b) + u_2^2 q(b) \bigr) =
q(a) q(b) q(x) ,
\end{align*}
whence the claim. 
\end{proof}

\nin Likewise,
when  $a=b$,  the identity $q(\langle axa \rangle) = q(a)^2 q(x)$ follows from Lemma \ref{lemma:Q}.

For a general proof of the ternary composition rule,
 one may inject $\K^2$ into 
its  ``concrete Clifford-quaternion algebra''  $\bfH_q$, 
(see Subsections \ref{sec:Clifford1} and \ref{sec:Clifford2}),
which is a composition algebra {\em with} invertible elements (namely, it has
a unit element); then we can apply the general arguments to be developed
in the next section (Theorem \ref{th:binary-ternary}, Item 4), and then restrict to $\K^2$ again. 

The last statements of the theorem are consequences of the preceding results, see Theorem \ref{th:torsortheorem}.
\end{proof}

\subsection{Examples, and matrix realization.}

\begin{example}\label{ex:hyp} 
Assume the form $q$ is hyperbolic. Without loss of generality, assume that
$q(x) = x_1 x_2$. 
We claim that  then  the unique trilinear map is given by
\begin{equation}\label{eqn:hyperbolic}
\langle a x b \rangle  = \begin{pmatrix}
a_1 x_2 b_1 \\
a_2 x_1 b_2 \end{pmatrix} . 
\end{equation}
Indeed, it
is $q$-compatible:
 for $x=a$, the right hand side gives $a_1 a_2 b = q(a)b$ and for $x=b$ it gives
$b_1 b_2 a = q(b)a$, as it should, whence the claim.
In this case,  it follows directly from the explicit formula 
(\ref{eqn:hyperbolic}) that the ternary
product $\langle a x b \rangle$ is {\em associative} and {\em commutative}, 
without using the general result. 
For $\K = \R$ or $\bC$, since the set of hyperbolic forms is open in the space of all forms,
this example could be used to deduce general results ``by polynomial density''.
\end{example}

\begin{example}\label{ex:product}
Generalizing the preceding example, let us assume that
$q(x) = \phi(x) \psi(x)$ where $\phi$ and $\psi$ are two linear forms.
The bilinear form 
$b(x,y) = \phi(x) \psi(y)$ satisfies
$q(x)=b(x,x)$, and from (\ref{eqn:drei}) we get
$$
\langle x \vert y \vert z \rangle  = \phi(x) \psi(y) z + \psi (z) \phi(y) x - \phi(x) \psi(z) y.
$$
When $\phi$ and $\psi$ are linearly independent in the dual of $\K^2$, then 
$(\K^2,q)$ is a \href{https://en.wikipedia.org/wiki/Minkowski_plane}{\em hyperbolic (Minkowski) plane}: by a change of coordinates, we are back in the preceding example.
When $\phi = \psi$, then we obtain a
\href{https://en.wikipedia.org/wiki/Laguerre_plane}{\em Laguerre plane}. 
We'll see (Theorem \ref{th:Minko})
 that such formulae define  a group spherical  space on {\em any} $\K$-module, which
than can be considered as {\em extended Minkowski (resp.\ Laguerre) planes}.
\end{example}
 
\subsection{Matrix realization: the dihedral algebra of $q$.} 
In the following, let
 $q(x)= \alpha x_1^2 + \beta x_1 x_2 +\gamma x_2^2$ be an arbitrary  binary quadratic form on $V=\K^2$.
 We'll give explicit formulae, showing that the whole theory generalizes usual formulae
known from the Euclidean case $\alpha = 1 = \gamma, \beta=0$, describing relations
valid for the linear operators, called {\em spirations} and {\em spiflections},
see Eqn.\ (\ref{eqn:ABC}).
These operators can be identified with $2 \times 2$-matrices, which we will give explicitly,  by
rewriting  (\ref{eqn:ternary}) as follows. 
From the second equality one can read off the matrix of $R_{x,y}$ and from the third
equality, the matrix of $S_{x,z}$ (recall notation $[x,y] = x_1 y_2 - x_2 y_1$),
\begin{align*}\label{eqn:ternarybis}
\langle x\vert y\vert z \rangle &  = 
\begin{pmatrix}
\alpha x_1 y_1 z_1 + \beta x_1 y_2 z_1 + \gamma (x_1 y_2 z_2 + x_2 y_2 z_1 - x_2 y_1 z_2) 
\\
\gamma x_2 y_2 z_2 + \beta x_2 y_1 z_2  + \alpha (x_1 y_1 z_2 + x_2 y_1 z_1 - x_1 y_2 z_1)
\end{pmatrix}
\\
& = \begin{pmatrix}
\alpha x_1 y_1 + \beta x_1 y_2 + \gamma x_2 y_2 & \gamma [x,y] \\
\alpha[y,x] & \alpha x_1 y_1 +\gamma x_2 y_2 + \beta x_2 y_1 \end{pmatrix}
\begin{pmatrix} z_1 \\ z_2 \end{pmatrix}
\\
& =
\begin{pmatrix}
\alpha x_1 z_1 - \gamma x_2 z_2 & \beta x_1 z_1 + \gamma (x_1 z_2 + x_2 z_1) \\
\beta x_2 z_2 + \alpha (x_1 z_2 + x_2 z_1) &  \gamma x_2 z_2 - \alpha x_1 z_1 
\end{pmatrix}
\begin{pmatrix} y_1 \\ y_2 \end{pmatrix}
\end{align*}

\begin{Definition}\label{def:dihedral}
The  {\em (right) spiration algebra} $\bfC_q^R$ is the submodule of $M(2,2;\K)$
generated by all
$R_{x,y}$, $x,y \in \K^2$, and by $\id_{\K^2}$.
The
{\em spiflection algebra} $\bfC_q^S$ is the submodule of $M(2,2,;\K)$ generated by
all $S_{x,z}$, $x,z \in \K^2$.
The {\em dihedral algebra of $q$} is the submodule $\bfD_q := \bfC_q^R + \bfC_q^S$ of
$M(2,2;\K)$.
\end{Definition}

\nin
To describe these algebras by generators and relations, let us abbreviate
\begin{equation}
R_{ij}:= R_{e_i,e_j}, \qquad S_{ij}:= S_{e_i,e_j} , \qquad I := 2 \times 2 \mbox{ unit matrix}.
\end{equation}
From Table \ref{table:products}, or from the matrices given above, one can read off:
\begin{equation}\label{eqn:R-matrix} 
R_{12} = \begin{pmatrix} \beta & \gamma \\ - \alpha & 0 \end{pmatrix},
\quad
R_{21} = \begin{pmatrix} 0 & - \gamma \\ \alpha & \beta \end{pmatrix} ,
\quad
R_{11} = \alpha I, \quad  R_{22} = \gamma I,
\end{equation}
\begin{equation}
\label{eqn:S-matrix} 
S_{12} = \begin{pmatrix} 0 & \gamma \\ \alpha & 0 \end{pmatrix} = S_{21}, \quad
S_{11} = \begin{pmatrix} \alpha &  \beta \\ 0 & - \alpha \end{pmatrix}, \quad
S_{22} = \begin{pmatrix} -\gamma & 0 \\ \beta & \gamma \end{pmatrix} .
\end{equation}

\begin{theorem}\label{th:spiration2} 
Assume $(\K^2,q)$ with $q(x)= \alpha x_1^2 + \beta x_1 x_2 +\gamma x_2^2$
 is a binary quadratic space, und use notation introduced above.
Then, for all $a,b \in \K^2$, 
\begin{align*}
\tr (R_{a,b}) & = b_q(a,b) ,
\\
\det (R_{a,b}) & = q(a) q(b),
\\
\tr (S_{a,b}) & = 0 ,
\\
\det (S_{a,b}) & = - q(a) q(b),
\\
R_{a,b}^2 & = b_q(a,b) R_{a,b} - q(a) q(b) \id,
\\
S_{a,b}^2 & = q(a) q(b) \id.
\end{align*}
The adjugate matrix of $R_{a,b}$ is $(R_{a,b})^\sharp = R_{b,a}$, and the one of $S_{a,b}$ is $-S_{a,b}$.
In particular, the spiration algebra $\bfC_q^R$ is  stable under its involution ``adjugate'',
and on $\bfC_q^S$ the ``adjugate'' map agrees with $-\id$. 
\end{theorem} 

\begin{proof} 
Using para-associativity, we have
\begin{align*}
R_{a,b}^2  & = R_{\langle aba \rangle, b}   = R_{ b_q(a,b) a,b} - R_{q(a)b,b}  = 
b_q(a,b) R_{a,b} - q(a) q(b) \id 
\\
 S_{a,b} S_{b,a}& = R_{a,\langle abb\rangle} = q(b) R_{a,a} = q(b) q(a) \id . 
\end{align*}
To compute the trace of $R_{a,b}$, it suffices to do it for $(a,b)= (e_i,e_j)$,
$i,j=1,2$, i.e., for 
 the $4$ matrices (\ref{eqn:R-matrix}). 
 In all 4 cases, the trace is $b_q(a,b)$, whence $\tr R_{a,b} = b_q(a,b)$. 
Now, every endomorphism $X \in \End(\K^2)$ satisfies the characteristic equation
$$
X^2 - \tr(X) X + \det(X) \id  = 0 .
$$
Let $X = R_{a,b}$ and compare with the relation $R_{a,b}^2 =
b_q(a,b) R_{a,b} - q(a) q(b) \id$: it follows that
$\det R_{a,b} = q(a) q(b)$. 
Concerning spiflections, the relation $\tr(S_{a,b})=0$ is gotten 
from the four matrices (\ref{eqn:S-matrix}).
As above, this  implies that $\det (S_{a,b}) = q(a) q(b)$. 
The adjugate matrix $\tilde X$ of $X \in M(2,2;\K)$ is given by
$\tilde X = \tr(X) I - X$, whence 
$\wt R_{a,b} = b_q(a,b)  \id  - R_{a,b} =   R_{b,a}$
and
$\wt S_{a,b} = - S_{a,b}$. 
\end{proof}

From (\ref{eqn:R-matrix}) and (\ref{eqn:S-matrix}), we get the following linear relations
\begin{equation}
\begin{matrix}
\gamma R_{11} = \alpha R_{22} = \alpha \gamma I, & { } \quad
R_{12}+R_{21} = \beta I, 
\\
\gamma S_{11} + \alpha S_{22} = \beta S_{12}, & S_{12}=S_{21} .
\end{matrix}
\end{equation}
(``By density'', the third of these relations should imply that, for all $a,b$,
\begin{equation}
q(b) S_{a,a} + q(a) S_{b,b}  = b_q(a,b) S_{a,b} .
\end{equation}
A direct computation would be long, and the relation won't be needed in the sequel.)
Next, we record product relations among the matrices:
from $R_{12}^2 = \beta R_{12} - \alpha \gamma I$ we get the
following two versions of the ``multiplication table'' of the
spiration algebra $\bfC_q^R$:
\ssk
 \begin{center}
\begin{tabular}{|*{10}{c|}}
\hline
$\circ$ & $R_{11}$ & $R_{12}$ \\
\hline
$R_{11}$ & $\alpha R_{11}$ & $ \alpha R_{12}$ \\
\hline
$R_{12}$ & $\alpha R_{12} $ & $  \beta R_{12} -  \gamma R_{11} $   \\
\hline
\end{tabular}
$\qquad \qquad$
\begin{tabular}{|*{10}{c|}}
\hline
$\circ$ & $I$ & $R_{12}$ \\
\hline
$I$ & $I$ & $ R_{12}$ \\
\hline
$R_{12}$ & $R_{12} $ & $  \beta R_{12} - \alpha \gamma I $   \\
\hline
\end{tabular}
\end{center}

\begin{proposition}
Assume that $I$ and $R_{12}$ are linearly independent (note : {\rm (\ref{eqn:R-matrix})} shows
that  this is always the case
if $q$ has invertible elements). Then
the  spiration algebra $\bfC_q^R$ is isomorphic to the quotient algebra
$\K[X] / (X^2 - \beta X +\alpha \gamma)$.
\end{proposition}

\begin{proof} By the table, $R_{12}$ satisfies the same quadratic relation as
$[X]$ in the quotient algebra. 
\end{proof}

\begin{remark}\label{rk:scaling}
Note that $\bfC_q^L$ and $\K[X]/(X^2 - \beta X + \alpha \gamma)$ are invariant under scaling
of $q$ by invertible scalars.
In particular, they cannot distinguish between $q$ and $-q$.
\end{remark}

\nin
From para-associativity and commutativity we get $S_{a,a}S_{b,b} = R_{a,b} L_{a,b} = R_{a,b}^2$, and
\begin{equation}
S_{a,a} S_{a,b}(x)= \langle a \langle axb \rangle a \rangle =
\langle ab \langle xaa \rangle \rangle =
q(a) R_{a,b} (x),
\end{equation}
and likewise $S_{a,b}\circ S_{a,a}= q(a) R_{b,a}$,
whence the following table:
\ssk
 \begin{center}
\begin{tabular}{|*{10}{c|}}
\hline
$\circ$ & $S_{11}$ & $S_{22}$ & $S_{12}$  \\
\hline
$S_{11}$ & $\alpha^2 I$ & $ R_{12}^2 $ & $\alpha R_{12}$  \\
\hline
$S_{22}$ & $R_{21}^2 $ & $\gamma^2 I  $ &   $\gamma R_{12}$  \\
\hline
$S_{12}$ & $\alpha R_{21}$   & $\gamma R_{21} $  & $\alpha \gamma I$ \\
\hline
\end{tabular}
\end{center}
\ssk
Next, compositions between $R$- and $S$-operators: from para-associativity, 
we have $S_{11} R_{12} (x) = \langle e_1 \langle e_1 e_2 x \rangle e_1 \rangle =
\langle e_1 e_1 \langle e_2 x e_1 \rangle \rangle = \alpha S_{21} (x)$, etc., so:
\ssk
 \begin{center}
\begin{tabular}{|*{10}{c|}}
\hline
 & $S_{11}$ & $S_{22}$ & $S_{12}$  \\
\hline
$R_{12} \circ S_{ij}$ &  $\beta S_{11} - \alpha S_{12}$ & $\gamma S_{12}$  & $\gamma S_{11}$  \\
\hline
$S_{ij} \circ R_{12} $ & $\alpha S_{12}$   & $\beta S_{22} - \gamma S_{12}$  & $\alpha S_{22}$  \\
\hline
\end{tabular}
\end{center}
\ssk
When computing the products of {\em three} $R$- or $S$-operators, we essentially
get back the ternary composition on $V$:

\begin{theorem}\label{th:bfR} 
Let $e \in V^\times$ and 
$S:= s_e =  \frac{1}{q(e)} S_{e,e}$.
Then 
$$
f: \bfC_q^R \to \bfC_q^S, \quad X \mapsto S X
$$
is a bijection. Moreover, $SX \circ SY \circ SZ = S XY^\sharp Z$. 
\end{theorem}

\begin{proof} 
Since  $S$ is bijective and  $S^2 = I$, by the general composition rules from
Theorem \ref{th:spiration1}, 
multiplication by $S$, 
induces a bijection $f$.
By Theorem \ref{th:involution1}, $SXS=X^\sharp$, whence the last claim. 
\end{proof}

\begin{example}
Let $\beta = 0$, $\alpha = 1 = \gamma$ (elliptic plane, $q(x)=x_1^2 + x_2^2$).
Then $\bfC_q^R$ is given by the usual formulae of ``complex numbers'', and
$\bfC_q^S$ by ``complex conjugations''. If $2$ is invertible in $\K$, then
 $M(2,2,\K) = \bfC_q^R \oplus \bfC_q^S$
 (just as in the real case, $M(2,2,\R) \cong \bC \oplus \overline \bC$).
But, e.g., when $\K=\Z$, then the sum remains direct, but
 the dihedral algebra is a proper subalgebra
of $M(2,2;\K)$.
\end{example}

\begin{example}
Let $\beta = 1$, $\alpha = 0 =\beta$ (hyperbolic plane, $q(x)=x_1x_2$).
Let $E_{ij}$ the usual elementary matrix. 
Then 
$R_{12}=E_{11}$, $R_{21} = E_{22}$, so $\bfC_q^R$ is the space of diagonal matrices,
and
$S_{12}=0$, $S_{11} = E_{12}$, $S_{22}= E_{21}$, so $\bfC_q^S$ is the space of
anti-diagonal matrices, and we
have
$M(2,2,\K)=\bfC_q^R \oplus \bfC_q^S$.
\end{example}

\begin{example}
Let $q(x) = x_1^2$, so $\alpha = 1$, $\beta=0=\gamma$.
Then $-R_{12}=E_{12} = S_{12}$ and $S_{11} = E_{11}-E_{22}$,
$S_{22}=0$.
In this case $\bfC_q^R \cap \bfC_q^S = \K E_{12}$.
\end{example}

\begin{example}
Consider the ``Eisenstein form'' $\alpha = \gamma = 1 = - \beta$.
We have the matrices $R_{11}= I$,
$$
R_{12} = \begin{pmatrix} - 1 & 1 \\ - 1 & 0 \end{pmatrix}, \,
R_{21} = \begin{pmatrix} 0 & - 1 \\ 1 & -1 \end{pmatrix}, \,
S_{11} = \begin{pmatrix} 1 & - 1 \\ 0 & - 1 \end{pmatrix}, \,
S_{12} = \begin{pmatrix} 0 & 1 \\ 1 & 0 \end{pmatrix}, \,
S_{22} = \begin{pmatrix} - 1 & 0 \\ - 1 & 1 \end{pmatrix},
$$
which form a group isomorphic to  $\s_3$. 
If $3$ is invertible in $\K$, then $(R_{12}, I, S_{11},S_{12})$ forms a basis
of $M(2,2;\K)$.
But, again, for $\K = \Z$, the dihedral algebra is a proper subalgebra of
$M(2,2,\K)$.
\end{example}

\begin{remark}
Here are some other remarkable formulae for the theory of $(\K^2,q)$.
Using the bilinear form $b(x,y) = \alpha x_1 y_1 + \beta x_1 y_2 + \gamma x_2 y_2$
 such that $b(x,x)=q(x)$,
 given by (\ref{eqn:binaryb}), the matrix $R_{x,y}$ 
given by the second sign of equality above can be written
$$
R_{x,y} = \begin{pmatrix} b(x,y) & \gamma [x,y] \\
\alpha [y,x] & b(y,x) \end{pmatrix}.
$$
According to Theorem \ref{th:spiration2}, its determinant
$b(x,y)b(y,x) + \alpha \gamma [x,y]^2$ equals $q(x) q(y)$.
The preceding formula can also be written 
$$
R_{x,y} = b(x,y) \id -  [x,y] R_{e_1,e_2}.
$$
Yet another way of writing it is 
$$
[u, \langle xyv \rangle ] = b(x,y) \, [u,v] + b(u,v) \, [y,x] .
$$
Indeed,
by (\ref{eqn:drei}) and uniqueness,
$ \langle xyv \rangle  = b(x,y)v - b(x,v) y + b(y,v) x$, whence
\begin{align*}
[u, \langle xy v \rangle]  &  = b(x,y) [u,v]  - b(x,v) [u,y] + b(y,v) [u,x]
\\
&=  b(x,y) [u,v]  +  b( [y,u] x + [u,x] y, v) 
\\
& =  b(x,y) [u,v]  - b( [x,y] u,v ) =  b(x,y) [u,v]  - [x,y] b(u,v )
\end{align*}
where we have used  the $3$-identity $[u,x]y+[x,y]u+[y,u]x=0$.
\end{remark}


\section{Group spherical  spaces}\label{sec:groupspherical}

\subsection{Basic definitions.}
We turn to the general case.
Recall from the introduction,
subsection \ref{sec:Gsp}:

\begin{Definition}\label{def:spherical}
Consider a {\em quadratic space} $(V,q)$ such that $V^\times$ is not empty. 
\begin{enumerate}
\item[(K)]
We  say that 
a $\K$-trilinear ``product'' map $V^3 \to V$, $(x,y,z) \mapsto \langle x\vert y \vert z \rangle$ (often  written shorter:
$\langle xyz \rangle$, if there is no danger of confusion)  is
 {\em $q$-compatible} if it satisfies the following
 {\em left and right Kirmse identities}: 
 $\forall x,y \in V$, 
$$
\langle x \vert x \vert y \rangle = q(x) y = \langle y \vert x \vert x \rangle .
$$
\item[(PA)]
We say that a trilinear map $V^3 \to V$, $(x,y,z) \mapsto \langle x\vert y \vert z \rangle$ is
{\em para-associative} if
it satisfies 
 Condition (AT2) from Appendix \ref{app:torsors}, i.e., 
$\forall a,b,c,d,e \in V$,
$$
\langle a b \langle cde \rangle \rangle = 
\langle a \langle dcb \rangle e \rangle =
\langle \langle abc \rangle de \rangle ,
$$
\item[(Com)]
 it is 
{\em commutative} if it is symmetric in the outer variables: $\forall x,y,z\in V$,
 $$
 \langle x y z  \rangle = 
\langle z y  x \rangle,
$$
\item[(TS)]
and 
{\em totally symmetric} if it is invariant under all permutations from ${\mathfrak S}_3$, i.e., it is commutative
and satisfies moreover,  $\forall x,y,z\in V$,
$$
\langle x y z \rangle = \langle yx z \rangle .
$$
\item[(TC)]
We say that a trilinear map $V^3 \to V$, $(x,y,z) \mapsto \langle x\vert y \vert z \rangle$ 
satisfies the {\em ternary composition law} if $\forall x,y,z \in V$,
$$
q( \langle xyz \rangle) = q(x) q(y) q(z).
$$
\end{enumerate} 
A {\em group spherical space} is a quadratic space 
together with a ternary product map satisfying (K), (PA) and (TC).
It is called {\em commutative (totally symmetric)} if the ternary product map has the
respective property.
Non para-associative versions will be considered later (Definition \ref{def:Moufang-spherical}).
\end{Definition}


\begin{Definition}
A {\em  morphism of group spherical spaces} is an  isometry $f$ such that 
$f(\langle xyz \rangle) = \langle f(x),f(y),f(z) \rangle$. 
\end{Definition}

It is clear that group spherical spaces with their morphisms form a category.
Note that the two conditions defining a morphism are not independent of each other:
if $f$ is compatible with trilinear products, then
$$
q(f(x)) f(e) = \langle f(x),f(x),f(e) \rangle = f(\langle xxe \rangle)  = q(x) f(e),
$$
so if $f(e) \in V^\times$, it follows that $f$
 automatically an isometry.
Conversely,  it is not always true that an isometry is automatically compatible with 
trilinear products (but this  is true if $\dim V = 2$ since then the quadratic form uniquely determines
the trilinear map).
Definitions are designed for the following to hold:

\begin{theorem}\label{th:torsortheorem}
In every group spherical space $(V,q)$, the set
$V^\times$ of invertible elements
with $(xyz) := \frac{\langle xyz \rangle}{q(y)}$ is a torsor.
Every sphere 
$S(c) = q^{1}(c)$ with $c \in \K^\times$ is a subtorsor
of $V^\times$, and so are, if $c \not= -c$, the sets
$$
D(c) := S(c) \sqcup S(-c).
$$
\end{theorem}

\begin{proof}
By the ternary composition law, 
\begin{equation}\label{eqn:qq}
q((xyz)) = \frac{ q(\langle xyz \rangle)}{q(y)^2} = q(x) q(y)^{-1} q(z).
\end{equation}
Together with para-associativity of $\langle --- \rangle$, we get para-associativity
of $(---)$:
$$
(ab(cde))=  \frac{1}{q(b) q(d) } \langle ab \langle cde \rangle \rangle =
\frac{1}{ q((dcb))} \langle a \frac{\langle dcb \rangle}{q(c)}  e \rangle
  = (a(dcb)e) = ((abc)de).
$$
Finally, (K) implies idempotency:
 $(xxy) = \frac{\langle xxy\rangle}{q(x)} = y = (yxx)$.
 Because of (\ref{eqn:qq}), $q(x)=q(y)=q(z)=c$ implies that
 $q((xyz)) = c$, so all spheres $S(c)$ are stable under the torsor law, and hence are
 subtorsors.
 Since the set $\{ c,-c \} \subset \K^\times$ is closed under the torsor law of $\K^\times$,
 if follows that $D(c)$ also is a subtorsor.
\end{proof}

Some, but not all, categorical notions take the usual form:

\begin{Definition}
A {\em subspace} of a group spherical space $(V,\langle --- \rangle)$ is a linear subspace
$E \subset V$ that is stable under the ternary product (a sub-triple system).

A subspace $E$ is called {\em direct} if there is a subspace $E'$ which is orthogonal
to $E$ for the bilinear form $b_q$ and such that $V = E \oplus E'$.

An {\em ideal} is a subspace $I$ such that
$\langle I VV \rangle + \langle VIV \rangle + \langle VVI \rangle \subset I$.
\end{Definition}

\nin
Clearly, a subspace is a spherical space in its own right, for the restriction $q\vert_E$.

\begin{lemma}
Every $1$- and every $2$-dimensional submodule of $V$ is a subspace.
\end{lemma}

\begin{proof}
Let $\dim E = 1$ and $e$ a basis of $E$. Then $\langle r e,se,te \rangle = rst \, q(e) e \in E$. 

Let $\dim E = 2$ and
 $u,v \in E$. It follows that $\langle uuv \rangle = q(u)v = \langle vuu\rangle \in E$ and
$\langle uvu \rangle = b_q(u,v)u - q(u)v \in E$. Taking a basis $e_1,e_2 \in E$, this implies as in
the proof of Theorem \ref{th:ternary} that $E$ is stable under the ternary product.
\end{proof}

\begin{lemma}
If an ideal $I$ contains invertible elements and $\K$ is a field, then  $I=V$. If the form $q$ is anisotropic, then $V$ is 
{\em simple} (it contains no non-trivial ideals).
\end{lemma}

\begin{proof}
Let $e \in I$ be invertible.
For all $v\in V$, $q(e) v = \langle eev \rangle \in I$, so $v \in I$, and $I=V$.
Thus, if $q$ is anisotropic, a non-zero ideal must be equal to $V$.
\end{proof}

We'll see that there is a notion of ``representation'', but
there is no general notion of ``direct sum'' or ``tensor product''.

\subsection{Inner operators: spirations and spiflections.}\label{sec:spirations} 

\begin{Definition}\label{def:spirations}
In a group spherical space, we define linear
{\em operators of left, right and middle multiplication} by
$\langle xyz\rangle = L_{x,y} (z) = R_{z,y}(x) = S_{x,z}(y)$.
Altogether,  we refer to them as {\em inner operators}. 
We also call them  {\em (left, right) spirations} (short from 
\href{https://en.wikipedia.org/wiki/Spiral_similarity}{\em spiral similarities}),
 and {\em symmetries  (spiflections)}. 
\end{Definition}

Every algebraic identity satisfied by the trilinear product can be rewritten, in various ways,
as an operator identity. For instance, commutativity means that 
$L_{x,y} = R_{x,y}$ (short: $L=R$), and Kirmse (K) can be written
\begin{equation}
L_{x,x} = q(x) \, \id_V  = R_{x,x} .
\end{equation}
Para-associativity  implies that 
 {\em the composition of two left (right) spirations is a left 
(right) spiration}. It follows that  the linear span of all
$L_{x,y}$ with $x,y \in V$ is a subalgebra of $\End(V)$, and likewise for $R$.

\begin{Definition}
We take up Def.\ \ref{def:dihedral}:
the {\em left  (right) spiration algebra} of a group spherical  space $(V,q)$
 is the algebra generated by all $L_{x,y}$, resp.\ by all $R_{x,y}$, $x,y \in V$, 
$$
\bfC_q^L = \{  \sum_i  \lambda_i  L_{x_i,y_i}  \mid  \mu, \lambda_i \in \K, x_i,y_i \in V \} ,
\quad
\bfC_q^R =  \{  \sum_i  \lambda_i  R_{x_i,y_i}  \mid  \mu, \lambda_i \in \K, x_i,y_i \in V \} . 
$$
\end{Definition}


\begin{theorem} \label{th:spiration1} 
For every group spherical space $(V,q)$, the inner operators satisfy:
\begin{enumerate}
\item (right composed with right gives right): 
$R_{\langle x_3 x_4 x_5\rangle , x_2} = R_{x_5 , \langle x_4 x_3 x_2\rangle} = R_{x_5 x_4} R_{x_3 x_2}$

(left composed with left gives left):
$L_{x_1 x_2} L_{x_3 x_4} = L_{x_1,\langle x_4,x_3,x_2\rangle} = L_{\langle x_1 x_2 x_3\rangle ,x_4 }$,
\item
(middle with left or right gives middle):
$S_{x_1, \langle x_3 x_4 x_5 \rangle} = S_{x_1 x_5} L_{x_4 x_3} = R_{x_5 x_4} S_{x_1 x_3}$

$L_{x_1 x_2} S_{x_3 x_5} = S_{x_1 x_5 } R_{x_2 x_3} = S_{\langle x_1 x_2 x_3\rangle , x_5}$

\item
(left and right commute):
$L_{x_1 x_2} R_{x_5 x_4} = S_{x_1 x_5} S_{x_4 x_2} = R_{x_5 x_4} L_{x_1 x_2}$
\end{enumerate}
\nin and the following:
 for all $a,b,c,\ldots \in V$, 
\begin{align*} 
S_{a,b} \circ S_{u,v}\circ S_{y,z} & = 
S_{\langle avy \rangle, \langle zub \rangle} 
\\
R_{a,b}^2  = R_{\langle aba \rangle, b}  & =
b_q(a,b) R_{a,b} - q(a) q(b) \id 
\\
R_{a,b} \circ R_{b,c} & = q(b) R_{a,c}
\\
R_{a,b} \circ R_{b,a} &=  q(a) q(b) \id
\\
R_{a,b} + R_{b,a} & = b_q(a,b) \id 
\\
S_{a,b} \circ S_{b,a}  &  =  q(a) q(b) \id ,
\end{align*}
and the endomorphisms $R_{a,b}$, resp.\ $S_{a,b}$ are invertible if, and only if,
$q(a) q(b) \in \K^\times$, and then
$$
R_{a,b}^{-1} = \frac{1}{q(a) q(b)} R_{b,a},\qquad 
S_{a,b}^{-1} = \frac{1}{q(a) q(b)} S_{b,a} . 
$$
If $(V,q)$ is commutative, then the left and right spiration algebras are both commutative
and agree with each other.
\end{theorem}

\begin{proof}
The composition rules are just a way of rewriting the identity of 
para-associativity (identity (AT2)). Similarly,
\begin{align*}
S_{a,b} \circ S_{u,v}\circ S_{y,z} (x) &= 
\langle a \langle u \langle yxz \rangle v \rangle b \rangle = 
\langle \langle avy\rangle x \langle zub\rangle \rangle = 
S_{\langle avx \rangle, \langle zub \rangle} (x) .
\end{align*}
Polarizing $R_{a,a} = q(a) \id$, we get 
$R_{a,b} + R_{b,a}  = b_q(a,b) \id$.
Next, the relations for $R_{a,b}^2$ and $S_{a,b}S_{b,a}$ are proved as in
Theorem \ref{th:spiration2} . 
 By the equalities just proved, the condition for invertibility
is both necessary and sufficient. 

If $(V,q)$ is commutative, then $L_{a,b}=R_{a,b}$. But left and right spirations always commute, so in this case
right spirations commute among each other, and so do left spirations.
\end{proof}

\begin{remark}
Left- and right spirations $X$ always satisfy a quadratic relation
$X^2 + \lambda X + \mu =0$, but spiflections in general do not.
However,
when $V$ is commutative, then $S_{ab}=S_{ba}$, so
 $S_{a,b}^2 = S_{a,b} S_{b,a} = q(a) q(b) \id$.
\end{remark}

\begin{theorem}\label{th:involution1}
Assume $e \in V$ is an invertible element. Then, for all $a,b\in V$,
$$
S_{e,e} \circ L_{a,b} \circ S_{e,e} = R_{S_{e,e}(a),S_{e,e}(b)} .
\leqno{\rm (A)}
$$
If $V$ is commutative, 
$$
S_{e,e} \circ L_{a,b} \circ S_{e,e} =  q(e)^2 L_{b,a}  ,
\leqno{\rm (B)}
$$
so conjugation by $S_{e,e}$ is an automorphism 
 $\sharp: \bfC_q \to \bfC_q$ sending
$L_{a,b}$ to $L_{b,a}$.
\end{theorem}

\begin{proof}
The left hand side of (A), applied to an element $x$, gives
$q(e) \langle \langle xeb\rangle ae \rangle$.

The right hand side gives
$q(e) \langle xe \langle bae \rangle \rangle$.
Both agree.  Now cancel out the invertible factor $q(e)$.

If $V$ is commutative, 
$R_{S_{e,e}(a),S_{e,e}(b)} (x) = \langle \langle eae \rangle , \langle ebe \rangle , x \rangle =
q(e)^2 \langle bax \rangle$, whence (B).
Since $S_{e,e}^2 = q(e)^2 \id$, conjugation by $S_{e,e}$ thus is an (inner) automorphism
sending $L_{a,b}$ to $L_{b,a}$.
\end{proof}

\begin{remark}
When $V = \K^2$, then conjugation by $S_{e,e}$ is the adjugate map
(which can be used to define an
automorphism $\sharp$  even if $q$ has no invertible elements).
\end{remark}

\subsection{Link binary-ternary.}
Just like a group is obtained by specifying an arbitrary base point in a torsor (Theorem \ref {th:Torsors}),
we get a binary structure by specifying a base point in a group spherical space, and
we can go back and forth.
The following definition is more general than usual  since we drop 
the  assumption that the norm $N$  be
{\em non-degenerate} (see \cite{McC}, p.156, \cite{Fa}, p.54,
\cite{KMRT}, p. 456), so maybe one should speak of ``generalized composition algebras''
(but we will omit the word ``generalized''):

\begin{Definition}\label{def:compalg}
An involution of a unital $\K$-algebra $(\bA,1)$ is an anti-automorphism $x \mapsto x^\sharp$ of order $2$.
It is called a {\em central involution} if, for all $x \in \bA$,
the elements $x+x^\sharp$ and $xx^\sharp$ belong to the {\em center} of $\bA$
(set of elements that commute and associate with all other elements), and it is called
a {\em scalar involution} if for all $x \in \bA$, 
$$
x + x^\sharp \in \K 1, \qquad xx^\sharp \in \K 1.
$$
Every scalar involution is central.
A (generalized) {\em composition algebra (over $\K$)} is a unital algebra together with
a scalar involution.
 \end{Definition}

\begin{lemma}\label{la:compalg} 
In any (generalized) composition algebra $\bA$, the map $N:\bA \to \K$, defined by
$$
x x^\sharp = N(x) 1
$$
 is a quadratic form, called the {\em norm} of $\bA$, and we have
 $x+x^\sharp = b_n(1,x)1$.
If $\bA$ is  moreover  {\em associative}, then $V =\bA$, for any
$\lambda \in \K^\times$, with trilinear map
$$
\langle xyz \rangle = \lambda \,  x y^\sharp z 
$$
is a group spherical  space over $\K$, with quadratic form $q=\lambda N$. 
\end{lemma}

\begin{proof}
Let $a+a^\sharp = t(a) 1$ with $t(a) \in \K$.
Since $a^\sharp =  t(a) - a$ commutes with $a$, we have
$a a^\sharp = a^\sharp a = N(a) 1$.
Clearly $N(ra)=r^2 N(a)$ for all $r\in \K$, and
$$
(N(a+b) - N(a) - N(b) )1 = (a+b)(a+b)^\sharp - aa^\sharp - bb^\sharp =
ab^\sharp + ba^\sharp
$$
is $\K$-bilinear, so $N$ is a quadratic form, and moreover
$b_n(a,1)=a+a^\sharp = t(a) 1$ (``trace form''). 
We have the ``Cayley-Hamilton equation''
\begin{equation}\label{eqn:CH}
a^2 - t(a) a + N(a)1  = a (a - t(a)) + a a^\sharp = - aa^\sharp + aa^\sharp = 0.
\end{equation}
It follows that $\langle xxy \rangle = N(x)y = \langle yxx\rangle$,
so the trilinear product is $q$-compatible.
Since the binary product is associative by assumption,
 and $\sharp$ is an involution, the para-associative law
(AT2) holds:
$$
\langle a b \langle cde \rangle \rangle = 
ab^\sharp c d^\sharp e =
a (d c^\sharp b)^\sharp e =
\langle a \langle dcb \rangle e \rangle 
$$
and moreover, using again  associativity, 
\begin{equation}
N(ab) = (ab)^\sharp (ab) = (b^\sharp a^\sharp) (ab) = b^\sharp (a^\sharp a) b = N(a) b^\sharp b =
N(a) N(b)
\end{equation}
whence the ternary composition rule.
\end{proof}

\nin
The lemma has a converse, whose most subtle issues are related to
{\em normalization}:
the neutral element $1$ of a composition algebra must satisfy $N(1)=1$,
so norm and unit element depend on each other, whereas the notion of group
spherical space is invariant under scaling of forms.

\begin{theorem}\label{th:binary-ternary}
There is a bijection between associative (generalized) composition algebras and
group spherical  spaces, up to choice of base point. More precisely, one
direction is given by the preceding lemma. For the other direction,
let $(V,q)$ be a group spherical  space, and
fix  $e \in V^\times$. Then
\begin{enumerate}
\item
the bilinear product $xz:= x \cdot_e z = \frac{\langle xez \rangle}{q(e)}$ turns $V$ into an associative algebra with
neutral element $e$, and with involution $\sharp = s_e = \frac{S_{e,e}}{q(e)} : V \to V$,
$$
x^\sharp = s_e(x)= 
 \frac{ \langle exe \rangle}{q(e)} = \frac{ b_q(x,e)}{q(e)} e - x ,
$$
and
this algebra is a {\em composition algebra} with respect to the {\em norm}
$N(x):= \frac{q(x)}{q(e)}$.
\item
The construction from Item 1 is inverse to the construction from Lemma 
\ref{la:compalg}:
$$
\langle xyz \rangle = q(e) x \cdot y^\sharp \cdot z .
$$
\item
The algebra $(V,\cdot_e)$ is isomorphic to $\bfC_q^L$, via

$V \to \bfC_q^L$, $v \mapsto L(v):= \frac{1}{q(e)}L_{v,e}$, with inverse
$\bfC_q^L \to V$, $f \mapsto f(e)$.
\item
The algebra $(V,\cdot_e)$ is anti-isomorphic to $\bfC_q^R$, via

$V \to \bfC_q^R$, $v \mapsto R(v):= \frac{1}{q(e)}R_{v,e}$, with inverse
$\bfC_q^R \to V$, $f \mapsto f(e)$.
\item
The ternary composition law follows already from the Kirmse laws and para-associativity:
[(K) $\land$ (PA) ] $\Rightarrow$ (TC)
\item
Assume $\K$ is a field. Then
$q$ is anisotropic iff $(V,\cdot,e)$ is a skew-field.
\end{enumerate}
\end{theorem}

\begin{proof}
1.
Associativity of the binary product follows directly from para-associativity of the ternary one.
 The element $e$ is neutral:
$x \cdot_e e  = \frac{\langle xee \rangle}{q(e)} = \frac{q(e)x}{q(e)} = x = e \cdot_e x$.
Let us show that $\sharp$ is a scalar involution.
By Lemma \ref{lemma:Q}, $\sharp = \sigma_e$ is an isometry of order $2$ fixing $e$.
It is an anti-automorphism: apply para-associativity twice to get
$$
y^\sharp \cdot x^\sharp =
\frac{ \langle \langle eye \rangle e \langle exe \rangle \rangle}{q(e)^3}
 = \frac{ \langle e \langle xey \rangle e\rangle}{q(e)^2} =
(xy)^\sharp .
$$
Finally, it is a scalar involution, since
$$
x+x^\sharp  = \frac{ b_q(x,e)}{q(e)} e \in \K e, \qquad
x x^\sharp = \frac{ \langle xe \langle exe \rangle \rangle}{q(e)^2} =
\frac{q(x) q(e) e}{q(e)^2} = N(x) e \in \K e.
$$

2. 
Start with a spherical quadratic space, fix the element $e$, and compute the
ternary map (with $\lambda = q(e)$)
associated to the new binary algebra:
$$
\lambda (xy^\sharp) z = \lambda  \frac{
\langle \langle x e \langle eye \rangle \rangle e z \rangle}
{q(e)^3} =\lambda  \frac{ \langle xyz \rangle}{ q(e)} = \langle xyz \rangle .
$$
Conversely, starting with a binary composition algebra with unit $1$, 
the binary product at $1$ is obviously
$\lambda \langle x1 y \rangle = \lambda xy$, the product we started with
(up to a scaling factor $\lambda$).

3.
Both maps are inverse of each other:  $L(v)e=v$ and $L(L_{a,b}(e))(x) = L_{a,b}(x) $,
and $L$ is a morphism:
$$
L(v \cdot_e w) = \frac{1}{q(e)^2} L_{\langle vew\rangle , e} 
=  \frac{1}{q(e)} L_{v,e} \circ \frac{1}{q(e)} L_{w,e} = 
L(v) \circ L(w) .
$$

4. As for 3., but now
$R_{\langle vew \rangle,e} = R_{w,e}\circ R_{v,e}$, 
so
$R(v \cdot_e w) = R(w) \circ R(v)
$.

5.
The proof of Item 2 did not require the assumption that $\langle --- \rangle$ satisfies
the ternary composition law.
Therefore, assuming (K) and (PA), we can use the bijection established in Item 3.
Now, the proof of Lemma \ref{la:compalg} shows that (TC) follows from
the properties of an associative composition algebra.

6.
This is immediate from the definitions:
a form over a field  is anisotropic iff $V^\times = V \setminus \{ 0 \}$.
\end{proof}

\begin{remark}
In general, 
the anti-isomorphism $\bfC_q^L \cong \bfC_q^R$ resulting from 3. and 4. depends on the choice of the
base point $e$. Only when $V$ is commutative it is independent of such choice, given by Theorem
\ref{th:involution1}.
\end{remark}

\subsection{Polarized spaces, and associative pairs.}\label{sec:polarized}
A quadratic space $(V,q)$ is called {\em polarized} if $V = V_1 \oplus V_2$ carries a direct sum
decomposition into
{\em totally isotropic} subspaces $V_i$ of $q$ (we won't assume that $q$ is non-degenerate).
We write $(x,y)=(x,0)+(0,y)$.
Then we have a bilinear form $b:V_1 \times V_2 \to \K$ given by
\begin{equation}
 b(x,y) := q((x,y)) = 
 q((x,0)+(0,y)) = q((x,y)) - q((x,0)) - q((0,y)) .
\end{equation}
The maps $b$ and $q$ are the same, but in the definition of $b$, the
 domain is considered as direct product of modules. 

\begin{Definition}
A {\em left polarized group spherical space} is a group spherical space that is a
 polarized  quadratic space $(V,q)=(V_1\oplus V_2,b)$,
 and
such that the trilinear map $\langle xyz \rangle$ satisfies the following two conditions:
\begin{enumerate}
\item
$\langle V  V V_i \rangle \subset V_i$ (meaning that
 $V_i$ for $i = 1,2$ are {\em left ideals}),
 \item
$\langle V_i V_i V \rangle = 0 $, i.e.,  $L_{V_i,V_i}=0$
(the left ideals operate trivially on $V$).
\end{enumerate}

\nin
In other terms,  the trilinear product $V^3 \to V$ has only the following four types of possibly
non-zero restriction:

(A)
$\qquad V_1 \times V_2 \times V_1 \to V_1$, and
$V_2 \times V_1 \times V_2 \to V_2$, and

(B) $\qquad V_1 \times V_2 \times V_2 \to V_2$, and
$V_2 \times V_1 \times V_1 \to V_1$.

\nin
{\em Right} polarized group spherical spaces are defined in the same way, using a decomposition into
totally isotropic {\em right} ideals $V_i$, such that $R_{V_i,V_i}=0$.
\end{Definition}

\begin{example}\label{ex:Hyp}
Let $V=\K \oplus \K$ with $q(x)=x_1 x_2$, so $V_1 = \K \oplus 0$, $V_2 = 0 \oplus \K$, and
 $b(x,y)=xy$. Recall from 
Example \ref{ex:hyp}, 
eqn. (\ref{eqn:hyperbolic}) that
$\langle xyz \rangle = \begin{pmatrix} x_1 y_2 z_1 \\ x_2 y_1 z_2 \end{pmatrix}$.
Clearly, conditions 1. and 2. are satisfied for $\langle xyz \rangle$.
 Moreover, restrictions of Type (B) are zero, and those of type (A) correspond to the usual product
 $rst$ of scalars. 
\end{example}

\begin{example}\label{ex:M1}
We decompose  $V = M(2,2;\K)$ into its two ``canonical left ideals''
(with $E_{ij}$: elementary matrix):
$$
V_1 = \K E_{11} \oplus \K E_{21} =
\{ \begin{pmatrix} x_1 & 0 \\ x_2 & 0 \end{pmatrix} \mid x_1,x_2 \in \K\},
$$
 (matrices with second column zero)
 and $V_2 = \K E_{22} \oplus \K E_{12}$ (matrices with first column zero).
The $V_i$ are totally isotropic for $q(X)=\det(X)$, and $V = V_1 \oplus V _2$, 
$$
q( x_1 E_{11} + x_2 E_{21} + y_1 E_{12}+y_2 E_{22}) =
x_1 y_2 - x_2 y_1 = [x,y] 
$$
is the canonical duality between two copies of $\K^2$. 
The trilinear product is given by
$\langle XYZ \rangle = XY^\sharp Z$, and
 $V_1$ and $V_2$ are {\em left} ideals, whence Property 1.
(If we decompose according to rows, we would get a decomposition into right ideals.)
We check 2.:  for $X,Y \in V_1$,
$$
XY^\sharp = \begin{pmatrix} x_1 & 0 \\ x_2 & 0\end{pmatrix}
\begin{pmatrix} 0 & 0 \\ -y_2 & y_1 \end{pmatrix} = 0,
$$
 and likewise for $X,Y \in V_2$.
(Computations are continued in Example \ref{ex:M2}.)
\end{example}

\begin{Definition}
(Cf.\ \cite{Lo75}, II.6.15).)
An {\em associative pair} is a pair $(\bA^+,\bA^-)$ of $\K$-modules together with two
trilinear maps
$\bA^\pm \times \bA^\mp \times \bA^\pm \to \bA^\pm$, $(x,y,z) \mapsto
\langle xyz \rangle_\pm$, such that the {\em para-associative law} is satisfied:
$$
\forall x, z, v \in \bA^\pm, u,y \in \bA^\mp : 
\qquad
\langle \langle xyz \rangle_\pm uv \rangle_\pm =
\langle x \langle u z y \rangle_\mp  v \rangle_\pm =
\langle xy \langle zuv \rangle_\pm  \rangle_\pm  .
$$
\end{Definition}

\begin{example}
Let $b:U \times V \to \K$ be bilinear; then $(U,V)$ is an associative pair, with
$\langle xyz \rangle_+ = b(z,y) x$ and
$\langle uvw \rangle_- = b(v,w)u$.
Proof by direct check:
$$
\langle \langle xyz \rangle_+  uv \rangle_+ =
b(z,y) b(x,u)v = b(z,y) \langle xuv\rangle_+ = 
\langle x \langle u z y \rangle_-  v \rangle_+ .
$$
\end{example}

\begin{lemma}\label{la:typeA}
The two restrictions of Type {\rm (A)} in a left polarized group spherical space form an associative pair
$(V_1,V_2)$, of the form given by the preceding example. 
\end{lemma}

\begin{proof} 
Let $x_i \in V_i$, $x=x_1+x_2$, and $z \in V_1$. The right Kirmse identity gives
$$
\langle z  x x \rangle = q(x)z = b(x_1,x_2) z .
$$
On the other hand,
$\langle z,x_1+x_2,x_2+x_2 \rangle = \langle zx_1 x_2 \rangle + \langle z x_2 x_1 \rangle =
\langle z x_2 x_1 \rangle$, whence
$\langle z uv \rangle_+ = b(v,u )z$.
In the same way, with $z \in V_2$, we get
$\langle z v u \rangle_- = b(v,u)z$, so $(V_1,V_2)$ with restriction of type (A) forms indeed an
associative pair of the form given in the example.
\end{proof}

\begin{theorem}\label{th:polarized1}
Let $(V,q)=(V_1\oplus V_2,b)$ be a left polarized group spherical space.
Then the ternary map of $V$ is uniquely determined by the bilinear form $b$, and it is given by,
writing $x=x_1 + x_2=\begin{pmatrix}x_1\\ x_2 \end{pmatrix}$, $x_i \in V_i$, etc.,
$$
\langle xyz \rangle =
\begin{pmatrix}
b(z_1 ,y_2) x_1 
- b(z_1,x_2) y_1 + b(y_1, x_2) z_1
\\
 b(y_1,z_2) x_2  - b(x_1,z_2) y_2
 + b(x_1,y_2) z_2
\end{pmatrix} .
$$
Conversely, for every bilinear form $b:V_1 \times V_2 \to \K$, we define a trilinear map on 
$V = V_1 \oplus V_2$ by the preceding formula. Then:
\begin{enumerate}
\item  This trilinear map satisfies the left and
right Kirmse identities.
\item
Para-associativity of the trilinear product is equivalent to the following property 
{\rm (BA)} of $b$:
for all $x^1,x^2 \in V_1$ and all $y^1,y^2,y^3 \in V_2$,
\begin{align*}
0  = &  \quad \,  b(x^1 , y^1) b(x^2,y^2) y^3 - b(x^1,y^2) b(x^2,y^1) y^3 \\
&  + b(x^1, y^3) b(x^2,y^1)  y^2 - b(x^1, y^1) b(x^2,y^3) y^2 \\
& +  b(x^1,y^3) b(x^2,y^2) y^1 - b(x^1,y^2) b(x^2, y^3) y^1 ,
\end{align*}
and similarly for r\^oles of $V_1$ and $V_2$ interchanged.
Condition {\rm (BA)}  can also be written, for $\{ i,j \} = \{ 1,2 \}$
$$
\forall y^1,y^2 \in V_i, \forall  z^1,z^2,z^3 \in V_j :
\quad
\sum_{\sigma \in \s_3} {\mathrm{sgn}} (\sigma) \cdot  b(y^1,z^{\sigma(1)}) \,   b(y^2,z^{\sigma(2)}) \, z^{\sigma(3)}  = 0 .
$$
\end{enumerate}
\end{theorem}

\begin{proof}
Assuming the Kirmse identities, we get
 for $x = x_1 + x_2$, $z \in V_1$:
$$
\langle x x z \rangle = q(x) z = b(x_1,x_2)z,
$$
$\langle x_1 x_2 z \rangle + \langle x_2 x_1 z \rangle =
b(z,x_2) x_1 + \langle x_2 x_1 z \rangle$, whence
$$
\langle u_2 x_1 z_1 \rangle = b(x_1,u_2) z_1 - b(z_1,u_2) x_1
$$
(In particular, for $x_1 = z_1$ we see that this is alternating).
Using this, we expand
$\langle xyz \rangle =
\langle x_1 + x_2,y_1 + y_2,z_1 + z_2 \rangle$ and get the formula announced in the theorem.

Conversely, for any bilinear $b$, letting $x=y$ or $y=z$, we get the Kirmse identities.

Para-associativity  holds for all $5$-tuples  iff it holds for all $5$-tuples of homogeneous elements.
For tuples  of parity type $12121$ or $21212$ (restriction of type (A))
it holds without using the condition on $b$ (by Lemma \ref{la:typeA}).
By the same type of computation, for tuples of type
$\langle 21 \langle 122 \rangle \rangle$, para-associativity holds without using the condition on $b$.
For the type
$\langle 12 \langle 122\rangle \rangle$, para-associativity is equivalent to the condition from the
theorem, by direct computation. 
\end{proof}

\begin{theorem}\label{th:123}
We use notation from the preceding theorem.
\begin{enumerate}
\item
If $b$ is of rank $1$, then Condition {\rm (BA)} holds.
\item
If $b$ is of rank greater than $2$, then Condition {\rm (BA)} does not hold.
\item
Assume $b$ is of rank $2$ and non-degenerate.
Then Condition {\rm (BA)} holds.
\end{enumerate}
\end{theorem}

\begin{proof}
1. 
If $b$ is of rank $\leq 1$, then there exist linear forms $\phi_i:V_i \to \K$
such that $b(x,y)=\phi_1(x) \phi_2(y) $, and then each of the three individual terms
(each line of the displayed equation) is zero, so (BA) holds.

2. 
Assume the rank of $b$ is greater or equal than $3$.
Then there exist $y^i$ and $z^j$, $i,i=1,2,3$, such that
$\det(b(y^i,z^j)_{i,j=1,2,3}) \not= 0$.
Therefore the right-hand side of (BA) must be
non-zero, and thus the trilinear product is not para-associative.

3. 
Assume $b$ of rank $2$ and non-degenerate, so
 $V_1$ and $V_2$ are two-dimensional.
In this case, two different proofs are possible:

(a)
Choosing appropriate bases in $V_1$ and $V_2$, the form 
$b$ is given by $b(x,y)=x_1 y_2 - x_2 y_1 =[x,y]$.
Then the formulae reduce to those from Example \ref{ex:M2} below,
describing the matrix algebra $M(2,2;\K)$ which clearly is para-associative.

(b)
A different proof is given by noticing that
 $\det(b(y^i,z^j)_{i,j=1,2,3}) = 0$ for all choices of $y^i,z^j$,
$i,j=1,2,3$, since $b$ is of rank $2$.
Since  $b$ is non-degenerate, this implies that
the expression from (BA)  is zero, for 
all choices of $y^i,z^j$,
$i=1,2,3$, $j=1,2$, whence para-associativity.
\end{proof}

\begin{example}\label{ex:M2}
Assume $V_1,V_2$ two-dimensional and
$b:V_1 \times V_2$ of rank $2$, given by
$b(x,y) = x_1 y_2 - x_2 y_1 =[x,y]$.
Using the identity $[x,y]z+[y,z]x+[z,x]y=0$ (Remark \ref{rk:dreier}), the formula for the
trilinear product on $V = V_1 \oplus V_2$ reduces to
\begin{equation}\label{eqn:*}
\langle xyz \rangle =
\begin{pmatrix} [z_1,y_2] x_1 + [z_1,y_1] x_2 \\
[z_2,y_1] x_2 + [z_2,y_2] x_1 \end{pmatrix} .
\end{equation}
This is indeed the formula for the triple matrix product  $XY^\sharp Z$ in $M(2,2,\K)$.
Indeed, with notation from Exemple \ref{ex:M1}, we 
 compute restrictions of Type (A), $V_1 \times V_2 \times V_1 \to V_1$: 
\begin{equation} \label{eqn:Pair!}
 \begin{pmatrix} x_1 & 0 \\ x_2 & 0\end{pmatrix}
  \begin{pmatrix}  0 & y_1 \\ 0 & y_2 \end{pmatrix}^\sharp
   \begin{pmatrix} z_1 & 0 \\ z_2 & 0\end{pmatrix}
   = (y_2 z_1 - y_1 z_2)   \begin{pmatrix}  x_1 & 0  \\  x_2 & 0  \end{pmatrix} =
   [z,y]   \begin{pmatrix}  x_1 & 0  \\  x_2 & 0  \end{pmatrix} ,
\end{equation}
and of Type (B), $V_1 \times V_2 \times V_2 \to V_2$:
\begin{equation} \label{eqn:Pair!!}
 \begin{pmatrix} x_1 & 0 \\ x_2 & 0\end{pmatrix}
  \begin{pmatrix}  0 & y_1 \\ 0 & y_2 \end{pmatrix}^\sharp
   \begin{pmatrix} 0 & z_1  \\ 0 & z_2 \end{pmatrix}
   = (y_2 z_1 - y_1 z_2)   \begin{pmatrix}  0& x_1   \\  0 & x_2   \end{pmatrix} =
   [z,y]   \begin{pmatrix}  0& x_1  \\  0& x_2   \end{pmatrix} .
\end{equation}
The ternary matrix product can be decomposed in four components: 
write $X = X_1 + X_2$ with $X_i \in V_i$. Then
$$
XY^\sharp Z =
(X_1 + X_2)(Y_1+Y_2)^\sharp(Z_1 + Z_2)= 
( X_1 Y_2^\sharp Z_1 + X_2 Y_1^\sharp Z_1) +
( X_2 Y_1^\sharp Z_2 + X_1 Y_2^\sharp Z_2).
$$
$$
= [Z_1,Y_2] X_1 + [Z_2,Y_1] X_2 +
[Z_2, Y_2] (X_1)_2 +
[Z_1,Y_1](X_2)_1 ,
$$
and so we recover (\ref{eqn:*}). 
As mentioned above, this concludes Proof (a). 
\end{example}

\begin{remark}
What can one say in the remaining case ($b$ of rank $2$, but degenerate, so $V_1$ or
$V_2$ is of dimension bigger than $2$) ?
One may conjecture that Condition (BA) still holds -- see examples to be given next.
\end{remark}

\section{New examples:  representations, and split null extensions}\label{sec:representations}

\subsection{Representations.}
Following general ideas due to 
Eilenberg and Jacobson, O.\ Loos develops in \cite{Lo75} his theory of
Jordan Pairs by using {\em representations} as a tool.
A {\em general representation of an algebraic structure} $A$ should not be thought of
as some ``morphism into a matrix object'', but rather as a ``vector bundle over $A$
in a suitable category'', see \cite{BeD} where this geometric approach has been developed and 
applied to  Lie triple systems. 
The basic idea is already visible in the following:

\begin{example}\label{ex:C}
Let $V = \bC$ with its usual norm and binary product, and $W$ any complex vector space,
which we write as right $\bC$-module.
Then $\bC \oplus W$ with binary product and (degenerate) quadratic form $\widetilde q$,
$$
(v,w) \cdot (v',w') = (vv', w' v + w \overline{v'}), \quad \widetilde q ((v,w )) = q(v) = v \, \overline v
$$
is again a (generalized!) composition algebra: 
indeed, the product is associative:
\begin{align*}
((v,w) (v',w'))(v'',w'') &= (vv', w' v + w \overline{v'} )(v'',w'') =
(vv'v'', w'' vv' + (w' v + w \overline{v'} ) \overline{v''})
\\
&=  (vv'v'', w'' vv' + w' v \overline{v''} + w \overline{v'}  \overline{v''}) \\
&= (v,w) (v'v'', w''v' + w' \overline{v''}) = 
(v,w) ((v',w')(v'',w'') ) 
\end{align*}
(Note that commutativity of $\bC$ is heavily used.)
With
$(v,w)^\sharp = (\overline v, - w)$ we have
$$
(v,w) (v,w)^\sharp = (v \, \overline v , 0) = \widetilde q(v,w) 1 , \qquad
(v,w) + (v,w)^\sharp = (v +\overline v, 0) \in \R 1.
$$ 
Therefore we have a real group spherical space of real dimension $2 + 2 \dim_\bC W$.
It is not commutative, unless $W = 0$.
\end{example}

\begin{Definition}\label{def:rep}
Let $(V,q)$ be a group spherical  space, with right spiration algebra
$\bfC_q^R$.
A {\em right $(V,q)$-module} is simply a  $\bfC_q^R$-module $W$, i.e., a linear (left) action
of the associative algebra $\bfC_q^R$ on $V$.

The {\em split null extension of $V=:V_0$ by the right $(V_0,q)$-module $W=V_1$} is the linear space
$\widetilde V:= V_0 \oplus V_1$, together with the bilinear form
$\widetilde q:= q \oplus 0$ (i.e.,
$\widetilde q((x_0,x_1)) = q(x_0)$), and the trilinear map given by
$$
\langle (x_0,x_1) \vert (y_0,y_1)\vert (z_0,z_1) \rangle :=
\Bigl(
\langle x_0 y_0 z_0 \rangle,
R_{y_0,z_0} x_1 - R_{x_0,z_0} y_1 + R_{x_0,y_0} z_1
 \Bigr) .
$$
\end{Definition}

\begin{theorem}\label{th:splitnull}
For every {\em commutative} group spherical space $(V,q)$, and
right $(V,q)$-module $W$, the split null extension $V \oplus W$ is again a
group spherical  space.
It is not commutative, unless 

-- either, $V$ is totally symmetric, so $V = \K$, and $W$ any $\K$-module,

-- or $V$ is commutative but not totally symmetric, and
$W = 0$.
\end{theorem}

\begin{proof}
Left and Right Kirmse (K) follow by direct but slightly tricky computation, where commutativity
of $(V,q)$ is used in essential way. 
A direct check of para-associativity would be rather lengthy, therefore we shall
use the detour via the binary product, as in Example \ref{ex:C}:
fix $e \in V^\times$, without loss we may assume $q(e)=1$, so $q=N$, and
$\langle abc \rangle = a b^\sharp c = c b^\sharp a$, by commutativity, and the binary
product in $V \oplus W$ at $(y_0,y_1)=(e,0)$ is
$$
(x_0,x_1) \cdot (z_0,z_1) = ( x_0 z_0, x_1 z_0^\sharp + z_1 x_0) .
$$
As in the example, one checks by direct computation that this product is again associative,
and that $\sharp$ defined by
$(y_0,y_1)^\sharp = (y_0^\sharp, - y_1)$
is an anti-automorphism of order $2$ which is a scalar involution,
with associated norm $\widetilde q$.
Therefore we have again a group  spherical space $(V \oplus W,\tilde q)$.
Computing its ternary product
$(x_1,x_2) (y_1,y_2)^\sharp (z_1,z_2)$, we get the expression of the trilinear map from
Definition \ref{def:rep}.

Finally, if $V$ is totally symmetric, then necessarily
$V = \K$ and $\langle xyz \rangle = \lambda xyz$ for some $\lambda \in \K^\times$,
with trivial involution $\sharp$. Therefore the new product is again commutative
(but the new involution won't be trivial, unless $W=0$).
If $V$ is not totally symmetric, then $\sharp$ is non-trivial, and the new product
is not commutative (unless $W=0$).
\end{proof}

\subsection{The extended Minkowski plane.}
Let us consider split null extensions of the plane described in Example
\ref{ex:product}: $V_0 = \K^2$, 
\begin{equation}\label{eqn:Minko}
\langle x \vert y \vert z \rangle  = \phi(x) \psi(y) z + \psi (z) \phi(y) x - \phi(x) \psi(z) y.
\end{equation}
When $\phi$ and $\psi$ are linearly independent, then this is a hyperbolic plane,
with binary algebra isomorphic to a direct product $\K \times \K$.
The two projections of this algebra define one-dimensional representations, 
and using these we can define representations of any dimension, thus defining
split null extensions of any desired dimension. Remarkably, they
can all be defined by (\ref{eqn:Minko}):

\begin{theorem}\label{th:Minko}
Let $V$ by a $\K$-module, $\phi $ and $\psi$ be two linear forms $V \to \K$, and
$q(x) = \phi(x) \psi(x)$.
Then the trilinear product on $V$ defined by {\rm (\ref{eqn:Minko})} 
defines on $(V,q)$ the structure of a group spherical space.

It  is isomorphic to the split null extension by
$V_1 :=\ker(\phi) \cap \ker(\psi)$ of some complementary subspace
$V_0$ of $V_1$ in $V$. 

When $\phi = \psi$, then $V$ is commutative, else it is non-commutative, unless
$V_1 = 0$.
\end{theorem}

\begin{proof}
We give a direct proof of the first statement.
Left and right Kirmse identities are satisfied: letting $x=y$, two terms cancel
out, and the remaining ones give $q(x)z$.
Likewise for $z=y$.
To prove para-associativity, 
compute $\langle ab \langle cde \rangle \rangle$: observing that
\begin{equation}\label{eqn:cancel} 
\phi (\langle xyz \rangle) = \phi(x) \psi(y) \phi(z), \qquad
\psi (\langle xyz \rangle ) = \psi(x) \phi(y) \psi(z),
\end{equation} 
we see that from $9$ terms, $4$ cancel out, 
and the
remaining $5$ give
\begin{align*}
\langle ab \langle cde \rangle \rangle & =
\phi(a) \psi(b) \langle cde \rangle -
\phi(a) \psi(\langle cde\rangle) b +
\phi(b) \psi(\langle cde \rangle) a
\\
& =
\phi(b) \phi(d) \psi(c) \psi(e) \cdot  a \\
& - \phi(a) \phi(d) \psi(c) \psi(e) \cdot b \\
& + \phi(a) \phi(d) \psi(b) \psi(e) \cdot c \\
& - \phi(a) \phi(c) \psi(b) \psi(e) \cdot  d \\
& + \phi(a) \phi(c) \psi(b) \psi(d) \cdot e .
\end{align*}
Computing
$\langle a \langle dcb \rangle e \rangle$ and
$\langle \langle abc \rangle de \rangle$,
we get the same expression.

The ternary composition law holds:  using (\ref{eqn:cancel}),  we get
$$
q(\langle xyz \rangle) = \phi(\langle xyz \rangle) \cdot \psi(\langle xyz \rangle) =
\phi(x) \psi(y) \psi(z) \cdot \psi(x) \phi(y) \psi(z) = 
q(x) q(y) q(z).
$$

To describe the structure of $V$, recall that any $2$-dimensional submodule $V_0$ is a
group spherical subspace.
Choosing $V_0$ complementary to the radical
$V_1=\ker(\phi) \cap \ker(\psi)$ of $q$, it is readily seen that the structure of $V$ is isomorphic
to the split null extension of $V_0$ by $V_1$, where
$V_0$ acts on $V_1$ by
$x \cdot v = \phi(v) x + \psi(v) x$.
\end{proof}

\begin{example}\label{ex:3D}
The smallest ``new'' example is $V = \K^3$ with $\phi(x)=x_1$ and $\psi(y)=y_2$,
so $q(x) = x_1 x_2$. 
Let us choose the base point $ e = e_1 + e_2 = (1,1,0)$.
Then  binary product and involution are  given by
$$
x \cdot_e z = \begin{pmatrix}
x_1 z_1 + z_2 x_1 - x_1 z_2 \\
x_1 z_2 + z_2 x_2 - x_1  z_2 \\
x_1 z_3 + z_2 x_3 
\end{pmatrix}
=
\begin{pmatrix}
x_1 z_1  \\
 z_2 x_2  \\
x_1 z_3 + z_2 x_3 
\end{pmatrix} ,
\quad
x^\sharp = 
\begin{pmatrix} x_2 \\ x_1 \\ - x_3 \end{pmatrix}.
$$
The elements such that $x_1 x_2 \in \K^\times$ form the group $V^\times$, which clearly is 
isomorphic to the solvable group of upper triangular matrices
$\begin{pmatrix} x_1 & x_3 \\ 0 & x_2 \end{pmatrix}$.
The matrices with $x_1 x_2 = 1$ (the ``unit circle'') form a subgroup, which (if $2$ is invertible
in $\K$) form a subgroup isomorphic to the 
$2$-dimensional  $ax+b$-group (affine group of $\K$).
\end{example}

\begin{remark}
Theorem \ref{th:Minko} suggests to define, for any $\K$-module $V$, the 
following {\em quintary structure map} 
\begin{equation}
\begin{matrix}
\widetilde \Gamma :   V \times V^* \times V \times V^* \times V & \to &V, \\
 (x,a,y,b,z) & \mapsto &
 a(x) b(y) z -a(x) b(z) y  + a(z) b(y) x .
 \end{matrix}
\end{equation} 
This is a particular instance of the structure map $\Gamma$ defined in  \cite{BeKi1},
where the quotient spaces $\PP V$ and $\PP V^*$ are considered rather than $V$ and
$V^*$. For a fixed pair $(a,b)$, the map $(xyz)_{ab}=\widetilde\Gamma(x,a,y,b,z)$
 defines a torsor structure on the
projective space $\PP(V)$. 
\end{remark}

\subsection{Structure theorem for commutative group spherical spaces.}

\begin{theorem}\label{th:structure-commutative}
Assume $\K$ is a field of characteristic not $2$.
The commutative group spherical spaces $(V,q)$ over $\K$
are exactly those of the following forms:
\begin{enumerate}
\item
$V = \K$ and $q(x) = \lambda x^2$ (totally symmetric case),
\item
$V = \K \oplus W$, a non-trivial split null extension of the preceding type,
\item
$V = \K^2$ with an arbitrary non-degenerate quadratic form $q$.
\end{enumerate}
\end{theorem}

\begin{proof}
We have seen that these spaces exist and are commutative.

It remains to show that there are no other commutative group spherical  spaces.
Note first that a {\em commutative} $q$-compatible
 trilinear product on a quadratic space is {\em unique}
(Lemma \ref{la:existence}: it is one half times the Jordan map $D$).
When is it associative? 
By our assumption on $\K$, we can diagonalize the form $q$, and decompose
$q = q_0 \oplus q_1$, where $q_1$ is the zero form on the radical
$V_1=\{ x \in V \mid b_q(x,V)=0 \}$, and $q_0$ is non-degenerate on some complement $V_0$
of the radical. 
By uniqueness, the restriction of the trilinear product to the radical $V_1$ must be the zero product.
The same argument shows that $\langle xyz \rangle = 0$ whenever
two of the elements $x,y,z$ belong to $V_1$.
Thus only terms where at most one argument belongs to $V_1$ survive, and this leads
to a formula of the type given in Definition \ref{def:rep}.
Moreover, since $b_q(V_0,V_1)=0$, we must have
$\langle a bc \rangle = - \langle ba c \rangle$ whenever
$a \in V_0,b \in V_1$, leading to the formula given in the Definition.
The space $V_0$, being non-degenerate and commutative, can be at most
$2$-dimensional: this follows from the classical theory of (non-degenerate)
composition algebras, see e.g., \cite{McC}, p.164. 
Putting things together, $V_0$ is a non-degenerate composition algebra
of dimension $1$ or $2$, and $V_1$  must be a (right) $V_0$-module.
If $\dim V_0 = 2$, then commutativity implies $V_1=0$, and
if $\dim V_0 = 1$, then $V_1$ can be any $\K$-module.  
\end{proof}

\begin{remark}
If $2$ is not invertible in $\K$, one has to take more care with definition of 
radicals, and it is not quite clear if the structure theorem carries over to this case.
\end{remark}

\section{Quaternion-Clifford algebras}\label{sec:quaternions}

\subsection{From commutative to non-commutative group spherical spaces.}
There are several constructions associating to a {\em commutative} group spherical space
another, in general non-commutative, one.
Let us start with some necessary conditions.

\begin{lemma}\label{la:necessary}
Assume
$(V,q,\langle --- \rangle)$ is a group spherical space, 
$V_0 \subset V$ a subspace, and
 $e \in V_0^\times$  a base point, so the
binary product and involution $\sharp$ on $V$ are defined as in Theorem \ref{th:binary-ternary}.
Fix another element $m \in V^\times$ that is orthogonal to the subspace $V_0$, i.e.,
 $b_q(V_0,m)=0$.
Then we have, for all $a,b \in V_0^\times$,

\begin{enumerate}
\item
$m^2 = - \frac{q(m)}{q(e)} e = - N(m) e $, $m^\sharp = - m$, 
\item
$am = ma^\sharp$,
\item
$a(bm) = (ba)m$
\item
$(am)b = (ab^\sharp)m$
\item
$(am)(bm) = - N(m) ab^\sharp $
\end{enumerate}
In particular, $V_0 \oplus V_0 m$ is a subspace, and
 Item 3. implies that $V_0$ is necessarily commutative. 
\end{lemma}

\begin{proof}
We'lll use repeatedly  that,
if $b_q(x,y)=0$, then
$$\langle xyz \rangle + \langle yx z\rangle = b_q(x,y)z = 0, \mbox{ so }
\langle x y x \rangle = - \langle xx y \rangle = - q(x)y .
$$

1. 
In particular $m^2 = \frac{\langle mem \rangle}{q(e)} =  - N(m)e$.

2. Whenever $b_q(x,e)=0$, then $x^\sharp = \frac{\langle exe \rangle}{q(e)} = - x$. 
All spirations are conformal, i.e., they preserve $b_q$ up to a factor, so in particular they preserve
orthogonality.
Since by assumption $b_q(a,m)=0$, it follows that $b_q(am,e)=0$, whence
$$
am = - (am)^\sharp = - m^\sharp a^\sharp = m a^\sharp.
$$

3. Since $bm$ satisfies the same condition as $m$, the preceding item gives
$a(bm)=(bm)a^\sharp$. Together with associativity, this implies
$$
a(bm) = (bm)a^\sharp = b(ma^\sharp) = b(a m) .
$$
On the other hand, since $m$ is assumed to be invertible, this implies that necessarily $ab=ba$, by associativity.

4. By associativity,
$(am)b = a(mb)= a(b^\sharp m)=(ab^\sharp)m$.

5. Again, by associativity,
$(am)(bm) = ab^\sharp m^2 = - N(m) ab^\sharp$.

\nin
Altogether, these relations imply that $V_0 \oplus V_0m$ is stable under taking products.
\end{proof}

The proof of the lemma takes up  arguments from  \cite{McC}, p.\ 164/65, proving ``Jacobson necessity'':
necessarily, a subspace having non-trivial orthogonal complement is commutative, and the extension
$V_0 \oplus m V_0$ is a non-commutative subspace with relations 
given by the ``ACD''-extension, see below.

\begin{remark}\label{rk:dicyclic}
Let us assume $q(e) = q(m)$, that is, $N(m)=1=N(e)$.
Let $G = \{ x \in V \mid N(x)=1 \}$ be the unit sphere of the norm $N$, and
$H_0 := G \cap V_0$. Consider the subset $H:= H_0 \sqcup H_1 \subset G$.
The lemma implies that $H$ is stable under taking products; more precisely :
\begin{itemize}
\item
$H_0$ is a subgroup,
\item
$\forall u \in H_1$: $u^2 = - e$ (negative of the neutral element $e$),
\item
$\forall u \in H_1$, $\forall a \in H_0$: $u a u^{-1} = - uau = a^{-1}$.
\end{itemize}
If such a product is commutative, then necessarily
$H_0$ is commutative, and necessarily $H$ then is 
\href{https://en.wikipedia.org/wiki/Dicyclic_group#Generalizations}{\em the generalized dicyclic group
of $H_0$, with respect to the order-two element $-e \in H_0$}.
Summing up,
necessarily any group sub-sphere $H_0$ of $G$ having a non-trivial orthogonal complement gives rise
to a dicyclic extension $H \subset G$, which can only exist if $H_0$ was commutative.
\end{remark}

\subsection{The ACD-extension.}
Now we prove that the necessary rules from Lemma \ref{la:necessary} are also sufficient.
This is indeed a special case of  the general ACD-construction (Section \ref{sec:ACD}), which in the 
{\em commutative} case has the following interpretation in terms of matrices.

\begin{lemma}\label{la:ACD-com} 
Let $(V_0,q,e)$ be a commutative group spherical space, considered as commutative (generalized)
composition algebra $(\bA,e,\sharp)$.
Fix $\mu \in \K$. Then
$$
KD(\bA,-\mu):=
\Bigl\{ X_{a,b}:=
\begin{pmatrix}
a &  b^\sharp \\ \mu b & a^\sharp 
\end{pmatrix}
\mid \, a,b \in \bA \Bigr\} \subset M(2,2;\bA)
$$
is an associative algebra, with central involution
$(X_{a,b})^\sharp := X_{a^\sharp,-b}$, and imbedding
$\bA \to KD(\bA)$, $a \mapsto X_{a,0}$.
The new norm is
$$
N(X_{a,b}) = q(a) - \mu q(b).
$$
The element 
$m:= X_{0,e}$ satisfies $m^2 =  \mu e$,
and $am = ma^\sharp$ for all $a = X_{a,0}$.
\end{lemma}

\begin{proof}
Using commutativity,
$X_{a,b} X_{a',b'} =
X_{aa' + \mu b^\sharp b', ba' + ab'}
$, and

$X_{a,b} (X_{a,b})^\sharp = X_{aa^\sharp - \mu bb^\sharp , 0} = (q(a) - \mu q(b))e$.
\end{proof}

As a quadratic space, 
$KD(A,\mu)$ is a tensor product of $(V,q)$ with the binary space $(\K^2, x_1^2 + \mu x_2^2)$.
In particular, for $\mu = 1$, we have $N(X_{e,e}) = 0$, so the extension is split.
However, bilinear and trilinear product are not given by tensor products of algebras.
There is a link with Clifford algebras:

\subsection{The abstract Clifford algebra.}\label{sec:Clifford1} 
To every quadratic space $(V,q)$, there is canonically associated a {\em Clifford algebra}
$\Cl(V,q)$, together with an imbedding $i:V \to \Cl(V)$ such that
$\forall v \in V: i(v)^2 = q(v)1$, where $1$ is the unit of $\Cl(V)$, and $\Cl(V)$ is universal
for this relation (see, e.g., \cite{KMRT}).
The algebra $\Cl(V)$ is $\Z/2 \Z$-graded,
$\Cl(V) = \Cl(V)_0 \oplus \Cl(V)_1$. Elements of $i(V)$ are odd.
The even part $\Cl(V)_0$ is a unital subalgebra, and the odd part $\Cl(V)_1$ is stable
under taking the ternary product $\langle abc \rangle := abc$.
If $V = \K^n$, then $\Cl(V)$ is of dimension $2^n$, and $\Cl(V)_i$ of dimension $2^{n-1}$.

\begin{remark}\label{rk:proof}
In particular, if $n=2$, then for reasons of dimension,
$i(V) = \Cl(V)_1$, and therefore $i(V)$ is stable under the triple product
$\langle abc \rangle = abc$.
By the universal property,
$\langle aac \rangle = a^2 c = q(a)c = \langle caa\rangle$,
and so we have recovered the structure of group spherical space on $i(V)$.
Taking some care (to include the case of characteristic $2$), this can be used to 
give a complete proof of Theorem \ref{th:circles} (including a proof of the ternary
composition rule, which does not reduce to a triviality even with the theory of
Clifford algebras at hand).
\end{remark}

Still with $n=2$, the imbedding $V = \K^2 \to \Cl(\K^2)$ as odd part looks much like
the imbedding from Lemma \ref{la:ACD-com} , but it is not the same: 
for instance, when $q=0$, then $\Cl(\K^2,q)=\wedge(\K^2)$ is the exterior algebra of
$\K^2$, but $KD(\K^2)$ is the zero product algebra.
In $KD(\bA)$ with have $aa^\sharp = q(a)$, instead of the Clifford relation
$aa = q(a)$, which makes all the difference.
Related to this, $V$ imbeds into the {\em odd} part of $\Cl(V)$, whereas we find it more
convenient to consider the ``original'' copy of $\bA$ as {\em even} part $\bA_0$ (subalgebra)
of $KD(\bA)$.
Nevertheless,
whenever $q$ has invertible elements, replacing $q$ by a multiple, $\bA_0$ will become
unital, and in fact Clifford and quaternion algebra of $\K^2$ are isomorphic.
In a way, this amounts to switching odd and even parts.

\subsection{The ``concrete Clifford-quaternion algebra''.}\label{sec:Clifford2}
The following construction is kind of half way between the abstract Clifford algebra
of $V=\K^2$ and its ACD-double: 
it works for every {\em commutative} group spherical space $V$, which 
we realize as odd part,  to which we add as even part the spiration algebra
$\bfC_q^L$. 
Let $\langle xyz \rangle$ be the ternary product of $V$. Recall from Theorem \ref{th:involution1}
that the spiration algebra
$\bfC_q^L$ carries an   involution
$$
\sharp : \bfC_q^L \to \bfC_q^L,  \mbox{ such that :  } \forall x,y \in V : (R_{x,y})^\sharp = R_{y,x} .
$$

\begin{Definition}\label{def:concreteH} 
Assume $(V,q)$ is a commutative group spherical space.
We define its {\em concrete Clifford-quaternion algebra} $\bfH_q$   as the direct sum
of $\K$-modules
$\bfH_q := \bfC_q^L  \oplus V$, 
with  bilinear product  
$$
(f,v) \cdot (g,w) := (f \circ g + R_{v,w} , f w + g^\sharp v).
$$
By its definition,
this product is {\em $Z/2\Z$-graded}, where $\bfC_q^L$ is the {\em even part} (parity $0$), and
$V$ the {\em odd part} (parity $1$). We identify $v\in V$ with the odd element $(0,v)$, and
$f \in \bfC_q$ with the even element $(f,0)$.
The unit element is  $ 1 = (\id_V,0)$.
\end{Definition}

\begin{theorem}\label{th:concreteH} 
The algebra $\bfH_q$ is an associative $\Z/2\Z$-graded
composition algebra with
scalar involution, again denoted by 
$\sharp : \bfH_q \to \bfH_q$, given by
$$
(f,v)^\sharp = (f^\sharp, -v) .
$$
The product of three odd elements in this algebra gives the ternary
product in $V$:
$$
(0,x) \cdot (0,y) \cdot (0,z)  = 
 (0, \langle xyz \rangle),
$$
and we have the {\em Clifford relation for odd elements} $v$,
$$
v^2 = (0,v)^2  = (R_{v,v}, 0) = (q(v) \id_V,0) =  q(v) 1.
$$
\end{theorem}

\begin{proof}
The relation for three elements and the Clifford relation follow by direct computation from the
definitions.
To establish associativity $(XY)Z= X(YZ)$, we  check it in the following 
$2^3=8$ cases, according
to parity of the triple $(X,Y,Z)$. In 6 of these 8 cases, associativity
follows either directly from associativity of $\langle --- \rangle$ (in the various forms 
given by Theorem \ref{th:spiration1}) or from the definitions. 
In two  cases (even-odd-even, odd-odd-odd) we need also commutativity (of $\bfC_q^L$ or
$\langle --- \rangle$).
In the following table, $f,g,h \in \bfC_q^L$, $u,v,w \in V$: 
\begin{center}
\begin{tabular}{|*{10}{c|}}
\hline
parity type & equality & proof
\\
\hline
0 0 0  &  $(fg)h = f(gh)$ & 
associativity of $\bfC_q^L$
\\
\hline
0 0 1 &  $f(gv) = (fg)v$ & left action of $\bfC_q^L$ on $V$
\\
0 1 0 & $(fv)g = f(vg)$ & commutativity of $\bfC_q^L$: $f \circ g^\sharp = g^\sharp \circ f$
\\
 1 0  0 & $(vf)g= v(fg)$ & $\sharp$ is an involution: $(fg)^\sharp = g^\sharp f^\sharp$
\\
\hline
 0 1 1 & $ (fv)w = f(vw)$ & true for all generators $f=R_{x,y}$ by Theorem \ref{th:spiration1} 
\\ 
1 0 1 & $(vf)w = v(fw)$ & idem
\\
 1 1 0&  $(vw)f = v(wf)$ & idem
\\
\hline
 1 1 1 &  $(uv)w = u(vw)$ & $R_{u,v}(w) = \langle uvw \rangle = R_{w,v} (u) = (R_{v,w})^\sharp u$
 \\
 \hline
\end{tabular}
\end{center} 
Next, we check that $\sharp$ is an involution.
For $X = (f,v)$, $Y = (g,w)$ one sees by direct computation that
$(X^\sharp)^\sharp = X$ and   $(X Y)^\sharp = Y^\sharp X^\sharp$ 
(the main point is to use that $(R_{v,w})^\sharp = R_{w,v}$). 
We also compute 
\begin{align*}
X X^\sharp & = (f,v) \cdot (f^\sharp,-v) = (f f^\sharp - R_{v,v}, -fv + (f^\sharp)^\sharp v )
\\ & 
= (f f^\sharp - q(v),0) 
= (f f^\sharp -q(v)) 1 = X^\sharp X ,
\end{align*}
so the claim holds for the norm given by $N((f,v)) = f f^\sharp - q(v)$.
Further,
$$
X + X^\sharp = 
(f,v) + (f^\sharp,-v) = (f + f^\sharp, 0) = \tau(f) 1. 
$$
so $\sharp$ is a scalar involution.
\end{proof}

\begin{remark}
If we replace $q$ by $\lambda q$, the product in $KD(\bA)$ is also multiplied by $\lambda$.
Such scale invariance does not hold 
 for abstract or concrete quaternion algebras: indeed,
$\bfC_q = \bfC_{-q}$, so the algebra
$\bfH_{-q}$ is not isomorphic to $\bfH_q$.
\end{remark}

\subsection{Structure of non-commutative spherical quadratic spaces.} \label{sec:structure?} 
So far, we have encountered the following non-commutative group spherical spaces: 
\begin{enumerate}
\item
$V = \K^4$, a non-degenerate quaternion algebra,
\item
$V = U \oplus W$, a non-trivial split null extension by $W$ of
a commutative space of dimension $>1$.
\end{enumerate}
Also, as a result of the ACD-construction,
the quadratic forms on $\K^4=\K^2 \oplus \K^2$ corresponding to the first case are exactly
those of the form $\widetilde q = q \oplus \lambda q$ (orthogonal sum of a non-degenerate
binary form $q$ and a non-zero multiple of $q$), which are also exactly the tensor 
products of two non-degenerate binary quadratic forms.

One might conjecture that a ``structure theorem for non-commutative group spherical spaces''
holds, as an analog of Theorem \ref{th:structure-commutative}, saying that (under some
assumptions on $\K$), these are {\em exactly} the non-commutative group spherical spaces. 
However,
for the time being we have no argument ensuring that all extensions in Case 2 have to 
be ``split'': 
in the commutative case, the restriction of the ternary product to the radical had to be zero,
by symmetry; but in the non-commutative case, we cannot use this argument, and it might
be possible that some sort of ``non-split extension'' might exist. 
The following example gives some idea of that kind of complications:

\begin{example}\label{ex:Ext}
Let $U = \K^2$ with binary form $q(x)=x_1^2$.
Its associated algebra is the algebra $\K[X]/(X^2)$ of dual numbers (Theorem
\ref{th:bfR}). 
Let $W = \K^2$ be the ``adjoint'' $U$ module, i.e., the natural action of $U$ on itself on
the right, and
$\K^4 = U \oplus W$ be the corresponding split null extension.
It is a non-commutative space, since $U$ has non-trivial involution.
In other words, it is a very degenerate quaternion algebra -- in fact,  it is  just
the exterior algebra $\wedge \K^2$ with its central involution and ternary product
$\langle uvw \rangle = u \wedge v^\sharp \wedge w$.
Its quadratic form is again $x_1^2$.
Thus the ``sphere''  $\{ x \in \K^4 \mid x_1^2 = 1 \}$ has a commutative (split extension of $\K$)
 and a 
non-commutative (extension of $U$) group structure.
The latter is in fact a double extension, since $U$ is already an extension of $\K$.
Geometrically, it is the ``second tangent bundle of $\K$'',
$TT \K$. 
Also, we see that a degenerate quadratic form $q$ may admit several,
non-isomorphic, $q$-compatible ternary para-associative products: thus
the group sphere structure is not uniquely determined by the quadratic form.
\end{example}


\section{Cayley-Dickson doubling}

\subsection{The ACD (Albert-Cayley-Dickson) construction.}\label{sec:ACD}
Following the presentation given by McCrimmon (\cite{McC}, p.160 ff.),
we define, for any algebra $\bA$ with involution $\sharp$, and scalar $\mu \in \K$,
a new algebra
$KD(\bA):=\bA_0 \oplus \bA_1$ (where $\bA_i = \bA$, for $i=0,1$ are two copies of $\bA$)
with new involution again denoted by $\sharp$, as follows:
\begin{equation}\label{eqn:KD}
(x_0,x_1)\cdot (z_0,z_1) :=
(x_0 z_0 + \mu z_1^\sharp x_1,
z_1 x_0 + x_1 z_0^\sharp ),
\end{equation}
\begin{equation}\label{eqn:KD'}
(x_0,x_1)^\sharp = (x_0^\sharp, - x_1).
\end{equation}
We do not assume  that the scalar $\mu$ be invertible.
Note that for $\mu = 0$ we get the binary formula
for the split null extension from the preceding subsection. If $\bA_0$ has a unit $1$,
and $m=(0,1)$ is  the corresponding element of $\bA_1$, then (\ref{eqn:KD}) is equivalent to the
following (labels as in \cite{McC}): for $a,b \in \bA_0$,

(KD0) $am = m a^\sharp$

(KD1) $ab=ab$

(KD2) $a(bm) = (ba)m$

(KD3) $(am)b=(a b^\sharp) m$

(KD4) $(am)(bm) = \mu \, b^\sharp a$

\begin{remark}\label{rk:warning}
A word of warning: the formulae in \cite{Fa}, p.\ 51 ff.\ are not the same --
whereas McCrimmon considers $KD(\bA) = \bA \oplus \bA m$ (writing $\bA$ on the
{\em left} of $m$),
Faulkner considers it as
$\bA \oplus s \bA$, writing $\bA$ on the {\em right} of his element $s=(0,1)$.
This leads to ``reverse'' formulae, corresponding to our distinction between 
{\em left and right ternary products}, see below.
Note also that $KD(\bA)$ is in general not associative. More precisely,
for invertible $\mu$, by the ``KD Inheritance Theorem'', \cite{McC}, p. 162:
\end{remark}

\begin{theorem}  For invertible scalars $\mu$,
the algebra $KD(\bA)$ satisfies:
\begin{enumerate}
\item
The new map $\sharp$ is always an involution, which is scalar if so was the old one.
Then new norm and trace are
$$
N((x_0,x_1)) = N(x_0) - \mu N(x_1), \qquad
t((x_0,x_1)) = t(x_1).
$$
\item
$KD(\bA)$ is commutative iff $\bA$ is commutative with trivial involution,
\item
$KD(\bA)$ is associative iff $\bA$ is commutative and associative,
\item
$KD(\bA)$ is alternative iff $\bA$ is associative with central involution.
\end{enumerate}
\end{theorem}

\begin{remark}
As above, let $m=(0,1)$.
Then $N(m) = - \mu N(e)$. 
Therefore, starting with $N(e)=1$, one most often takes $\mu = -1$,
ensuring $N(m)=1=N(e)$.
For $\K=\R$, this leads to the sequence $\R,\bC,\bH,\bO$.
\end{remark}

\begin{Definition}
With notation as above, we call
\begin{itemize}
\item
{\em unarion algebra} $\bA = \K$ with product $xy = \lambda xy$,
\item
{\em binarion algebra} $\bA = \K^2$, algebra belonging to a binary quadratic form,
\item
{\em quaterninon algebra} $\bA = \K^4 = KD(\K^2)$, an extension of a binarion algebra,
\item
{\em octonion algebra} $\bA = \K^8 = KD(\K^4)$, an extension of a quaternion algebra.
\end{itemize}
\end{Definition}

\nin
Here, to define ``extensions''  we admit  any scalar $\mu$, thus slightly generalizing the terminology
from \cite{McC}.
The preceding theorem, slightly adapted, implies:
\begin{enumerate}
\item
all of these algebras have a scalar involution,
\item
a binarion algebra is commutative, associative, with non-trivial involution,
\item
a quaternion algebra is associative, and non-commutative, except freak cases,
\item
an octonion algebra is alternative, and non-associative, except freak cases.
\end{enumerate}
\nin
The last case will be considered later.
Let us, however, compute the {\em ternary product in $KD(\bA)$}, including the last
case (so there are two versions of ternary product, ``left''
$a(b^\sharp c)$ and  ``right''  $(ab^\sharp) c$, since $KD(\bA)$ is in general 
non-associative).

\begin{theorem}\label{th:ABCD}
Assume $\bA$ is an associative algebra with involution $\sharp$, and define  ternary products
on $\bA$ by
$\langle xyz \rangle := xy^\sharp z$ and on $KD(\bA)$ by
$$
\langle (x_0,x_1),(y_0,y_1),(z_0,z_1) \rangle_L :=
(x_0,x_1) \cdot \bigl( (y_0,y_1)^\sharp (z_0,z_1 ) \bigr)  .
$$
Then we have
\begin{align*}
\langle (x_0,x_1),(y_0,y_1),(z_0,z_1) \rangle_L & =
\Bigl( 
\langle x_0 y_0 z_0 \rangle - \mu 
\langle x_0 z_1 y_1 \rangle + \mu \langle y_0 z_1 x_1 \rangle -  \mu \langle z_0 y_1 x_1 \rangle ,
\\
& \quad \quad
\langle x_1 z_0 y_0 \rangle - \langle y_1 z_0 x_0 \rangle + \langle z_1 y_0 x_0 \rangle 
+ \mu \langle x_1 y_1 z_1 \rangle \Bigl) .
\end{align*}
There is a ``reverse'' formula, using the ``dual'' definition ($KD'$) of the KD-extension (see
Remark \ref{rk:warning}), and the ``right ternary product'' on $KD'(\bA)$ given by
\begin{align*}
\langle (x_0,x_1),(y_0,y_1),(z_0,z_1) \rangle_R & :=
\bigl( (x_0,x_1) \cdot  (y_0,y_1)^\sharp \bigr)  (z_0,z_1 )  
\\
& =
\Bigl( 
\langle x_0 y_0 z_0 \rangle - \mu 
\langle x_1 z_1 y_0 \rangle + \mu \langle y_1 z_1 x_0 \rangle -  \mu \langle z_1 y_1 x_0 \rangle ,
\\
& \quad \quad
\langle x_0 z_0 y_1 \rangle - \langle y_0 z_0 x_1 \rangle + \langle z_0 y_0 x_1 \rangle 
+ \mu \langle x_1 y_1 z_1 \rangle \Bigl) .
\end{align*}
\end{theorem}

\begin{proof}
Direct computation. 
\end{proof}

\begin{remark}
Note that  for $\mu=0$ we get the formula for the split null extension.
In this case, no product $\bA_1 \times \bA_1 \to \bA_0$ is needed, and for this
reason we can allow $\bA_1$ to be a more general space.
Similarly for left modules. 
\end{remark}

\section{Moufang loop spheres}\label{sec:Moufangspheres}

\subsection{From group to loop spheres.}
The category of group spherical  spaces is quite rigid and lacks many
``usual'' constructions:  group spheres are ``rare'' in nature;
there are no ``direct sums'' nor ``tensor products''.
The closest analog of a direct sum is the ACD-doubling construction,
but already  the unit sphere of the octonions is not a group, but a
{\em Moufang loop}, cf.\ \cite{CS}.
As is well-known (\cite{McC, Fa}), for non-degenerate forms, this follows necessarily
from the composition rule.
Therefore the category of {\em Moufang loop spheres} seems to be the most
natural generalization.
In a first step, we  develop a {\em ternary concept of Moufang loops}:
this is given in  Appendix \ref{app:torsors}, Definition \ref{def:TernaryMoufangLoop}.

\begin{Definition}\label{def:Mouf-quadratic}\label{def:Moufang-spherical}
A {\em right Moufang spherical space} is a quadratic space $(V,q)$
such that $V^\times$ is not empty,  with
a trilinear product $\langle xyz \rangle$ on $V$ such that:
\begin{enumerate}
\item
Ternary composition rule $q(\langle xyz \rangle)= q(x)q(y)q(z)$,
\item
Left and Right Kirmse $\langle xxy \rangle = q(x)y = \langle yxx \rangle$,
\item
The ternary product is a {\em right alternative triple system} in the sense of
\cite{Lo75}, p.57: it satisfies the three identities

(A1) $\langle uv \langle xyz \rangle \rangle +
\langle xy \langle uvz \rangle \rangle = 
\langle \langle uvx \rangle yz \rangle + \langle x \langle vuy \rangle z \rangle$

(A2)
$\langle \langle uvx\rangle yx\rangle  = \langle uv\langle xyx\rangle \rangle$ 

(A3)  $ \langle xy\langle xyz \rangle \rangle = \langle \langle xyx \rangle yz\rangle $
\end{enumerate}
A {\em left Moufang spherical  space} is defined similarly, replacing (A1), (A2), (A3)
by their ``dual'' identities obtained by reversing order of arguments (cf.\ \cite{Lo75}, p.59).
\end{Definition}

\begin{theorem}\label{th:binary-ternary2}
There is a bijection between alternative (generalized) composition algebras and
right Moufang spherical  spaces, up to choice of base point. More precisely, 
any alternative composition algebra with ternary product
$\langle xyz \rangle = x(y^\sharp z)$ and $q$  given by
$xx^\sharp = q(x) 1$, is a right  Moufang spherical  space.
For the other direction,
let $(V,q)$ be a right Moufang  spherical  space, and
fix  $e \in V^\times$. Then
\begin{enumerate}
\item
$V^\times$ with $(xyz) := \frac{\langle xyz \rangle}{q(y)}$ is a ternary right Moufang loop, and every sphere is a ternary subloop
of $V^\times$,
\item
the bilinear product $xz:= x \cdot_e z = \frac{\langle xez \rangle}{q(e)}$ turns $V$ into an alternative algebra with
neutral element $e$, and with involution $\sharp : V \to V$,
$$
x^\sharp = \frac{ \langle exe \rangle}{q(e)} = \frac{ b_q(x,e)}{q(e)} e - x ,
$$
and
this algebra is a {\em composition algebra} with respect to the {\em norm}
$N(x):= \frac{q(x)}{q(e)}$.
\end{enumerate}
In a similar way, alternative composition algebras also correspond to left ternary 
Moufang quadratic spaces. 
\end{theorem}

\begin{proof}
Assume first that $(\bA,e,\sharp)$ is an alternative composition algebra.
Every such  algebra gives rise to an alternative triple system via
$\langle xyz \rangle = x(y^\sharp z)$
(see \cite{Lo75}, p.60).
Since $x+x^\sharp = t(x) e$ with $t(x) \in \K$, so $x^\sharp = t(x)e-x$, 
from alternativity $L_x^2 = L_{x^2}$ and $R_x^2 = R_{x^2}$ we get
$L_x L_{x^\sharp} = q(x) \id$ and
$R_x R_{x^\sharp} = q(x) \id$, and from this both Kirmse identities follow:
$\langle xxy \rangle = L_x L_{x^\sharp} y = q(x)y = y (x^\sharp x)= \langle yxx \rangle$.
To prove the ternary composition law, it suffices to show that
$q(ab) = q(a)a(b)$ and $q(a^\sharp)=q(a)$. The latter follows from 
$aa^\sharp = q(a)e=a^\sharp a = q(a^\sharp)e$.
Every alternative ring satisfies the Moufang condition  
$(ab)(ca)=(a(bc))a=a((bc)a)$ (cf.\   \cite{Fa}, Lemma 3.16).
Using this,
\begin{align*}
(ab)(ab)^\sharp &=
(ab)(b^\sharp a^\sharp)=- (ab)(b^\sharp a )+t(x) (ab)b^\sharp \\
& =
- a q(b) a + t(x) a q(b) = q(b) (t(x)a-a^2) = q(b) aa^\sharp = q(b) q(a)e .
\end{align*}
Putting things together, we get the ternary composition rule. 
As to the converse:

1. Idempotency $(xxy)=y=(yxx)$ follows from the ternary composition law, as in the proof of
Theorem \ref{th:binary-ternary}.
Then (A2) and (A3) become the properties

(MT1)  $\langle \langle uvx\rangle yx\rangle  = \langle uv\langle xyx\rangle \rangle$ 

(MT2)  $ \langle xy\langle xyz \rangle \rangle = \langle \langle xyx \rangle yz\rangle $

\nin
defining a ternary Moufang loop in \cite{BeKi14}, and which are equivalent to the definition
of a right ternary Moufang loop
given in Appendix A, Def.\ \ref{def:TernaryMoufangLoop}, cf.\ Remark \ref{rk:BeKi}.
Thus $V^\times$ becomes a right ternary Moufang loop.

2. The fact that every homotope algebra $\cdot_e$ is alternative
follows directly from the defining axioms
(cf.\ \cite{Lo75}, p.64), and proving that $\sharp$ is a central involution follows the same line
of arguments as in the proof of Theorem \ref{th:binary-ternary}.

All arguments remain valid with ``right'' replaced by ``left''.
\end{proof}

Definition \ref{def:Mouf-quadratic} follows the axioms given by Loos \cite{Lo75}, and used
in \cite{BeKi14}. However, they are not very ``geometric''.
The following list of properties is longer, but more conceptual, and could be used to give
an equivalent definition (possibly under some assumptions on the $\K$-module $V$).
The observation is simply that every identity valid for Moufang loops has a
``$q$-analog'', which inserts a term $q(u)$ if in an identity two instances of the argument $u$
are ``cancelled''.
For instance, the $q$-analog of idempotency is Kirmse (K).

\begin{theorem}
A  left Moufang spherical  space has the following properties:
it satisfies the ternary composition rule and the left and right Kirmse identitis, and, 
for invertible elements $x,y,\ldots$: 
\begin{enumerate}
\item
the $q$-analogs of the defining identities of an inverse loop:

$\langle xy \langle \langle yxy \rangle y u \rangle \rangle =
q(x) q(y)^2 u = \langle \langle yxy \rangle y \langle xyz \rangle \rangle$

$\langle \langle u y \langle yxy \rangle \rangle yx \rangle =
q(x) q(y)^2 u = \langle \langle uyx \rangle y \langle yxy \rangle \rangle$,
\item
the $q$-analog of the left Chasles relation:

$\langle ab \langle bdx \rangle \rangle = q(b)  \langle adx \rangle$,
\item
and the $q$-analog of the autotopy relation from Def.\ \ref{def:TernaryMoufangLoop}:

$\langle \langle xab \rangle \langle yba \rangle \langle zba \rangle \rangle =
q(a)q(b) \langle \langle xyz \rangle ab \rangle$
\end{enumerate}
{\em Right Moufang spherical quadratic spaces} have similar properties,
replacing the
Left by the Right Chasles relation and the autotopy relation by its right analogue. 
\end{theorem}

\begin{proof}
These properties, without the terms $q(x)$, etc., are our defining properties of a 
left ternary Moufang loop, and therefore hold for the ternary loop structure defined in the
preceding theorem. Inserting the definition $(xyz) = \frac{\langle xyz \rangle}{q(y)}$
implies the formulae from the claim. 
(``By density'' these formulae will remain valid outside $V^\times$, and under some
assumptions extend to the whole
of $V$; we will not go into such details.)
\end{proof}

\subsection{Doubling: ABCD-construction,  octonions, and dicyclic loop of a group sphere.}
There exist {\em proper} Moufang spherical spaces, i.e., Moufang spheres that are not
associative: 
starting with a non-commutative group spherical space $(V,q)$, fixing a unit element
$e$, the ACD-double $KD(\bA)$ of the algebra $\bA = (V,e)$, is a non-associative
octonion algebra (Section \ref{sec:ACD}), hence defines a proper Moufang spherical space.
In order to avoid the apparent dependence on the choice of base point, and the involution
$\sharp$ depending on this choice, we have given in
Theorem \ref{th:ABCD} a ternary version (``ABCD construction'') of the ACD-construction.

\begin{example}
Starting with a non-zero binary quadratic form $q$ on $\K^2$, we may form its
concrete Clifford-quaternion algebra $\bfH_q$ (which is non-commutative and unital even if the form $q$ did not 
represent $1$), and then its ACD-extension $KD(\bfH_q)$ with $\mu=-1$, which is non-associative.
We call it the {\em octonion algebra associated with a binary quadratic form}. 
\end{example} 

\begin{theorem}
Let $(V,q)$ be group spherical space with base point $e \in V^\times$, 
and
$\bA$ its homotope composition algebra at $e$.
The Moufang loop  $KD(\bA)^\times$ contains as subloop the Moufang double
(Appendix \ref{app:Moufang-double}, with $\eps = -1$)
$$
D(V^\times) = \bA^\times \sqcup \bA^\times
$$
of the group of invertible elements of $V$.
Moreover, if $\mu = - 1$, then
a sphere $S$ of $KD(\bA)$ satisfies:
$S_0:=S\cap \bA \not= \emptyset$ iff
$S_1 := S \cap \bA m \not= \emptyset$, and then
$S_0 \sqcup S_1$ is the dicyclic loop  of $(S_0,-e)$ (Def.\ \ref{def:dih(G)}).
\end{theorem}

\begin{proof}
The rules (KD0) -- (KD4) from Section \ref{sec:ACD} correspond exactly to the definition
of the Moufang double of a group, and therefore this Moufang double is naturally contained
in $KD(\bA)$.
Moreover, for $\mu = -1$, we readily get the formulae defining the dicyclic extension of
$S_0$. 
\end{proof}

The preceding result is a kind of non-associative analog of the ``dicyclic extension of a commutative
group sphere'' (see Remark \ref{rk:dicyclic}):
the ``dicyclic extension of a non-commutative group'' is not a group, but a Moufang loop
(Appendix \ref{app:Moufang-double}). 
Note that this can also be formulated in terms of ternary products (Theorem \ref{th:Moufang-ternary}).

\begin{example}\label{ex:dicloop}
We start with the binary quadratic form
$q(x_1,x_2)=x_1^2 -  x_1 x_2 + x_2^2$ over $\K=\Z$.
The corresponding commutative group spherical space is the lattice
of {\em Eisenstein integers}  $\mathbb E = \mathbb E_2$.
Its unit sphere is $S = \{ \pm e_1, \pm e_2,\pm(e_1+e_2)  \}$; as a group, this
is the cyclic group $C_6$; by the way, here $S = V^\times$. 

Next, let $\mathbb E_4 = KD(\mathbb E)$ with $\mu = -1$.
It
contains the Moufang double of $(C_6,-1)$, which is the dicyclic group $Dic_3$ of cardinal $12$.

The final step
 $\mathbb E_8 = KD(\mathbb E_4)$ with $\mu = -1$ 
 contains a ``dicyclic Moufang double'' of $Dic_3$, which is a certain Moufang loop of cardinal
 $24$.
 
 After scalar extension $\otimes_\Z \R$ from $\Z$ to $\R$, this is isomorphic to the usual octonions over
 $\R$, but not with the usual sublattice leading the usual integral octonion loop of cardinal $16$
(see \cite{CS}, Chapter 9).
\end{example}

\subsection{Split null extensions, and the structure of Moufang spherical spaces.}
The split null extension (Def.\ \ref{def:rep}) of a group spherical space is the special case
$\mu=0$ of the ABCD-construction mentioned above, where the second factor need no longer
be isomorphic to the first, as $\K$-module.
Thus, by the general results on the ACD/ABCD-construction, it is Moufang spherical,
and we have the following analog of Theorem \ref{th:splitnull}:

\begin{theorem}\label{th:splitnull2}
For every group spherical space $(V,q)$, and every
right $(V,q)$-module $W$, the split null extension $V \oplus W$ is a
Moufang spherical  space.
It is not associative, unless $V$ is commutative or $W=0$. 
\end{theorem}

As in Subsection \ref{sec:structure?}, one might wonder if there is some kind of ``structure theorem
for Moufang spherical spaces'', giving some rough classification:
the non-degenerate spaces are octonion algebras, and the degenerate spaces iterated
(split or non-split) null extensions of $\K$. 
Such degenerate Moufang spheres do exist:
we have seen (Theorem \ref{th:Minko}) that there are non-commutative group spherical spaces of any dimension $n>2$;
taking an ACD-extension of such a space defines a non-associative group spherical space of any even dimension
$2n$. 
It should be interesting to study in more detail such families of ``degenerate octonion algebras''.

\section{Outlook}\label{sec:outlook}

My interest in the topics discussed in the work comes from the interplay of ``Lie- and
Jordan geometries'', \cite{Be00, Be14}, and which are related to ``associative'' and
``alternative'' geometries, \cite{BeKi1, BeKi14}.
Group and loop spheres there appear as ``atoms'' and important building blocks, and
for this reason seem to deserve a name and a theory which fits into a  general theoretic framework,
for which the present work is a starting point.
Here are some topics which I would like to enlarge upon, time permitting:

\subsection{Plane geometry over $\K$.}
Since several years I have taught, in Nancy, the undergraduate course on  ``plane 
affine and Euclidean geometry'' and among other sources I used lecture notes by Max Koecher,
which later became the book
{\em  Ebene Geometrie} coedited  by 
Aloys Krieg, \cite{KK}. The authors take  care to develop the analytic side of affine and Euclidean
geometry in a  detailed and intrinsic way, proving most results by explicit and base-independent formulas. 
In loc. cit.,  p.88, they notice that most formulae continue to make sense when replacing the real number field
$\R$ by a general base field $\K$, and conclude:
{\em  
With a convenient ``geometric'' interpretation, we have developed at the same time a
\textbf{geometry over $\K$}. This already should justifiy our computational effort...}
For the sake
of pure mathematics, I think one should try to develop this approach  as far as
possible. One  starting point of the present work was an essay to make a
step into this direction, which lead to the elementary presentation of 
 Theorem \ref{th:circles}  chosen here.
 It would certainly be interesting to continue developing the ``plane geometry over
 $\K$'', in the sense of \cite{KK}, by making use of the group sphere structure.
 
\subsection{The Lister-Loos algebra.}
Every group spherical space is in particular an {\em associative triple system}.  Lister and Loos,  \cite{Li71, Lo72},
construct, for such triple systems, a binary ``enveloping algebra'' with involution given
by an idempotent,
such that the triple system becomes the $-1$-eigenspace of the involution, and
the triple product becomes the product of three elements.
Such an algebra is not unique, but in case of a group spherical space, one can show that it is 
essentially the algebra $M(2,2;V_e)$ of $2\times 2$-matrices with coefficients in
the homotope algebra at $e$.

Of course, for Moufang spherical spaces, this construction brakes down, and 
one rather is directed towards the general construction from \cite{Lo75},
\S 8, imbedding
alternative pairs as Peirce $1$-spaces of a Jordan pair with idempotent. 
Can one simplify this construction in the special case of a Moufang spherical case?
One may guess that is related to the Jordan algebra
${\rm Herm}(3,V_e)$ of Hermitian matrices with coefficients in the alternative algebra $V_e$,
 but this remains to be worked out.

\subsection{Projective imbedding, projective geometry.}
The invertible elements in the 
Lister-Loos algebra define  what one might call the ``projective group of
a group spherical space'' (the analog of  $\PP\GL(2,\bC)$), 
and one can define a ``projective completion of $V$'',
on which this group acts, very much in the way the Riemann sphere completes
the complex plane. 
In the commutative case, this generalizes the theory of 
\href{https://en.wikipedia.org/wiki/Benz_plane}{\em Benz planes};
in the non-commutative case, the theory becomes more complicated.
It should have a similar shape as the theory of {\em associative geometries}
from \cite{BeKi1}, but it is not a special case of it, and there remain new things to be
understood.
In particular, combined with methods from \cite{BeNe}, this should lead to a nice
``projective geometry of Benz planes''. 

\subsection{Jordan theory and Jordan geometries.}
Since every sphere is a symmetric space,
there is an ``underlying functor'' from group or Moufang spherical spaces to
Jordan geometries (cf.\ Remark \ref{rk:Jordan}), and also the projective completion
and the projective geometry of such spaces are  particular instances of
``Jordan geometries'', see \cite{Be00, Be14},  which are important for the general
theory.

\subsection{Pure algebra.}
It is quite amazing how more and more
complicated structures arise
from the simple datum of binary quadratic form: 
starting with 
$q(x)=\alpha x_1^2 + \beta x_1 x_2 + \gamma x_2^2$, 
we can do all constructions we are used to in the case of the complex
plane, leading up to the analog of the octonions, the Jordan algebra
of Hermitian $3\times 3$-matrices with octonion coefficiencts, and the ``big'' groups
and spaces associated  to it.
In other terms, it appears that starting from a binary quadratic form, there arises a 
whole ``Freudenthal's magic square'' (cf.\ \cite{Fa}, Chapter 14), and
all of of these structures depend in an explicit and functorial way on the binary quadratic  form.

\subsection{Lattices, root systems.}\label{sec:roots2}
The base ring $\K = \Z$ is admitted for our approach, and thus the theory 
works for {\em integral forms}, and in particular for {\em lattices}, which are 
an important object of study (cf.\ \cite{CSl, CS}).
In the present work, I have not even touched upon this aspect. 
Via the theory of {\em root systems} it enters the ``Jordan-Lie theory'',
cf.\ \cite{LoNe, Lo85}. In a personal communication,
O. Loos admitted to me  that he would like to still better
understand the description  of the ``Jordan-Lie functor''  in terms of root systems
that he gave in \cite{Lo85}, and to generalize it beyond the compact case. 
Possibly, the language developed here might be useful in such an attempt. 

\subsection{Gauss composition.}
Still concerning the base ring $\K=\Z$, one should expect some link of the 
theory presented here with the classical
\href{https://en.wikipedia.org/wiki/Gauss_composition_law}{\em
Gauss composition of binary quadratic forms}, and with
Bhargava's beautiful work \cite{Bh}.


\subsection{Invertible elements, pair concepts.}
Since Theorem \ref{th:circles} is valid for {\em any} quadratic form, even
without invertible elements, I tried in a first version to develop also the 
sequel of the theory in  his generality. 
To a certain extent, this is possible, but tends to become quite formal, and I abandoned this try.
Possibly, like in general Jordan theory, the good way to pursue this would be
to replace the concept of a ``triple system'' by a ``pair concept'', see the arguments
given by Loos in the introduction to \cite{Lo75},
and Section \ref{sec:polarized}.

\subsection{Super-group spherical spaces, and the ``tenfold way''.}
Like most algebraic concepts, the one of ``group spherical space'' admits a ``super-version''.
This should be
a $\Z/2\Z$-graded $\K$-module, with a graded para-associative trilinear map $\langle - - - \rangle$,
satisfying certain axioms:
the homogenous parts $V_0$ and $V_1$ should be quadratic spaces, the Kirmse identities should be
valid for triples of homogeneous elements of same parity, along with the composition formula.
This would generalize the concept of ``associative super division algebra'', see
\cite{Baez10}.
My impression is that such a concept could serve to better understand the ``tenfold way'' (loc.cit.),
in terms of graded algebras, in much the same way as I have attempted this for Clifford algebras over
rings continaing $\frac{1}{2}$, in \cite{Be21}.

\appendix
\section{Ternary products}\label{app:torsors} 

\subsection{Binary and ternary products; left, right, middle.}
An {\em $n$-ary magma} is a set $M$ with an $n$-ary ``product map''
$\mu : M^n \to M$.
When $n=2$, this is a {\em binary product}, and we often use the notation
$\mu(x,y) = x \cdot y = xy$, and when $n=3$, it is a {\em ternary product}, and
we often use one of the following
notations for $\mu(x,y,z)$:
\begin{equation}
\langle x \vert y \vert z \rangle 
\mbox{ or }
\langle x, y, z \rangle
\mbox{ or }
\langle xyz \rangle 
\mbox{ or }
(x \vert y \vert z ) 
\mbox{ or }
(xyz) .
\end{equation} 
For a binary product, we have {\em operators of left and right multiplication},
\begin{equation}
L_x(y) = x \cdot y = R_y(x),
\end{equation} 
and for a ternary product $\langle -,-,- \rangle$, we define 
operators $L_{x,y}:M\to M$ of  {\em  left multiplication},
$R_{u,v}$ of  {\em right multiplication}, and
$S_{a,b}:M \to M$ of  {\em middle multiplication}, by
\begin{equation}\label{eqn:LRS} 
L_{x,y}(z) := \langle x,y,z \rangle =: R_{z,y}(x) := S_{x,z}(y).
\end{equation}

Fixing one of the three arguments in a ternary product defines a binary product.
If we fix the {\em middle} element $y$, we call the binary product
\begin{equation}
x z := x \cdot_y z := \langle x y z \rangle ,
\end{equation}
the {\em homotope at $y$}. 
Then $L_x = L_{x,y}$ and $R_z = R_{z,y}$.

A $\K$-{\em algebra} is a $\K$-module together with a $\K$-bilinear binary product,
and a {\em triple system}, or {\em ternary algebra} is a $\K$-module together with a
$\K$-trilinear ternary product. 
(In this case we often prefer the notation $\langle -,-,-\rangle$.)

\subsection{Binary and ternary quasigroups and loops.} 
A {\em quasigroup} is given by a binary product such that all operators of
left and right multiplication are bijective, and a
{\em ternary quasigroup} is given by a ternary product such that all operators of
left, right and middle multiplication are bijective. 

A {\em loop} is a binary quasigroup together with a {\em unit element} $e$, that is,
there exists $e$ such that  $L_e = R_e = \id_M$.

 An {\em inverse loop} is a loop such that every element $x$ admits an {\em inverse $x^{-1}$}
 satisfying $L_x^{-1} = L_{x^{-1}}$ and $(R_x)^{-1} = R_{x^{-1}}$, i.e.,
 $x(x^{-1} y) = x^{-1}(xy)=
  y = (y x)x^{-1} = (yx^{-1})x$ (for $y=e$, this implies $x=(x^{-1})^{-1}$).

A {\em ternary loop} is a ternary quasigroup satisfying  the
{\em idempotency identity} 
$$
\forall x,y \in M: \qquad 
\langle x y y \rangle = x = \langle yy x \rangle
\leqno{\mathrm{(IP)}} 
$$
In other words, the homotope at $y$ has $y$ as unit element: $L_{y,y} = \id = R_{y,y}$.
From a logical point of view, 
the binary loop property is an {\em existence requirement}, whereas
the ternary one is an {\em identity} that holds for {\em all} elements. 
As a consequence,
idempotency (IP) is not compatible with trilinearity, except for the zero product. 
A main thread of the present work is that, for triple systems, (IP) has to replaced by another identity,
namely by the Kirmse-identity (K).

\subsection{Identities: commutativity.}
Binary and ternary products may satisfy (or not) various {\em identities}.
A binary product is called {\em commutative}, if it satisfies the identity
$xy=yx$.
A ternary product is called {\em commutative} if it satisfies 
$$
(xyz) = (zyx ),
\leqno\mathrm{(C)}
$$
and we say that it is {\em totally symmetric} if it is invariant under all
$6$ permutations of the three arguments.
Thus our notation reflects a {\em choice}:  often (but not always) the {\em outer}
variables will play somewhat similar roles, whereas the {\em inner} variable has a
different kind of role.

\subsection{Associativity, groups and torsors.}
The most important identity is {\em associativity}.
In the binary case, it reads as usual $x(yz)=(xy)z$, whereas in the ternary case,
there are {\em two different versions}: 
($\forall a,b,c,d, e \ldots \in M$):
{\em associativity} (AT1), or {\em para-associativity} (AT2),
\begin{align*}
\langle a , b,  \langle c , d , e \rangle \rangle &=
\langle a, \langle b , c ,d \rangle ,  e \rangle =
\langle \langle a,b,c \rangle , d,e \rangle 
\tag{AT1}
\\
\langle a , b,  \langle c , d , e \rangle \rangle & =
\langle a, \langle d , c ,b \rangle ,  e \rangle =
\langle \langle a,b,c \rangle , d,e  \rangle 
\tag{AT2}
\end{align*}

\begin{example}
For every binary associative product
 $M^2 \to M$, $(x,y) \mapsto xy$, the ternary product 
$\langle xyz \rangle := (xy)z = x(yz)$ satisfies (AT1), and if moreover
$M \to M$, $x \mapsto x^*$ is an anti-automorphism of order $2$, then the new product
$\langle xyz \rangle := (x y^*) z = x (y^* z)$ satisfies (AT2). 
In the other direction, given a ternary product satisfying
(AT1) or (AT2), for any $y \in M$, the {\em homotope}
$xz := x \cdot_y z := \langle x, y , z \rangle$.
is a binary associative product on $M$. 
\end{example}

\begin{theorem}[Torsors]\label{th:Torsors}
In every group $(G,\cdot)$, we define its {\em torsor structure} by
\begin{equation}
(xyz) := xy^{-1} z.
\end{equation} 
It satisfies {\rm (IP)} and {\rm (AT2)}.
Conversely, in any non-empty set $G$ with ternary product satisfying {\rm (IP)} and {\rm (AT2)},
for any $y \in G$, the homotope is a group law on $G$. 
\end{theorem}

\begin{proof} Straightforward computation:
the neutral element in the homotope at $y$ is $y$,
and the inverse is $x^{-1} = (yxy)$.
\end{proof}

\begin{Definition}\label{def:torsor}
Following previous work \cite{BeKi1}, we  will use the term {\em torsor} for a set with ternary
product satisfying (AT2) and (IP).
By {\em semi-torsor} we mean a ternary product satisfying (AT2), but possibly not (IP).\footnote{Unfortunately, there is no universally  adopted terminology for these objects (categories, in fact) -- 
other terms are {\em (semi)heap, principal homogeneous space, (semi)groud,...} -- 
see \cite{BeKi1} for some remarks concerning
terminology and its history.
}
\end{Definition} 

The preceding discussion can be summarized by the slogan:
{\em ``Torsors are for groups what affine spaces are for linear spaces.''}
Note that we are thus  naturally lead to (AT2), and not to (AT1), and it is for this reason that 
(AT2) plays a more important role than  (AT1).

\begin{Definition}
An {\em associative triple system (of the first, resp.\ second kind)}  is  a triple system
over $\K$
satisfying (AT1), resp.\ (AT2).
\end{Definition}

See \cite{Li71, Lo72} for  a  theory of associative triple systems.

\subsection{Weakening of  associativity: binary case.}
The binary associative law may be weakened in several ways in order to 
define structures that are ``close to associative''.
For instance, a binary product is called {\em alternative} if
\begin{equation}\label{eqn:alternative}
L_x^2 = L_{x^2}, \, R_x^2 = R_{x^2}: \qquad
x(xy) = (xxy), \,  (yx)x = y(xx) .
\end{equation} 
A binary loop with unit $e$ is an {\em inverse loop} if it has a unary operation
``inverse'' $x \mapsto x^{-1}$ such that
\begin{equation}
L_x^{-1} = L_{x^{-1}}, \, R_x^{-1} = R_{x^{-1}}: \quad
x^{-1} (xy) = y = x (x^{-1} y) = (yx)x^{-1} = (yx^{-1})x.
\end{equation}
It follows that $(x^{-1})^{-1} = x$, and the condition $xy = z$ can be written in six
ways (in \cite{CS}, Section 7.1, this is called the ``hexad'' of relations):
\begin{equation}\label{eqn:hexad1}
xy=z, \,
x=zy^{-1}, \,
z^{-1} x= y^{-1}, \,
z^{-1} = y^{-1} x^{-1}, \,
yz^{-1} = x^{-1}, \,
y = x^{-1} z.
\end{equation}
In particular, inversion then is an {\em involution} of the binary product (anti-automorphism of order
$2$), and (p.\ 87 loc.cit.) the {\em flexibility relation} 
$L_a R_a= R_a L_a =: B_a$ holds:
\begin{equation}
(ax)a=a(xa) =: B_a(x),
\end{equation}
defining the {\em operator $B_a$ of bimultiplication by $a$}.
The relation $x(yz)=e$ is equivalent to $(xy)z=e$, and written as
$xyz = e$, giving rise to {\em triality}, see Prop.\ \ref{prop:triality}.

\subsection{Ternary inverse loops.}
The ternary (para) associative law (AT2) can be weakened 
in many ways.

\begin{Definition}
A {\em ternary inverse loop} is a ternary loop $M$ such that, for each fixed $y \in M$,
the binary homotope at $y$ is an inverse loop with inverse $x^{-1} = (yxy)$,
that is, the following identitites hold:
$$
L_{x,y} \circ L_{(yxy),y} = \id = L_{(yxy),y} \circ L_{x,y}, \quad
(xy ((yxy) yu )) = u = ((yxy) y (xyu)),
$$
$$
R_{x,y} \circ R_{(yxy),y} = \id = R_{(yxy),y} \circ R_{x,y}, \quad
((u y (yxy)) yx) = u  = ((uyx) y (yxy)).
$$
\end{Definition}

It follows that, in the homotope at $y$, the flexibility and the automorphic inverse properties
hold, that is, we have also the following identities:
\begin{equation}
B_{a,y} := L_{a,y} \circ R_{a,y} = R_{a,y} \circ L_{a,y}, \qquad
(ay(x ay)) = ((ayx)ay) ,
\end{equation}
\begin{equation}
S_{y,y} ((xyz)) = ( S_{y,y}z,y, S_{y,y}x), \qquad
(y(xyz)y)=((yzy)y(yxy)) .
\end{equation}
Also, inversion must be of order two,
$(y(yxy)y) = x$.


\subsection{Ternary left- and right half-torsors}

\begin{Definition}
A {\em left half-torsor} is a ternary inverse loop satisfying the  {\em left Chasles relation}:
$$
L_{x,y} \circ L_{y,z} = L_{x,z}, \quad
(xy(yzu))= (xzu).
$$
A {\em right half-torsor} is a ternary inverse loop satisfying the {\em right Chasles relation}:
$$
R_{x,y} \circ R_{y,z} =  R_{x,z}, \quad
((uzy)yx) = (uzx) .
$$
\end{Definition}

\begin{lemma}\label{la:Chasles}
Every left half-torsor is given by, using the homotope at $e$, for arbitrary $e$,
$$
(xyz) =  x (y^{-1} z) = (xe ((eye) ez ) .
$$
Conversely, given a binary inverse loop, we can define a left half-torsor on the
same underlying set by letting $(xyz):= x(y^{-1} z)$.
Similar statements hold for  right half-torsors given by
$(xyz) = (xy^{-1})z$. 
In both cases, the following identities holds:
$$
\begin{matrix}
L_{x,y}^{-1}=L_{y,x},  \, &
R_{x,y}^{-1} = R_{y,x}, \, &
S_{x,z}^{-1} = S_{z,x}   \\
(xy(yxu))=u, \, &
((uxy)yx)=u, \, &
(x(zux)z)=u .
\end{matrix}
$$
\end{lemma}

\begin{proof}
Assume given a left half-torsor. 
Then 
$(xyz)= (xe ( eyz))=(xe((eye)ez))$, by 2. and 3.
Conversely,
given a binary inverse loop, 
$(xyz) = L_x L_y^{-1} (z)$, so 
$L_{x,y} = L_x L_y^{-1}$,
whence $L_{x,x} = \id$, $L_{x,y} L_{y,z}= L_{x,z}$ and
$(xey)=e$ iff $L_x (y) = e$ iff $y = L_x^{-1} e = L_e L_x^{-1} e = (exe)$,
so $y=(exe)$ is the inverse of $x$ at $e$.

In a left half-torsor, $(L_{x,y})^{-1} (L_x L_y^{-1})^{-1} =
L_y L_x^{-1} = L_{y,x}$ and $R_{x,y}(z)=(zyx)=z (y^{-1}x) = R_{y^{-1}x,e}(x)$, whence,
using the automorphic inverse property,
$R_{x,y}^{-1} = 
R_{y^{-1}x,e}^{-1} =
R_{(y^{-1} x)^{-1},e} =
R_{x^{-1}y,e} =
R_{y,x}$.
Similarly for a right half-torsor.
\end{proof}

The last part of the proof shows that  apparently ``dual'' identities may have very different proofs.
The remarkable properties of inverse loops are due to the fact that left and right inverse 
coincide, given by the same inversion map $S_{y,y}$.


\begin{lemma}
A ternary product is a torsor if, and only if, it is both a left and a right half-torsor.
\end{lemma}

\begin{proof}
Clearly, every torsor is both a left and a right half-torsor.
Conversely, if $(xyz)$ is both a left and and right half-torsor, then, for every fixed element $e$,
we have $(xyz) = x(y^{-1} z) = (xy^{-1}) z$, by the preceding lemma.
Inversion at $e$ is of order two, so it
is a bijection, so we may replace $y$ by
$y^{-1}$ to see that the product at $e$ is associative, and since it admits inverses, it is a group
law, and $(xyz) = xy^{-1}z$ is the corresponding torsor product.
\end{proof}

\subsection{Autotopies, triality, and binary Moufang loops.}

\begin{Definition}
An {\em autotopy} of an $n$-ary product
$\langle x_1 \ldots x_n \rangle$  is an $n+1$-tuple 
$(f_1,\ldots,f_n;f_{0})$ of bijections $f_i : M \to M$, 
 such that always
\begin{equation}
f_{0} ( \langle x_1 \ldots x_n \rangle ) = \langle f_1(x_1) \ldots f_n(x_n) \rangle .
\end{equation}
The autotopies form a group, the {\em autotopism group} of the $n$-ary product. 
(Cf.\ \cite{BeKi1}, where such property is rather called ``structurality''.) 
The {\em automorphism group} is the ``diagonal subgroup'', given by
$f_i = f_j$ for all $i,j$.)
\end{Definition}

\begin{proposition}\label{prop:triality}
Let $(M,\cdot,e)$ be a binary inverse loop, and 
$j:M \to M$, $x\mapsto j(x)=x^{-1}$ its inversion map. 
Then $f=(f_1,f_2;f_0)$ is an autotopy iff any of the following is an autotopy:
$$
(jf_0 j, f_1; jf_2j), \,
(f_2, jf_0 j; jf_1 j), \,
(jf_2j,jf_1j;jf_0j), \,
(jf_1j,f_0;f_1), \,
(f_0, jf_2j;f_1).
$$
In other words, the autotopism group carries  6 automorphisms forming a group
isomorphic to $\s_3$. The  automorphism $T$ of order $3$ with
$T(f) = (jf_0 j, f_1; jf_2j)$ is called the
{\em triality automorphism} of the autotopism group.
\end{proposition}

\begin{proof}
(Cf.\ \cite{CS}, Chapter 7.)
The ternary relation $R \subset M^3$ given by
$(x,y,z) \in R$ iff
 $xy=z$ is invariant under six maps that are combination of permutation of factors and
 applying $j$ on certain factors: 
 see (\ref{eqn:hexad1}), rewritten
 $(x,y,z) \in R$ iff $(jz),x,j(y)) \in R$, iff
 $
 (j(x),z,y) \in R$, etc. 
 Conjugation by these six maps defines six automorphisms of the autotopy group.
 \end{proof}

So far we cannot guarantee the existence of non-trivial autotopies.
The following axiom is going to change this:

\begin{Definition}
A binary loop $M$ is called a {\em Moufang loop} (``Moup'', cf.\ \cite{CS}) if, for all $a \in M$, the triple
$(L_a,R_a; B_a)$, where $B_a = L_a \circ R_a $,
 is an autotopy, that is, iff the following
{\em Moufang identity} holds:
$$
(ax)(ya) = a(( xy )a)
$$
\end{Definition}

Every Moufang loop is in particular alternative, flexible, and an inverse loop. 
The defining identity is also equivalent to each of the following:
\begin{enumerate} 
 \item[(M1)]
 $L_z \circ L_x \circ L_z = L_{(zx)z}$, i.e.,
 $z(x(zy)) = ((zx)z)y$,
 \item[(M2)]
 $R_z \circ R_x \circ R_z = R_{z(xz)}$, i.e. $ ((yz)x)z = y(z(xz))$
 \end{enumerate}
Since $j L_a j (x)= (ax^{-1})^{-1} = x a^{-1} = R_{a^{-1}}(x)$, the ``hexad'' of autotopies
obtained from $(L_a,R_a;B_a)$ via Prop.\ \ref{prop:triality} is (cf.\ 
\cite{CS}, p.87):   the following are autotopies
\begin{equation} \label{eqn:hexad2}
\begin{matrix}
(L_a,R_a;B_a), &  \,
(B_a,L_{a^{-1}};L_a), &  \,
(R_a,B_{a^{-1}} ; R_{a^{-1}}, \\
(L_{a^{-1}},R_{a^{-1}}; B_{a^{-1}}), &  \,
(B_{a^{-1}},L_a; L_{a^{-1}}),& \,
(R_{a^{-1}},B_a;R_a).
\end{matrix}
\end{equation}

\subsection{Left and Right ternary Moufang loops.}
Whereas a single ternary concept (torsor) corresponds to groups, there are two different
ternary versions (left and right) corresponding to binary Moufang loops.
Thus fixing the ternary concept involves a {\em choice}, between ``left'' and ``right''.

\begin{Definition}\label{def:TernaryMoufangLoop}
A {\em left  ternary  Moufang loop} is a left  half-torsor satisfying the identity
$$
((xba)(yba)(zab))= ((xyz)ab) ,
$$
meaning that the quadruple
$(R_{a,b},R_{b,a},R_{b,a};R_{a,b})$
is an autotopy. 
A {\em right ternary Moufang loop} is a right half-torsor satisfying the identity
$$
((abx)(bay)(baz))= (ab(xyz)),
$$
meaning that the quadruple
$(L_{a,b},L_{b,a},L_{b,a};L_{a,b})$
is an autotopy.
\end{Definition}

\begin{theorem}[Ternary and binary Moufang loops]\label{th:ternaryMouf}
$ $
\begin{enumerate}
\item
In a left (resp., right) ternary Moufang loop, every homotope at $e$ is a binary
Moufang loop.
\item
Every left ternary  Moufang loop  can be recovered from a homotope 
$xz= (xez)$ via  $(xyz) = x(y^{-1} z)$ in the sense of Lemma \ref{la:Chasles},
and every right ternary Moufang loop via $(xyz) = (xy^{-1})z$.
\item
In a left ternary Moufang loop, 
$(B_{a,b},B_{a,b},L_{a,b};L_{a,b})$ is an autotopy;

in a right ternary Moufang loop,
$(R_{a,b},B_{a,b},B_{a,b} ; R_{a,b})$ is an autotopy.
\item
Both left and right ternary Moufang loops are {\em reflection spaces} in the sense
of \cite{Lo67} (cf.\ Subsection  \ref{sec:reflectionspace}), when equipped with the binary product
$$
\mu(x,y):= S_{x,x}(y)=(xyx) = x(y^{-1} x) = (xy^{-1})x.
$$
In a left ternary Moufang loop, the set of right multiplications is stable under the
composition $(f,g) \mapsto fg^{-1}f$, and a right ternary Moufang, the set left multiplications
is stable under this composition.
\end{enumerate}
\end{theorem}

\begin{proof}
1. 
Assume $(xyz)$ is a left ternary Moufang loop, and let $e \in M$.
We show that the homotope at $e$ is a Moufang loop.
If $(f_1,f_2,f_3;f_4)$ is a ternary autotopy,
so $f_4((xyz))=(f_1(x)f_2(yf_3(z)) = f_1(x) ( f_2(y)^{-1} f_3(z))$,
we let $y=e$: for all $x,z \in M$,
$$
f_4 (xz) = f_1(x) \cdot f_2(e)^{-1} f_3(z) ,
$$
so we get a binary autotopy
$(f_1, L_{f_2(e)}^{-1} \circ f_3; f_4)$.
In particular, the ternary autotopy 
$(R_{a,e},R_{a,e},R_{e,a};R_{e,a})$ gives the binary autotopy $(R_a,B_{a^{-1}};R_{a^{-1}})$, since
$$
a^{-1}(xz) = R_{e,a} (xez)= 
( R_{a,e}x,R_{a,e}y,R_{e,a}z) =
xa \cdot ( a^{-1} \cdot za^{-1} ) = R_a (x) B_{a^{-1}}(z),
$$
which is one of the six versions of the Moufang identity in the ``hexad''
(\ref{eqn:hexad2}), so the binary loop is a Moufang loop.
By the same computation, for $b=e=y$ in a right ternary Moufang loop,
the defining identity implies
$B_a(x) \cdot L_{a^{-1}}(x) = L_a(xz)$, which is again eqivalent to the Moufang
identity, by the ``hexad''.

2. Conversely, assume $M$ is a Mouf with unit $e$, and define the corresponding
left half-torsor by
$(xyz):= x(y^{-1} z)$.
It has been shown in \cite{BeKi14}, Lemma 3.4., that every homotope is again a binary
Moufang loop, thus we may compute in the homotope at $b$; that is, we may assume 
that $b=e$.
Using the ``hexad'' (\ref{eqn:hexad2}) twice, along with $(uv)^{-1}=v^{-1}u^{-1}$:
\begin{align*}
((xba)(yba)(zab)) & =
xa \cdot ( (ya)^{-1} \cdot (za^{-1})) 
 =
xa \cdot (L_{a^{-1}} y^{-1} \cdot R_{a^{-1}}z)
\\
&  =
R_a(x) \cdot B_{a^{-1}}(y^{-1}z) = R_{a^{-1}}(x(y^{-1}z))
= R_{b,a}(xyz).
\end{align*}

3.
In a left ternary Moufang loop, $(xyz)=x(y^{-1}z)$; as above we may assume
$b=e=y$ and use the hexad (\ref{eqn:hexad2}),
\begin{align*}
(B_{a,b}x,B_{a,b}y,L_{a,b}z) & =
axa \cdot (a^{-2} \cdot a z) =
B_a(x) \cdot a^{-1} z \\
& = L_a (x z) = L_{a,b} (xyz) ,
\end{align*}
so $(B_{a,b},B_{a,b},L_{a,b};L_{a,b})$ is an autotopy.
Similarly in a right ternary Moufang loop.

4. Loos (\cite{Lo67}, Satz 1.3) shows that every Mouf with $\sigma_x(y) = xy^{-1}x$ is a reflection space (and, moreover, that left and right multiplications are automorphisms of the reflection space).

The identities (M1) and (M2) imply that, for a binary Moufang loop, both the sets of
left and of right multiplications are stable under $(f,g) \mapsto fg^{-1}f$.
For a left ternary Moufang loop, $R_{a,b} = R_{b^{-1} a}$, so set of right multiplications is the
same as for the binary homotope, whence the last claim. 
\end{proof}

\begin{remark}\label{rk:BeKi}
The ternary Moufang loops defined in  \cite{BeKi14}, Def.\ 3.5, are precisely the
{\em right} ternary Moufang loops considered here. In loc.cit., the defining identities are

(MT0) idempotency $(xxy)=y=(yxx)$

(MT1) $R_{x,y} R_{x,v}= R_{(xyx),v}$ : $((uvx)yx) = (uv(xyx))$ 

(MT2) $L_{x,y}L_{x,y} = L_{(xyx),y}$ : $ (xy(xyz)) = ( (xyx)yz)$

\nin
In loc.\ cit., Lemma 3.4., it is shown that this definition again leads to
$(xyz) = (xy^{-1})z$. Although (MT1), (MT2) involve less variables than the identities
used here, their  conceptual meaning is less clear, since the ``duality'' between
right and left ternary Moufang loops remains hidden.
\end{remark}




\section{The Moufang double of a group}\label{app:Moufang-double}

The disjoint union of two copies of a group $G$ carries a
rather natural structure of Moufang loop.
This result is du to Orin Chein (\cite{Ch}). (See also
\url{https://en.wikipedia.org/wiki/Moufang_loop#Examples}.)
We  generalize this result, with the aim to show that this construction is
essentially equivalent to the ACD-construction, broken up into its pieces
(KD0)-(KD4) (cf.\ Section \ref{sec:ACD}). 
We also give a ternary version of this construction
(``ABCD-construction''). 
Recall that a {\em central involution}  $\sharp$ of a group $G$  is an involution
such that $gg^\sharp = g^\sharp g$ is always in the center of the group, as is the
case for group inversion $g^\sharp = g^{-1}$. 
Another example is the adjugate involution of $\Gl(2,\K)$.

\begin{Definition} 
The {\em double} of a set $M$ is the union of two disjoint copies of $M$:
$$
D(M):= M \times \Z/2\Z = M_0 \sqcup M_1 = (M \times \{ 0 \} )\sqcup (M \times \{ 1 \}) .
$$
For $x \in M$ we let
$x_i := x \times \{ i \} \in M_i$.  Elements of $D(M)$ are called {\em even} if they belong to $M_0$
and {\em  odd} if they
belong to $M_1$.
Assume $\cdot$ is a  binary product on $M$.
We want to extend it to a {\em graded  binary product} $\bullet$ on $D(M)$, i.e., such that
 always 
 $$
 x_i \bullet y_j \in M_{i+j \mod 2} .
 $$
 In other terms, we have to describe $4$ products on $M$,
 its {\em components},  corresponding to the $4$
products  $M_i \times M_j \to M_{i+j \mod 2}$:
for $x,y \in M$, we have to define the values
$$
x_0 \bullet y_0 \in M_0, \quad
x_0 \bullet y_1 \in M_1, \quad
x_1 \bullet y_2 \in M_1, \quad
x_1 \bullet y_1 \in M_0.
$$
The {\em trivial graded extension} is $x_i \bullet y_j := (xy)_{i+j \mod 2}$, i.e., all $4$ products
are equal to the original product in $M$. 
\end{Definition}





\begin{theorem}\label{th:Moufang-double2} 
 Let $G$ be a group, together with a central   involution $\sharp$, and fix
two elements $\eps$ and $\mu$ of the the center $Z(G)$, fixed under $\sharp$, and
with $\eps^2 = e$.
Then we define a graded  product $\bullet$ on $D(G)$, and a new map $\sharp : D(M) \to D(M)$, by:
\ssk
\begin{center}
\begin{tabular}{|*{10}{c|}}
\hline
 parity of $x$ & parity of $y$ &  $x \bullet y$
 \\
 \hline
 $0$ & $0$ &  $(xy)_0$
 \\
 $0$ & $1$ & $(yx)_1$
 \\
$ 1$ & $0$ & $(x y^\sharp)_1$
 \\
$ 1$ & $1$ & $\mu \eps (y^\sharp x)_0$
 \\
 \hline
 \end{tabular}
 $\qquad$
 \begin{tabular}{|*{10}{c|}}
\hline
 parity of $x$ &  $x^\sharp$
 \\
 \hline
 0 & $(x^\sharp)_0$
 \\
 1 & $ \eps x_1$
 \\
 \hline
 \end{tabular}
 \end{center} 
 \ssk
Then the following holds:
\begin{enumerate}
\item
If $G$ is commutative and $\sharp$ trivial, then $D(G)$ is a commutative group.
\item
If $G$ is commutative but $\sharp$ non-trivial,
then  $D(G)$ is a non-commutative group.
\item
If $G$ is not commutative, 
then $(D(G),\bullet)$ is a non-associative
{\em Moufang loop}.
\end{enumerate}
In any case, the new
 $\sharp:D(G) \to D(G)$ is again a central  involution, and it is related to the inversion map
 via
 $x^{-1} = (x^\sharp x)^{-1} x^\sharp$, where $ (x^\sharp x)^{-1} $ is the
(group) inverse in the center of $D(G)$. 
\end{theorem}

\begin{proof}
(1) and (2):
If $G$ is commutative, the inversion map $g \mapsto g^{-1}$ is an automorphism of order
$2$, so the semi-direct product $G \rtimes C_2$ is well-defined; the elements
$\eps,\mu$ introduce a cocycle term, so we have a central extension of $G$. 
Conversely,
assume $D(G)$ is associative and let $f,g,h \in G$.
The equality $(f_0g_0) h_1  = f_0 ( g_0 h_1)$ gives
$( h (fg))_1 = ((hg)f)_1 , \mbox{ whence }  fg = gf \mbox{ in } G$,
whence $G$ is commutative.

To prove (3)  we show that the Moufang identity (M1) holds:  let $x,y,z \in G$, and check (M1) for the
$8$ possible cases taking $x_i,y_j,z_k$ with  $i,j,k=0$ or $1$. 
 We compute the products directly 
by applying definitions.
For better readability, we omit the index
$\ell = i+j+k \mod(2)$ in the last two columns. In Lines 3 -- 6 we use that $zz^\sharp$ belongs to the center.
We find that the last two columns always coincide:

\begin{center}
\begin{tabular}{|*{10}{c|}}
\hline
parity of $z_k$  & parity of $x_i$ & parity of $y_j$ &  $z(x(zy))$ & $((zx)z)y$
\\
\hline
0 & 0 & 0  & $zxzy$ & $zxzy$
\\
0 & 0 & 1  & $yzxz$ & $yzxz$
\\
0 & 1 & 0  & $zz^\sharp \, xy^{\sharp}$ & $zz^\sharp \, xy^{\sharp}$
\\
0 & 1 & 1 & $\mu \eps \,  zz^\sharp \, y^\sharp  x$ & $\mu \eps \, zz^\sharp \, y^{\sharp }x$ 
\\
1 & 0 & 0 &  $\mu \eps \, zz^\sharp \, x^{\sharp}y$ & $\mu \eps \, zz^\sharp \, x^{\sharp} y$
\\
1 & 0 & 1  & $\mu \eps \, zz^\sharp \, yx^{\sharp}$ & $\mu \eps \, zz^\sharp \, y x^{\sharp}$ 
\\
1 & 1 & 0  & $\mu \eps \, zx^{\sharp} z y^{\sharp}$ & $\mu \eps \, zx^{\sharp} z y^{\sharp}$
\\
1 & 1 & 1 & $\mu^2 \,  y^{\sharp}zx^{\sharp} z$ & $\mu^2 \,  y^{\sharp}zx^{\sharp} z$
\\
\hline
\end{tabular}
\end{center}
Clearly, $e$ is a unit for the product, and $D(G)$ is a quasigroup
(since left and right multiplications and inversion are bijections in $G$).
Explicitly, the inverse of $x_0$ is its inverse $(x_0)^{-1}$ in $G_0$, and the inverse of
$x_1$ is $(\eps (xx^\sharp)^{-1} x )_1$. Thus $D(G)$ is a Moufang loop.
The new map $\sharp:D(G) \to D(G)$ is of order $2$, and it is an anti-automorphism, as is seen
by distinguishing $4$ cases to compute
$(xy)^\sharp$ and $y^\sharp x^\sharp$.
\end{proof}

\begin{Definition}\label{def:dih(G)}
Let $G$ be a group and $Z  \in G$ be a central element of order $2$ (if $G$ is a matrix group containing
$- \id$, we will take $Z = - \id$).
We define 
\begin{enumerate}
\item
the {\em dihedral loop of $G$} to be
$Dih(G):= D(G)$, by taking on $G$ the involution
 $g^\sharp = g^{-1}$ and $\mu = e$ and $\eps = e $ (neutral element),
\item
the {\em dicyclic loop of $(G,Z)$} to be
$Dic(G,Z):= D(G)$ by taking on $G$ the involution $g^\sharp = g^{-1}$ and $\mu= e$ and $\eps = Z$.
\end{enumerate}
Note that, when $G$ is commutative, then
$Dih(G)$ is indeed the
\href{https://en.wikipedia.org/wiki/Generalized_dihedral_group}{generalized dihedral group of $G$},
and  $Dic(G,Z)$  the 
\href{https://en.wikipedia.org/wiki/Dicyclic_group#Generalizations}{generalized dicyclic group of $(G,Z)$}
(Remark \ref{rk:dicyclic}).
\end{Definition} 

\begin{remark}
The dicyclic extension $Dic(\GL(n,\K),-I)$ should be particularly interesting for the purposes
of linear algebra and group theory. 
\end{remark}

\begin{theorem}\label{th:Moufang-ternary}
Let assumptions and notations be as in the preceding theorem. 
Then the ``new''  ternary product 
$\langle abc \rangle := a \bullet (b^\sharp \bullet c)$ on $D(G)$ is given by the following table, in terms
of the ``old''  product $\langle xyz\rangle = x y^\sharp  z$  living on $G$: 
\ssk
\begin{center}
\begin{tabular}{|*{10}{c|}}
\hline
$i$  &  $j$  &  $k$ &  $\langle x_i y_j z_k \rangle = x_i \bullet (y_j^\sharp \bullet z_k)$
\\
\hline 
$0$ & $0$ & $0$  & $\langle xyz\rangle_0$
\\
\hline
$0$ & $0$ & $1$  & $\langle zyx \rangle_1$
\\
$0$ & $1$ & $0$  & $ \eps \cdot  \langle yzx\rangle_1 $
\\
$1$ & $0$ & $0$  & $\langle xzy \rangle_1$
\\
\hline
$1$ & $1$ & $0$  & $ \mu \cdot \langle zyx\rangle_0 $
\\
$1$ & $0$ & $1$  & $ \eps \mu \cdot  \langle yzx\rangle_0$
\\
$0$ & $1$ & $1$  &  $ \mu \cdot \langle xzy\rangle_0$
\\
\hline
$1$ & $1$ & $1$  & $\mu \cdot  \langle xyz\rangle_0 $
\\
\hline
\end{tabular}
\end{center}
In particular, when $\eps = 1 =\mu$, then the ternary Moufang structure
on $D(G)$ is given by permutations, depending on the tri-degree of
$(x,y,z)$, of the order in the original torsor law $(xyz)$.
\end{theorem}

\begin{proof}
Direct computation.
\end{proof}

\begin{remark}
The warning from Remark \ref{rk:warning} applies here, too:
our definition of the Moufang double follows McCrimmon's conventions.
By following Faulkner's convention, we get another, ``dual'', definition, $\bullet'$
which is adapted to the right ternary product
$(a \bullet' b^\sharp) \bullet' c$.
\msk
\begin{center}
\begin{tabular}{|*{10}{c|}}
\hline
 parity of $x$ & parity of $y$ &  $x \bullet' y$
 \\
 \hline
 0 & 0 &  $(xy)_0$
 \\
 0 & 1 & $(x^\sharp  y)_1$ 
 \\
 1 & 0 & $(yx)_1$
 \\
 1 & 1 & $\mu \eps (yx^\sharp)_0$
 \\
 \hline
 \end{tabular}
 $\qquad$
 \begin{tabular}{|*{10}{c|}}
\hline
 parity of $x$ &  $x^\sharp$
 \\
 \hline
 0 & $(x^\sharp)_0$
 \\
 1 & $ \eps x_1$
 \\
 \hline
 \end{tabular}
 \end{center} 
 \ssk
 Summing up, one should distinguish a ``left'' from a ``right Moufang double''.
\end{remark}


\begin{example}
Certain sequences $G, D(G),D(D(G)),D(D(D(G)))$ are analogs of the sequence $\mathbb{R,C,H,O}$.
For instance:
\begin{enumerate}
\item
$G = C_2 = \{ \pm 1 \}$ and $\eps = - 1$, and $\mu=1$.
Then $D(G) = C_4$, $D(D(G)) =  {\mathbf Q}$ (quaternion group of order $8$), 
$D^3(G)$ is $\mathbf O$, the octonion loop of order $16$.
\item
$G= C_2$, $\eps = \mu = -1$, then
$D(G) = C_2 \times C_2$,
$D(D(G))= D_4$ (dihedral group), and
$D^3(G)$ is the  split octonion loop of order $16$. 
\item
$G = C_3$ with $g^\sharp = g^{-1}$ and
$\eps = \mu = 1$, then $D(G) = D_3$, $D(D(G)) $ is smallest possible non-associative
 Moup, of cardinal $12$.
 \item
 $G = C_6$ with $g^\sharp = g^{-1}$ and $\mu=1$, $\eps = -1$ (element of order $2$ in $C_6$),
 $D(G) = Dic_3$, $D(D(G))$ a Moufang loop of cardinal $24$ 
 (cf. Example \ref{ex:dicloop}).
\end{enumerate}
\end{example}

\end{document}